\documentclass[11pt]{amsart}
\usepackage[margin=1in]{geometry}
\usepackage{amsmath,amsthm,amssymb}
\usepackage{thmtools,thm-restate}
\usepackage{lmodern}
\usepackage{xcolor}
\usepackage{tikz,placeins,needspace}
\usetikzlibrary{arrows.meta,calc}
\numberwithin{figure}{section}

\usepackage{microtype}
\definecolor{linkteal}{HTML}{007C83}
\usepackage[hypertexnames=false,colorlinks=true,allcolors=linkteal]{hyperref}
\usepackage{xypic}
\usepackage{bbm}
\usepackage{mathtools}
\usepackage{enumitem}
\newtheorem{thm}{Theorem}[section]
\newtheorem{lemma}[thm]{Lemma}
\newtheorem{prop}[thm]{Proposition}
\newtheorem{cor}[thm]{Corollary}
\theoremstyle{definition}
\newtheorem{rem}[thm]{Remark}
\newtheorem{defn}[thm]{Definition}
\newtheorem{ex}[thm]{Example}

\newtheorem*{ques*}{Question}
\numberwithin{equation}{section}
\newcommand{\bbA}{\mathbb{A}}
\newcommand{\bbB}{\mathbb{B}}

\newcommand{\bbH}{\mathbb{H}}

\newcommand{\bbP}{\mathbb{P}}
\newcommand{\bbQ}{\mathbb{Q}}
\newcommand{\bbR}{\mathbb{R}}

\newcommand{\bfG}{\mathbf{G}}
\newcommand{\bfH}{\mathbf{H}}
\newcommand{\bfI}{\mathbf{I}}
\newcommand{\bfJ}{\mathbf{J}}

\newcommand{\bfL}{\mathbf{L}}

\newcommand{\bfN}{\mathbf{N}}

\newcommand{\bfV}{\mathbf{V}}

\newcommand{\bfZ}{\mathbf{Z}}
\newcommand{\bfa}{\mathbf{a}}
\newcommand{\bfGL}{\mathbf{GL}}

\newcommand{\calA}{\mathcal{A}}
\newcommand{\calB}{\mathcal{B}}
\newcommand{\calC}{\mathcal{C}}
\newcommand{\calD}{\mathcal{D}}
\newcommand{\calE}{\mathcal{E}}
\newcommand{\calF}{\mathcal{F}}
\newcommand{\calG}{\mathcal{G}}

\newcommand{\calI}{\mathcal{I}}
\newcommand{\calJ}{\mathcal{J}}
\newcommand{\calK}{\mathcal{K}}
\newcommand{\calL}{\mathcal{L}}

\newcommand{\calN}{\mathcal{N}}
\newcommand{\calO}{\mathcal{O}}
\newcommand{\calP}{\mathcal{P}}
\newcommand{\calQ}{\mathcal{Q}}
\newcommand{\calR}{\mathcal{R}}
\newcommand{\calS}{\mathcal{S}}
\newcommand{\calT}{\mathcal{T}}
\newcommand{\calU}{\mathcal{U}}
\newcommand{\calV}{\mathcal{V}}

\newcommand{\calZ}{\mathcal{Z}}

\newcommand{\frakD}{\mathfrak{D}}

\newcommand{\fraka}{\mathfrak{a}}

\newcommand{\frakh}{\mathfrak{h}}

\newcommand{\frakl}{\mathfrak{l}}

\newcommand{\frakso}{\mathfrak{so}}

\newcommand{\sfm}{\mathsf{m}}
\newcommand{\sfp}{\mathsf{p}}
\newcommand{\sfq}{\mathsf{q}}
\newcommand{\sfR}{\mathsf{R}}
\newcommand{\sfT}{\mathsf{T}}
\newcommand{\sfX}{\mathsf{X}}
\newcommand{\sfY}{\mathsf{Y}}

\DeclareMathOperator{\SO}{SO}

\DeclareMathOperator{\Prob}{Prob}

\DeclareMathOperator{\BR}{BR}
\DeclareMathOperator{\BMS}{BMS}

\DeclareMathOperator{\Stab}{Stab}
\DeclareMathOperator{\supp}{supp}
\DeclareMathOperator{\Hull}{Hull}
\DeclareMathOperator{\core}{core}
\DeclareMathOperator{\vol}{vol}
\DeclareMathOperator{\Span}{span}
\DeclareMathOperator{\interior}{int}
\DeclareMathOperator{\RF}{RF}
\DeclareMathOperator{\Bl}{Bl}
\DeclareMathOperator{\pr}{pr}
\DeclareMathOperator{\Gr}{Gr}
\DeclareMathOperator{\Lie}{Lie}
\DeclareMathOperator{\Comm}{Comm}

\DeclareMathOperator{\Id}{Id}

\DeclareMathOperator{\Ad}{Ad}
\DeclareMathOperator{\Tr}{Tr}
\DeclareMathOperator{\tr}{tr}

\DeclareMathOperator{\Hom}{Hom}

\DeclareMathOperator{\AR}{AlgRep}

\DeclareMathOperator{\Res}{Res}

\newcommand{\trieq}{\trianglelefteq}

\newcommand{\ssm}{\smallsetminus}

\title[Geodesic submanifolds of infinite-volume hyperbolic manifolds]{Arithmeticity, superrigidity, and totally geodesic submanifolds of infinite-volume hyperbolic manifolds}
\author{Nicholas Miller}
\address{Department of Mathematics and Statistics, Villanova University, Villanova, PA, 19085}
\email{nicholas.miller@villanova.edu}
\date{}

\begin{document}

\maketitle

\begin{abstract}
Given an infinite-volume convex cocompact hyperbolic manifold with Fuchsian ends and infinitely many maximal properly immersed totally geodesic submanifolds of dimension at least two intersecting its convex core, we show that it (infinitely) covers an arithmetic hyperbolic manifold.
This answers the general version of a question of Oh.
As a byproduct of our techniques, we also give a Ratner-type classification of $\SO_0(m,1)$-invariant Radon measures whose support intersects the frame bundle of the interior of the convex core.
\end{abstract}

\section{Introduction}

Recent work of Bader, Fisher, Stover, and the author has shown that the presence of infinitely many properly immersed maximal totally geodesic submanifolds of real dimension at least two in a finite-volume real or complex hyperbolic manifold forces arithmeticity \cite{BFMS,BFMS2}.
For closed hyperbolic $3$-manifolds, the arithmeticity conclusion was also obtained independently by Margulis and Mohammadi \cite[Thm~1.1]{MM22} and for complex hyperbolic manifolds containing infinitely many maximal complex totally geodesic submanifolds by Baldi and Ullmo \cite[Thm~1.2.1]{BaldiUllmo}.
Whereas this gives an optimal answer in the class of locally symmetric spaces, it is natural to ask how far a similar theorem extends.
There are two particularly intriguing directions in which to try to generalize such a result.

The first is to relax the hyperbolicity condition to merely negative curvature.
In this direction, Filip, Fisher, and Lowe \cite{FFL} have recently shown that a closed Riemannian manifold of dimension at least three with negative sectional curvature, a real-analytic metric, and infinitely many properly immersed totally geodesic hypersurfaces must in fact be a hyperbolic manifold, up to rescaling.
In particular, the confluence of \cite{BFMS} and \cite{FFL} implies that this manifold must therefore be arithmetic hyperbolic of simplest type.
Conjecturally, the condition of real analyticity should be able to be removed and similarly codimension one should be able to be replaced with maximality in dimension at least two, where one concludes that the manifold is an arithmetic locally symmetric space, up to rescaling (see \cite[Conj 6.9]{FilipSurvey}).

The second direction is to try to relax the restriction of finite volume.
In this direction, Oh \cite[Rem 7.5(3)]{OhSurvey} has asked the following question.

\begin{ques*}[Oh]
Given a rigid acylindrical hyperbolic $3$-manifold $M$ with infinitely many properly immersed (unbounded) geodesic planes intersecting its convex core, does $M$ necessarily cover an arithmetic hyperbolic $3$-manifold?
\end{ques*}

The purpose of this paper is to answer this question in the affirmative and prove its natural generalization to maximal geodesic submanifolds in higher dimensions.
Here a properly immersed totally geodesic submanifold is called \emph{maximal} if its image is not contained in a similar proper submanifold of strictly larger dimension.

\begin{restatable}{thm}{arithmeticity}\label{thm:arithmetic}
Let $n\ge3$.
Suppose that $Y_\infty$ is an infinite-volume convex cocompact hyperbolic $n$-manifold with Fuchsian ends and infinitely many distinct maximal properly immersed totally geodesic submanifolds of dimension at least two meeting the convex core.
Then $Y_\infty$ covers an arithmetic hyperbolic manifold $Y$.
\end{restatable}

Note that a convex cocompact hyperbolic $3$-manifold with Fuchsian ends is precisely rigid acylindrical, hence this is the natural general form of this question.
As is well known, there are infinitely many commensurability classes of manifolds satisfying the hypotheses of Theorem~\ref{thm:arithmetic} (see Example~\ref{ex:arithmeticpieces}).
The hypothesis of Fuchsian ends also cannot be omitted, see Theorem~\ref{thm:counterexamples} below.
We expect, however, that Theorem~\ref{thm:arithmetic} extends to geometrically finite manifolds with Fuchsian ends, although cusps in the boundary introduce degenerations that the recurrence and nonconcentration arguments of Sections~\ref{sec:higherdimensionabsolute} and \ref{sec:surfaceabsolute} do not address (see \cite[Thm~1.11]{ZhangElementary}).
There is also a prior obstacle to carrying out our strategy in this generality, since even the orbit closures of geodesic planes are not yet understood for general geometrically finite manifolds with Fuchsian ends; see for instance Oh's survey \cite{OhICM26} for this question in the setting of Apollonian circle packings.
Thus the orbit closure results recalled in Section~\ref{sec:topologicalrigidity} would first need to be extended to this setting before we could carry out our strategy.

We give a detailed sketch of the theorem at the end of the introduction but for now we mention some ingredients so that we can state the main results.
Throughout, we let $\Gamma=\pi_1(Y_\infty)$.
Our proof of Theorem~\ref{thm:arithmetic} rests on the following three steps.
\begin{enumerate}
\item Proving a measure rigidity statement for the limit measure of the normalized volume measures on the maximal geodesic submanifolds,
\item Lifting such a measure to an appropriate suspension space, and
\item Proving a superrigidity statement for certain natural representations of $\Gamma$.
\end{enumerate}
For the reader familiar with \cite{BFMS}, this is broadly the same outline as the proof therein.
However, the first and third steps require several new arguments in this setting, which we discuss at the end of this section.
Indeed, in \cite{BFMS} item (1) was an elaboration of work of Ratner \cite{Ratner90,Ratner2} and Mozes--Shah \cite{MozesShah} and item (3) used a new take on machinery developed by Bader and Furman \cite{BF2}.

In contrast, in this setting our measures are generally infinite so we cannot apply Ratner's theorems directly to them.
Consequently, one of the main ingredients is a measure rigidity result that replaces this.
Throughout this paper, weak-$*$ convergence of Radon measures is understood to be against compactly supported continuous test functions.
After passing to a subsequence, we fix the common dimension $m$ of the maximal geodesic submanifolds and write $G=\SO_0(n,1)$ and $W=W_m=\SO_0(m,1)$.
The space $\calC^*\subset G/W$ consists of the normally framed geodesic $m$-planes meeting the interior of the lifted convex core and $\lambda$ denotes the restriction to $\calC^*$ of Haar measure on $G/W$.
See Section~\ref{sec:background} for the detailed definitions.

\begin{restatable}{thm}{measurerigidity}\label{thm:measurerigid}
Let $2\le m<n$.
Suppose that $Y_\infty$ is an infinite-volume convex cocompact hyperbolic $n$-manifold with Fuchsian ends.
Let $\{\Sigma_i\}$ be an infinite sequence of pairwise distinct maximal properly immersed totally geodesic $m$-submanifolds meeting the convex core of $Y_\infty$.
Let $\mu_i$ be their adapted-frame measures normalized by the volumes of their convex cores and let $\mu_\infty$ be a weak-$*$ limit.
Write $\eta_\infty$ for the restriction to $\calC^*$ of the transverse current associated to $\mu_\infty$.
Then there exists a positive constant $c$ such that $\lambda=c\eta_\infty$.
\end{restatable}

Here, an ambient oriented orthonormal frame is adapted to $\Sigma$ if its last $m$ vectors are tangent to $\Sigma$.
We obtain its adapted-frame measure by lifting hyperbolic volume on $\Sigma$ to the bundle of adapted frames, giving each fiber Haar probability measure, and pushing the resulting measure to $\Gamma\backslash G$ (see Section~\ref{sec:limitingmeasure}).

The conclusion of Theorem~\ref{thm:measurerigid} can also be stated directly on the frame bundle.
Let $\pi_0:\Gamma\backslash G\to Y_\infty=\Gamma\backslash\bbH^n$ be the natural projection and put
$$\calF^*=\{x\in\Gamma\backslash G\mid \pi_0(xW)\cap\interior(\core(Y_\infty))\ne\emptyset\},$$
so that $\calF^*$ consists of the frames whose associated geodesic $m$-plane meets the interior of the convex core.

\begin{cor}\label{cor:limitinghaar}
Under the hypotheses of Theorem~\ref{thm:measurerigid}, there exists $c>0$ such that
$$\mu_\infty\vert_{\calF^*}=c\,m_{\mathrm{Haar}}\vert_{\calF^*},$$
where $m_{\mathrm{Haar}}$ is Haar measure on $\Gamma\backslash G$.
\end{cor}

As described in Section~\ref{sec:limitingmeasure}, we recover the frame measure from its transverse current by adjoining Haar measure along the $W$-fibers.
The same construction applied to $\lambda$ gives $m_{\mathrm{Haar}}\vert_{\calF^*}$, so the result follows from Theorem~\ref{thm:measurerigid}.

The same boundary and carrier arguments we will use to prove Theorem~\ref{thm:measurerigid} will also lead to a general classification of $W$-invariant measures.

\begin{restatable}{thm}{generalmeasurerigidity}\label{thm:generalmeasurerigidity}
Let $2\le m<n$. 
Suppose that $Y_\infty=\Gamma\backslash\bbH^n$ is an infinite-volume convex cocompact hyperbolic $n$-manifold with Fuchsian ends.
Then every nonzero $W_m$-invariant, $W_m$-ergodic Radon measure on $\calF^*$ has the form
\begin{equation}\label{eqn:Ratnerstyle}
\mu=c\,m_{yH}\vert_{\calF^*},
\end{equation}
where $c>0$, $H$ is a closed subgroup containing $W_m$, $yH$ is closed in $\Gamma\backslash G$, and $m_{yH}$ is Haar measure on the full orbit.
Moreover, there exists $r$ with $m\le r\le n$ such that
$$W_r\le k_0Hk_0^{-1}\le N_G(W_r),$$
for some $k_0\in C_G(W_m)$.
\end{restatable}


There is an extensive history of measure classification for horospherical actions on infinite-volume hyperbolic manifolds.
In the geometrically finite setting, this includes the work of Burger \cite{Bur90}, Roblin \cite{Roblin03}, and Winter \cite{Winter}, with alternative proofs given by Schapira \cite{Sch16} and Mohammadi--Oh \cite[Thm~5.7]{MO16}.
On abelian covers of compact hyperbolic surfaces, Sarig \cite{Sar04} showed that the measures constructed by Babillot and Ledrappier \cite{BL98} exhaust the horocycle-invariant ergodic Radon measures.
Oh and Pan \cite{OP19} established the corresponding classification for maximal horospherical actions on frame bundles of abelian covers of compact hyperbolic $n$-manifolds.
Classifications for horocycle flows and horospherical foliations on more general regular covers and weakly tame hyperbolic surfaces are also obtained in \cite{LS07,Led08,Sar10}.
For broader classes of geometrically infinite quotients, Landesberg--Lindenstrauss \cite{LL22} and Landesberg \cite{Lan21} established normalizer quasi-invariance, while the recent work of Choi and Kim \cite{CK26} gives a classification on the geodesically recurrent locus.
Theorem~\ref{thm:generalmeasurerigidity} complements these horospherical results by giving a homogeneous classification of ergodic Radon measures for the semisimple $W_m$-action on $\calF^*$ for convex cocompact manifolds with Fuchsian ends.

To carry out the second step listed above we follow the suspension construction of \cite{BFMS}, as described in Section~\ref{sec:suspension}.
For the representations arising in the arithmeticity criterion, this gives a $W$-invariant lift of $\mu_\infty$ to a bundle with projective fiber.
Disintegration and Theorem~\ref{thm:measurerigid} then give a $\Gamma$-equivariant, $\lambda$-measurable map to probability measures on that fiber.
This provides the input to the superrigidity theorem below.

For the final step, we recall the definition of compatibility from \cite{BFMS}.
Let $k$ be a local field of characteristic $0$ and let $\bfH$ be a connected adjoint semisimple $k$-algebraic group.
Write $P=MAN<G$ for the minimal parabolic subgroup and let $U=N\cap W$, with the choices made in Section~\ref{sec:background}.
We say that the triple $(G,\bfH,k)$ is compatible if, for every nontrivial $k$-algebraic subgroup $\bfL\le\bfH$ and every continuous homomorphism $\tau:P\to(\bfN_{\bfH}(\bfL)/\bfL)(k)$, the Zariski closures of $\tau(U)$ and $\tau(N)$ agree.

Under this hypothesis, which holds for the target groups of the natural representations arising in the arithmeticity criterion, we then have the following.

\begin{restatable}{thm}{superrigidity}\label{thm:superrigid}
Let $2\le m<n$.
Suppose that $Y_\infty$ is an infinite-volume convex cocompact hyperbolic $n$-manifold with Fuchsian ends and put $W=W_m$ with corresponding $\calC^*$.
Let $k$ be a local field of characteristic $0$ and let $\bfH$ be a connected adjoint semisimple $k$-algebraic group such that $(G,\bfH,k)$ is compatible.
Suppose that $\rho:\Gamma\to\bfH(k)$ has unbounded, Zariski dense image and that $V$ is a faithful irreducible $k$-rational representation of $\bfH$ of dimension at least two.
If there exists a $\Gamma$-equivariant, $\lambda$-measurable map
$$\psi:\calC^*\longrightarrow\Prob(\bbP(V)(k)),$$
where $\Gamma$ acts on $\bbP(V)(k)$ through $\rho$, then there is a continuous homomorphism
$$\widetilde{\rho}:G\longrightarrow\bfH(k),$$
whose restriction to $\Gamma$ agrees with $\rho$.
In particular, $\rho$ is superrigid.
\end{restatable}

For the reader familiar with \cite{BFMS}, though we have stated the theorem similarly to that paper, the proof requires significantly more technology in Bader and Furman's theory of Algebraic Representations of Ergodic Actions (AREA).
Indeed, in the finite-volume case of \cite{BFMS}, Haar measure on $\Gamma\backslash G$ is mixing for any noncompact closed subgroup.
In particular, in that setting one has initial objects in all of the relevant categories and can quickly use that to find an extension of $\rho$.
Here, we do not have this ergodicity and we end up having to juggle several measure classes on the same underlying representations in order to conclude our theorem.
This involves finding categories of algebraic representations which we have deemed $T$-linked and also involves a new take on the final step to construct an extension of $\rho$.

Finally, we show that the hypothesis of Fuchsian ends cannot be omitted by showing counterexamples to the weakened statement which also have large critical exponent.

\begin{restatable}{thm}{counterexample}\label{thm:counterexamples}
Let $n\ge 3$. Then there are infinitely many commensurability classes of infinite-volume convex cocompact hyperbolic $n$-manifolds with infinitely many properly immersed totally geodesic hypersurfaces meeting the convex core and with critical exponent $\delta_\Gamma$ at least $n-2$ but which do not cover an arithmetic manifold.
These manifolds do not have Fuchsian ends.
\end{restatable}

These are outlined in Example~\ref{ex:counterexample}.
We do not claim novelty in the ideas behind these constructions; indeed, most of the work in proving these statements follows from theorems in the literature.
However we were unable to find such manifolds written explicitly in the literature in this way so we call this a theorem herein.

\bigskip\noindent\textbf{Outline of the proofs.}
We now give a more detailed sketch of the proof, following the three steps described above.
The first step, carried out in Section~\ref{sec:limitingmeasure}, is to construct the limiting measure on which the remainder of the argument depends.
To this end, we pass to a sequence of constant dimension $m$, normalize the volume measure on each submanifold by the volume of its convex core, and average over the compact normal-frame fibers.
The geometry of the Fuchsian ends then gives uniform bounds on compact sets, so that, after passing to a subsequence, we obtain a nonzero $W_m$-invariant Radon measure $\mu_\infty$.
Passing to its transverse current gives a $\Gamma$-invariant measure on the plane space, whose restriction to $\calC^*$ we call $\eta_\infty$.

The orbit closure results of Lee and Oh \cite{LO} already tell us that the original maximal submanifolds become dense.
However, this topological information does not control where their normalized volume measures concentrate so an additional argument is needed for the limiting current.
For hypersurfaces, an atom on a closed hyperplane orbit would, after blowing up the corresponding point of hyperplane space, give an invariant probability measure for the projectivized derivative action of its stabilizer.
Zariski density of the hyperplane stabilizer in its intrinsic isometry group rules out this possibility.
In higher codimension, one must similarly exclude concentration on an entire family of planes contained in a proper closed geodesic submanifold.
For $m\ge3$, we blow up this family and compare the support restrictions imposed by recurrence at two distinct boundary intersections.
The compact case applies also when $m=2$ and, together with the orbit closure results of Lee and Oh, shows that $\eta_\infty$ is nonzero and almost every plane crosses two distinct boundary lifts.

We next show that $\eta_\infty$ is a constant multiple of Haar measure.
To convey the idea of the proof, we first describe the argument in the case of hypersurfaces of dimension at least three and then the necessary modifications.
A hyperplane crossing a lift $Z$ of a geodesic boundary component is described by its intersection with $Z$ and its angle of intersection.
Transferring the current to the frame bundle over the compact boundary component gives finite measures when the angle is restricted to a compact subinterval.
Disintegrating over this parameter allows us to apply Ratner's theorem to the conditional measures.
Each is the sum of a Haar component and an ergodic integral of homogeneous measures projecting to closed hypersurfaces of the boundary component.
The exceptional hyperplanes form countably many angle-parametrized families at each lift.
Comparing the decompositions at two distinct lifts, the Haar part at one lift gives zero mass to the exceptional families at the other, while the two exceptional sets have countable intersection.
Non-atomicity therefore eliminates both exceptional parts and the resulting boundary invariances give local $G$-invariance.
Here $m\ge3$ ensures that the boundary intersection group $W_{m-1}$ is generated by unipotents, as required for this application of Ratner's theorem.

In higher codimension, the boundary intersections need not be maximal, so the homogeneous components may be supported on smaller geodesic submanifolds of the boundary.
To circumvent this, we choose an algebraic set of least dimension carrying a positive part of the exceptional measure and compare its descriptions at two boundary lifts.
The point of this minimality is that every positive-measure subset is Zariski dense in the set.
The homogeneous boundary measures force it to be preserved by the intrinsic isometry groups of the corresponding boundary subspaces.
These groups generate the connected isometry group of their geodesic span, so the algebraic set is the family of planes in that span.
The restricted current is locally invariant under the same groups and is therefore a multiple of the intrinsic smooth measure.
Local finiteness forces the translates of this subspace to be discrete and the orbit closure theorem of Lee and Oh then makes it properly immersed.
The exceptional measure would therefore concentrate on a proper closed geodesic submanifold, which ambient maximality and the blow-up argument exclude.
The boundary conditional measures are consequently Haar, and comparison at two lifts gives the desired multiple of $\lambda$.
This is the argument of Section~\ref{sec:higherdimensionabsolute}, while Appendix~\ref{sec:hypersurfaceappendix} gives the shorter codimension-one proof.

For $m=2$, the intersections with the geodesic boundary are themselves geodesics, so the transferred measures are only known to be invariant under the frame flow.
Before we can use Ratner's theorem, we need to show that these measures are also invariant under nontrivial unipotent subgroups.
Section~\ref{sec:surfaceabsolute} obtains this additional invariance by comparing the intersections of the same family of planes with two distinct boundary lifts.
Although the resulting boundary measures live on different spaces, they describe the same current, and this imposes a relation between their entropies.

On the positive-entropy part, this relation forces the unstable leafwise measures to admit nontrivial translation symmetries up to scale.
We start with nearby points on the same stable plaque and use the frame flow and continuity to compare their unstable leafwise measures.
At the second boundary lift, the rescaled endpoint maps for these points differ asymptotically by a nontrivial translation.
Indeed, in the absence of such symmetries, strict concavity would force entropy growth incompatible with the comparison at the two boundary lifts.
Recurrence then gives invariance under nontrivial unipotent subgroups, so Ratner's theorem allows us to return to the carrier argument above.
As before, maximality and the blow-up construction would turn a proper carrier into an equivariant probability field on the projectivized tangent bundle of the plane space.
Recurrence at two boundary lifts rules out such a field, identifying the positive-entropy current with a multiple of Haar measure.
The same obstruction excludes a remaining zero-entropy current, since its small-scale behavior would also produce such a field.
Thus the entire limiting current is a multiple of $\lambda$.

In Section~\ref{sec:generalratner}, we obtain the general measure classification.
We pass from a $W_m$-ergodic Radon measure to its ergodic transverse current on the space of underlying planes.
If this current is atomic, it is supported on a single $\Gamma$-orbit, giving a properly immersed totally geodesic submanifold.
For a non-atomic current, if almost no plane crosses two boundary lifts, the frame measure is supported on a compact set and is finite, so Ratner's theorem applies.
In the remaining case, the homogeneous boundary decomposition and carrier argument above identify the current with intrinsic smooth measure on the planes in a properly immersed carrier and its translates.

To recover the original frame measure, ergodicity fixes the position of the carrier in the normal-frame coordinates and selects a normal-holonomy coset.
After a rotation centralizing $W_m$, this gives a single closed frame orbit.
We show that its restricted Haar measure is $W_m$-ergodic, and compact averaging then identifies the original measure with a multiple of this measure.
Translating back gives the subgroup description in Theorem~\ref{thm:generalmeasurerigidity}.

With the limiting current in hand, we can carry out the second step using the suspension construction from \cite{BFMS}, as described in Section~\ref{sec:suspension}.
As explained in Section~\ref{sec:arithmeticitycriterion}, the adjoint trace field of $\Gamma$ is a number field, since doubling the convex core gives a closed hyperbolic manifold whose fundamental group contains $\Gamma$.
For a natural representation arising from a place of this field, the stabilizers of the geodesic submanifolds give fixed points in a suitable projective representation.
These fixed points allow us to lift the adapted-frame measures to the corresponding suspension space.
Compactness of the projective fiber gives a limiting lifted measure and disintegrating this measure gives a $\Gamma$-equivariant map to probability measures on the projective fiber.
Theorem~\ref{thm:measurerigid} is precisely what makes this map measurable with respect to $\lambda$.

The final step, carried out in Section~\ref{sec:AREA}, is to turn this map into an extension of $\rho$.
As explained in Section~\ref{sec:insufficient}, we cannot simply follow \cite{BFMS}, since the relevant category of $U$-algebraic representations for the lifted Haar measure does not have initial objects.
Instead, we realize the same map using the lifted Bowen--Margulis--Sullivan (BMS) and Burger--Roblin (BR) measure classes, whose pushforwards to $\calC^*$ both agree with $[\lambda]$ by Section~\ref{sec:measureclasscomparison}.
After passing from probability measures to a proper coset space $\bfH/\bfL(k)$, we therefore have two algebraic representations coming from the same map but with different measure classes.
The notion of $T$-linking, developed in Section~\ref{sec:Tlinked}, provides the transfer functors that allow us to compare them.

In Section~\ref{sec:stabilizerreduction}, we work in the BR measure class and obtain an initial object in the category of $P$-algebraic representations with target $\bfH/\bfL_0(k)$, later called a $P$-gate.
If $\bfL_0$ is nontrivial, compatibility promotes the $U$-invariance of the map to $\bfH/\bfL(k)$ to $N$-invariance, while its transfer to the BMS measure class is invariant under the Weyl element.
These invariances occur on different measured spaces, but Lemma~\ref{lem:linkedconstancy} uses the transfer identity to compare them and prove constancy.
Since $\bfL$ is proper, this contradicts Zariski density of $\rho(\Gamma)$, so $\bfL_0$ must be trivial.

We complete the argument in Section~\ref{sec:superrigidityproof}.
We now have a $P$-equivariant map into $\bfH(k)$ in the BR measure class, whose transfer to the BMS measure class is also equivariant under the normalizer of $A$.
If $k$ is non-archimedean, connectedness of $P$ and total disconnectedness of $\bfH(k)$ give a contradiction by the same constancy argument.
In the archimedean case, the $P$-equivariance and the Weyl element give a real-analytic map on an open Bruhat cell.
The inequality $\delta_\Gamma>n-2$ then allows us to promote the almost-everywhere identities coming from the Patterson--Sullivan measure to real-analytic identities.
These identities give a continuous homomorphism on $G$, and after conjugation we obtain the extension of $\rho$ we were looking for.
Applying this to the natural representations from the arithmeticity criterion places $\Gamma$ in an arithmetic lattice $\Lambda$.
The compact totally geodesic boundary implies that $\Lambda$ is of simplest type and a final torsion-removal argument gives a torsion-free finite-index subgroup containing $\Gamma$.
Thus $Y_\infty$ covers an arithmetic hyperbolic manifold, completing the proof of Theorem~\ref{thm:arithmetic}.

\bigskip\noindent\textbf{Acknowledgements:} The author was partially supported by NSF grant DMS--2405264/2630020.
We are grateful to Uri Bader and Hee Oh for comments on an earlier draft of the paper that improved its exposition.

\bigskip\noindent\textbf{AI Disclosure:} Several of the ideas in Section \ref{sec:surfaceabsolute} were developed through conversations with an LLM.
In addition, the author used LLM tools for literature searches, proofreading, figure creation, and proof checking.

\section{Background}\label{sec:background}

We work in the hyperboloid model, writing $\bbR^{n,1}$ for $\bbR^{n+1}$ equipped with the standard $(n,1)$ symmetric bilinear form
$$\langle x,y\rangle=x_1y_1+\cdots+x_ny_n-x_{n+1}y_{n+1}.$$
In this model,
$$\bbH^n=\{x\in\bbR^{n,1}\mid \langle x,x\rangle=-1,\ x_{n+1}>0\},$$
and we identify $\partial\bbH^n\cong S^{n-1}$ with the rays in the null cone $\{\xi\mid \langle\xi,\xi\rangle=0,\ \xi_{n+1}>0\}$.
The orientation-preserving isometry group of $\bbH^n$ is $G=\SO_0(n,1)$, the identity component of the group preserving this form.
For our basepoint and distinguished unit tangent vector, we take $o=e_{n+1}$ and $v_0=e_n$, with endpoints $\xi^+=[e_{n+1}+e_n]$ and $\xi^-=[e_{n+1}-e_n]$.
The stabilizer $K=\Stab_G(o)\cong\SO(n)$ is a maximal compact subgroup and acts by rotations on $T_o\bbH^n=\Span(e_1,\ldots,e_n)$, therefore the map $gK\mapsto go$ identifies $G/K$ with $\bbH^n$.

To describe the corresponding boundary and geodesic subspaces, we use the minimal parabolic subgroup $P=\Stab_G(\xi^+)$ and its Langlands decomposition $P=MAN$.
The compact factor $M=\Stab_G(o,v_0)\cong\SO(n-1)$ acts by rotations on $\Span(e_1,\ldots,e_{n-1})$, while $A=\{a_t\mid t\in\bbR\}$ moves along the geodesic determined by $v_0$ and is given by
$$a_te_n=\cosh(t)e_n+\sinh(t)e_{n+1},\qquad a_te_{n+1}=\sinh(t)e_n+\cosh(t)e_{n+1},$$
with $a_t$ fixing $e_1,\ldots,e_{n-1}$.
The remaining factor is the unipotent radical of $P$, the horospherical group which we always denote by $N\cong\bbR^{n-1}$.
Writing $g^\pm=g\xi^\pm$ for $g\in G$, we can identify $G/P$ with $\partial\bbH^n$ by the map $gP\mapsto g^+$.
Right translation by $P$, hence by $N$, preserves this positive endpoint.
Keeping track of a unit tangent vector instead gives the identification $gM\mapsto(g o,g v_0)$ of $G/M$ with $T^1\bbH^n$, under which right translation by $a_t$ is the geodesic flow.
Passing to $G/MA$ forgets the position along the geodesic and identifies this quotient with the ordered pairs of distinct boundary points via $gMA\mapsto(g^-,g^+)$.

We will also use the compact description of the boundary coming from the Iwasawa decomposition $G=KAN$.
This identifies $G/AN$ with $K$ via $k\mapsto kAN$ and since $M=K\cap P$, the positive endpoint map gives $K/M\cong\partial\bbH^n$ by $kM\mapsto k\xi^+$.
Thus the projection $G/AN\to G/P$ becomes $K\to K/M$, whose fibers are right $M$-cosets.

For the spaces of totally geodesic planes, we fix an integer $2\le m<n$ and use the standard subspace
$$E_m=\Span(e_{n-m+1},\ldots,e_{n+1}).$$
There are two stabilizers to keep track of, depending on whether we remember the normal frame
$$W=W_m=\Stab_G(e_1,\ldots,e_{n-m})\cong\SO_0(m,1),\qquad S_m=\Stab_G(E_m).$$
The first is the lower-right block copy of $\SO_0(m,1)$ fixing the ordered normal frame pointwise, while the second preserves the underlying subspace.
The full stabilizer satisfies $S_m=N_G(W)$, with identity component $S_m^\circ=W\SO(n-m)$ and compact quotient $S_m/W$.
We write $C_0=\bbH^n\cap E_m$ and use the same notation in other dimensions, with $W_n=S_n=G$.
Our choice of standard subspace ensures that $A<W$, and intersecting the boundary stabilizer with $W$ gives
$$U=U_W=N\cap W\cong\bbR^{m-1},\qquad P_W=P\cap W=(M\cap W)AU.$$
Thus $P_W$ is the stabilizer of $\xi^+$ in $W$ and $U$ is its unipotent radical.

The resulting homogeneous spaces are $\calC=G/W$ and $\overline{\calC}=G/S_m$.
A point of the former records a totally geodesic $m$-plane together with its ordered orthonormal normal frame, while a point of the latter records only the underlying plane.
Forgetting the normal frame gives the natural map $p:\calC\to\overline{\calC}$, which has compact fibers $S_m/W$.
In the sequel, we call points of $\calC$ \emph{normally framed planes}, or simply planes when the normal frame plays no role.

Each underlying plane is determined by its ideal boundary, which is a round $(m-1)$-sphere in $S^{n-1}$.
Since $\partial C_0$ contains $\xi^+$, the projection $G\to G/W$ maps each right $P$-coset $gP$ onto the normally framed spheres through $g^+$, with fibers modeled on $P_W$.

\subsection{Convex cocompact hyperbolic manifolds}\label{sec:backgroundCC}

We now pass to the quotient by a torsion-free discrete subgroup $\Gamma$ of $G$, which we almost always identify with the fundamental group of the oriented hyperbolic $n$-manifold $Y_\infty=\Gamma\backslash\bbH^n$.
Recall that its \emph{limit set} $\Lambda_\Gamma\subseteq\partial\bbH^n$ consists of the accumulation points of an orbit $\Gamma o$ on the boundary, independently of the choice of $o\in\bbH^n$.
Taking its closed convex hull gives $D_\Gamma=\Hull(\Lambda_\Gamma)$, the intersection of the closed hyperbolic half-spaces whose closures in $\bbH^n\cup\partial\bbH^n$ contain $\Lambda_\Gamma$.
The quotient $\core(Y_\infty)=\Gamma\backslash D_\Gamma$ is the \emph{convex core} of $Y_\infty$ and we call the manifold \emph{convex cocompact} when this core is compact.

Following \cite[Def~1.1]{LO}, a convex cocompact hyperbolic manifold has \emph{Fuchsian ends} if its convex core has nonempty interior and totally geodesic boundary.
We assume throughout that $Y_\infty$ is an infinite-volume convex cocompact hyperbolic manifold with Fuchsian ends.
Thus $D_\Gamma$ is full-dimensional in $\bbH^n$ and the boundary of the convex core is nonempty and consists of totally geodesic hypersurfaces, that is, totally geodesic submanifolds of codimension one.
Geometrically, each connected component outside the convex core is a funnel $\Sigma\times(0,\infty)$, with metric $dt^2+\cosh^2(t)h_\Sigma$, where $\Sigma$ is a boundary component of the core and $h_\Sigma$ is its induced hyperbolic metric.

In dimension three, a compact core is \emph{acylindrical} if it has incompressible boundary and every essential annulus is boundary parallel.
Convex cocompact hyperbolic $3$-manifolds with nonempty Fuchsian ends are called \emph{rigid acylindrical} in \cite{MMO}.

For the dynamics, we work in the oriented frame bundle $\calF=\Gamma\backslash G$ of $Y_\infty$.
Here we want to single out the frames whose geodesics remain in a bounded part of the manifold, either for all time or in the forward direction.
Upstairs in $G$, the relevant sets are
$$X_0=\{g\in G\mid g^+,g^-\in\Lambda_\Gamma\},\qquad
X_+=\{g\in G\mid g^+\in\Lambda_\Gamma\}.$$
Their quotients give the \emph{renormalized frame bundle} and its one-sided counterpart
$$\calR=\Gamma\backslash X_0,\qquad
\calR_+=\Gamma\backslash X_+.$$
Thus $\calR$ requires both endpoints to lie in $\Lambda_\Gamma$, while $\calR_+$ requires only the positive endpoint.
Equivalently, if $A^+=\{a_t\mid t\ge0\}$ then
$$\calR=\{x\in\calF\mid xA\text{ is relatively compact}\},\qquad
\calR_+=\{x\in\calF\mid xA^+\text{ is relatively compact}\}.$$
These are the sets denoted $\RF Y_\infty$ and $\RF_+Y_\infty$ in \cite{LO}.
The set $\calR$ is compact and $A$-invariant, since its frames follow complete geodesics contained in the compact convex core.

This gives us a compact setting for the geodesic dynamics within the infinite-volume frame bundle.
In particular, $\calR$ supports the finite Bowen--Margulis--Sullivan measure introduced in Section~\ref{sec:PSBMSBR} so Poincar\'e recurrence applies to the frame flow on this set.
Although $\calR$ is generally not invariant under unipotent subgroups of $N$, the Fuchsian ends hypothesis gives thick returns to $\calR$ for trajectories starting there, which is a central input to the orbit closure arguments recalled in Section~\ref{sec:topologicalrigidity}.
The larger set $\calR_+$ is $N$-invariant and provides the natural setting for the corresponding unipotent orbit closures and the Burger--Roblin measure.

\subsection{The plane space and its smooth measure}\label{sec:planespace}

The space of normally framed planes meeting the interior of the lifted convex core is
$$\calC^*=\{C\in\calC\mid C\cap\interior(D_\Gamma)\ne\emptyset\},$$
and similarly its image under the projection map $p$ is denoted $\overline{\calC}^*$.
Both are open $\Gamma$-invariant subsets of their respective homogeneous spaces.
The corresponding subset $\calF^*=\Gamma\backslash\pi_W^{-1}(\calC^*)$ of $\calF$ is right $W$-invariant, where $\pi_W:G\to G/W$ is the quotient map.
Its definition records the plane associated to a frame though the basepoint of that original frame need not itself lie in the convex core.

The space $\calC$ carries a nonzero $G$-invariant Radon measure $\lambda_{\calC}$, unique up to multiplication by a positive constant.
This is the quotient measure induced by Haar measure $\mu_G$ on $G$.
It is smooth, meaning that in a smooth coordinate chart of dimension $d$ it has the form
$$d\lambda_{\calC}(x)=f(x)\,dx_1\cdots dx_d,$$
where $f$ is smooth and strictly positive.
Thus its null sets agree with those of Lebesgue measure in every chart.
The analogous invariant measure $\overline{\lambda}_{\calC}$ on $\overline{\calC}$ is normalized so that $\lambda_{\calC}$ is its Haar probability lift along $p$.

Throughout, we write $\lambda=\lambda_{\calC}\vert_{\calC^*}$ and refer to it as Haar measure on $\calC^*$.
Similarly, $\overline{\lambda}$ is its projection to $\overline{\calC}^*$.
These are smooth $\Gamma$-invariant Radon measures, positive on every nonempty open subset.
The measure class $[\lambda]$ is independent of its normalization.
Whenever $\lambda$ is regarded as a measure on all of $\calC$, we mean its extension by zero outside $\calC^*$.

\subsection{Topological orbit closure rigidity}\label{sec:topologicalrigidity}

Ratner's orbit closure theorem \cite{RatnerOrbit91} gives the finite-volume model for the rigidity results discussed in this subsection.
In our infinite-volume setting, the geometric hypotheses on $Y_\infty$ allow analogous conclusions for unipotent flows and totally geodesic submanifolds.
We recall the work of McMullen--Mohammadi--Oh \cite{MMOhoro,MMO,MMOcore} in dimension three and its higher-dimensional extension by Lee--Oh \cite{LO} and explain the consequences for the space $\calC^*$ used here.
The latter authors also establish properness and finiteness results for totally geodesic submanifolds contained in the convex core in \cite{LOcore}.

First consider the setting where $n=3$.
A geodesic plane means the image of a totally geodesic immersion $\bbH^2\to Y_\infty$.
The following is the part of the classification in \cite[Thms~1.1, ~1.5]{MMO} that concerns planes meeting the interior of the convex core.

\begin{thm}[McMullen--Mohammadi--Oh]\label{thm:MMOplanes}
Suppose that $n=3$ and let $\Sigma\subset Y_\infty$ be a geodesic plane meeting $\interior(\core(Y_\infty))$.
Then either $\Sigma$ is a properly immersed convex cocompact hyperbolic surface or $\Sigma$ is dense in $Y_\infty$.
In the former case, the surface may be compact or have infinite area.
\end{thm}

McMullen--Mohammadi--Oh also treat convex cocompact acylindrical $3$-manifolds whose convex core boundary need not be totally geodesic \cite[Thms~1.1--1.3]{MMOcore}.
For such a manifold $Z$, put $Z^*=\interior(\core(Z))$.
Then they show that every nonempty intersection $\Sigma\cap Z^*$ is closed or dense in $Z^*$, there are only countably many closed planes in $Z^*$, and any sequence of distinct closed planes becomes dense in $Z^*$.
Benoist and Oh \cite{BOPlanes} extend the closed-or-dense dichotomy and countability of closed planes in $Z^*$ to geometrically finite acylindrical hyperbolic $3$-manifolds.
This is not the case if one considers all of $Z$ instead of $Z^*$ \cite[Thm~1.2]{Zhang}.

The higher-dimensional result is naturally stated on the renormalized frame bundle $\calR$ from Section~\ref{sec:backgroundCC}.

\begin{thm}[Lee--Oh, {\cite[Cor~1.4]{LO}}]\label{thm:LOrelative}
Let $H\le G$ be a nontrivial connected closed subgroup generated by one-parameter unipotent subgroups and normalized by $A$.
For every $x\in\calR$, there is a connected closed subgroup $L\le G$ containing $H$, with $xL$ closed in $\calF$, such that
$$\overline{xH}\cap\calR=xL\cap\calR.$$
\end{thm}

The closure in this formula is taken in the oriented frame bundle $\calF$ from Section~\ref{sec:backgroundCC}.
In particular, this gives a relative homogeneity statement; it describes the closure on $\calR$, which need not itself be $H$-invariant.
Theorem~1.2 of \cite{LO} gives more precise orbit closure statements on the appropriate invariant sets, including that orbit closures of nontrivial connected closed subgroups $U\le N$ are of the form $\overline{xU}=xL\cap\calR_+$ for $x\in\calR_+$ and for a suitable closed homogeneous orbit $xL$.

As $W=W_m$ is generated by unipotent subgroups, these theorems apply to the geodesic submanifolds considered here.
For the geometric orbit statements, it is convenient to forget the normal frames and simply use the space of planes $\overline{\calC}=G/S_m$.
Indeed, closed $S_m$-orbits correspond to properly immersed geodesic $m$-dimensional submanifolds, whereas closed $W$-orbits may retain additional normal holonomy.
The description in \cite[Cor~1.9, Thm~1.7(2)]{LO} gives the following.

\begin{cor}\label{cor:planeclosures}\label{cor:hyperplaneclosures}
For every underlying plane $C\in\overline{\calC}^*$, there is a lift $F$ to $\bbH^n$ containing $C$ of a properly immersed closed geodesic submanifold such that the closure of $\Gamma C$ in $\overline{\calC}^*$ consists of the core-intersecting $m$-planes contained in the translates of $F$.
There are only countably many properly immersed geodesic submanifolds of dimension at least two meeting the core.
Moreover, if $\{\Gamma C_i\}$ is an infinite sequence corresponding to distinct maximal such submanifolds of fixed dimension $m$, every nonempty open subset of $\overline{\calC}^*$ meets $\Gamma C_i$ for all sufficiently large $i$.
\end{cor}

In codimension one, the first statement reduces to the usual closed-or-dense dichotomy for hypersurfaces.
The final statement of Corollary~\ref{cor:hyperplaneclosures} is a purely topological one.
We will obtain the requisite measure-theoretic conclusions in Sections~\ref{sec:limitingmeasure}--\ref{sec:surfaceabsolute}.

For use later, we record the following consequence of this corollary.

\begin{lemma}\label{lem:aeboundarylift}
Every underlying plane in $\overline{\calC}^*$ crossing at most one boundary lift is contained in a lift of a compact closed ambient geodesic submanifold.
Consequently the subset $\calU_\partial\subseteq\calC^*$ of normally framed planes crossing two distinct boundary lifts is $\lambda$-conull.
\end{lemma}

\begin{proof}
Let $C\in\overline{\calC}^*$ cross at most one boundary lift, and choose $F\supseteq C$ as in Corollary~\ref{cor:planeclosures}.
Put $\Gamma_F=\Gamma\cap\Stab_G(F)$ and $D_F=F\cap D_\Gamma$.

Suppose that $\Gamma_F\backslash F$ is not compact.
Then $\Gamma_F\backslash D_F$ is a compact convex core with nonempty totally geodesic boundary.
There must be at least two components of $\partial D_F$.
Indeed, if there were only one, say $B$, then $\Gamma_F$ would preserve $B$, so $\Lambda_{\Gamma_F}\subseteq\partial B$ and hence $D_F\subseteq B$, contrary to the fact that $D_F$ has nonempty interior in $F$.
Choose points in two distinct components of $\partial D_F$ and an $m$-plane $C'\subseteq F$ containing the geodesic through them.
Then $C'$ meets $\interior(D_\Gamma)$ and crosses the two corresponding ambient boundary lifts.
Corollary~\ref{cor:planeclosures} gives $C'\in\overline{\Gamma C}$.
Since crossing these two lifts is an open condition, some translate of $C$ crosses both, a contradiction.
Thus $\Gamma_F\backslash F$ is compact.

By Corollary~\ref{cor:planeclosures}, there are only countably many such compact submanifolds and their lifts $F$.
Each $F$ is proper since $Y_\infty$ has infinite volume.
For each fixed $F$, the set $\{C'\in\overline{\calC}\mid C'\subseteq F\}$ is a submanifold of positive codimension and therefore has zero $\overline{\lambda}$-measure.
Taking the countable union and lifting along $p:\calC\to\overline{\calC}$ proves the second assertion.
\end{proof}

\subsection{Patterson--Sullivan, BMS, and BR measures}\label{sec:PSBMSBR}

We now fix the measures and notation used throughout the paper.
Recall that the \emph{critical exponent} of $\Gamma$ is
$$\delta_\Gamma=\inf\left\{s>0\mid \sum_{\gamma\in\Gamma}e^{-s\,d(o,\gamma o)}<\infty\right\}.$$
For $\xi\in\partial\bbH^n$, we use the Busemann cocycle
$$\beta_\xi(z,z')=\lim_{r\to\infty}\bigl(d(z,c(r))-d(z',c(r))\bigr),$$
where $c$ is any geodesic ray ending at $\xi$.
The Patterson--Sullivan construction gives a family $\{\nu_z\mid z\in\bbH^n\}$ of finite, non-atomic measures supported on $\Lambda_\Gamma$ satisfying
$$\gamma_*\nu_z=\nu_{\gamma z},\qquad
\frac{d\nu_z}{d\nu_{z'}}(\xi)=e^{-\delta_\Gamma\beta_\xi(z,z')}.$$
In our convex cocompact setting this density is unique up to scale.
Throughout, $\nu$ will always denote the Patterson--Sullivan measure $\nu_o$, normalized by $\nu(\partial\bbH^n)=1$.
The conformal relation shows that all $\nu_z$ have the same measure class $[\nu]$, which is $\Gamma$-invariant.
For the construction and its basic properties, see \cite{Patterson76,Sullivan79} and \cite[Ch~1]{Roblin03}.
We also write $\omega_z$ for the visual probability measure on $\partial\bbH^n$ based at $z$ obtained from normalized spherical measure on $T_z^1\bbH^n$ and put $\omega=\omega_o$.
The family $\{\omega_z\}$ is a $G$-equivariant conformal density of dimension $n-1$ satisfying the same Radon--Nikodym formula with $\delta_\Gamma$ replaced by $n-1$.

Recall that $\supp\nu=\Lambda_\Gamma$ is Zariski dense in the boundary $\partial\bbH^n$.
In our infinite-volume setting, we have the strict inequality $\delta_\Gamma=\dim_H\Lambda_\Gamma<n-1$ so in particular $\nu$ is singular with respect to $\omega$.
On the other hand, the stabilizer of a geodesic boundary lift is a cocompact lattice in $\SO_0(n-1,1)$, with critical exponent $n-2$ and infinite index in $\Gamma$.
Strict monotonicity of the critical exponent for this divergent subgroup \cite[Prop~2]{DOP} gives $\delta_\Gamma>n-2$.
In particular, for every $m\ge2$ considered here
$$\delta_\Gamma+\dim U=\delta_\Gamma+m-1>n-1,$$
which is the inequality used in the BR ergodicity and measure-class comparison arguments below.

The Hopf coordinates on $G/M=T^1\bbH^n$ are
$$gM\longmapsto(x,y,t)=\bigl(g^+,g^-,\beta_{g^+}(o,go)\bigr),$$
where $x,y$ are distinct.
Thus right translation by $a_s$ sends $t$ to $t+s$ and fixes the two endpoints.
After choosing a local trivialization of $G\to G/M$, we adjoin the compact coordinate in $M$ and write $dm_M$ for Haar probability measure on $M$.
In these coordinates, with $p=go$, the lifted Bowen--Margulis--Sullivan and Burger--Roblin measures on $G$ are given by
\begin{align*}
 d\widetilde{\mu}_{\BMS}(g)
 &=e^{\delta_\Gamma\bigl(\beta_x(o,p)+\beta_y(o,p)\bigr)}
   \,d\nu(x)\,d\nu(y)\,dt\,dm_M,\\
 d\widetilde{\mu}_{\BR}(g)
 &=e^{\delta_\Gamma\beta_x(o,p)+(n-1)\beta_y(o,p)}
   \,d\nu(x)\,d\omega(y)\,dt\,dm_M.
\end{align*}
These formulas are independent of the basepoint used to express them and define left $\Gamma$-invariant Radon measures.
They descend to measures $\mu_{\BMS}$ and $\mu_{\BR}$ on the frame bundle $\Gamma\backslash G$.
Throughout, tildes distinguish the measures on $G$ from their quotient measures on $\Gamma\backslash G$.
Our BR convention places $\nu$ at the positive endpoint, which is the endpoint fixed by right translation by $N$.
For these local descriptions and the passage from the unit tangent bundle to the frame bundle, see \cite[\S2.2, 4.1]{MS} and \cite[\S2.5--2.6]{Winter}.
See also \cite{Roblin03} for a fuller treatment.

The sets $X_0$ and $X_+$ are the supports of $\widetilde{\mu}_{\BMS}$ and $\widetilde{\mu}_{\BR}$, respectively, and their quotients $\calR$ and $\calR_+$ are the supports of $\mu_{\BMS}$ and $\mu_{\BR}$.
The BMS measure is right $MA$-invariant and is preserved by a Weyl element exchanging the endpoints, while the BR measure is right $MN$-invariant and its class is preserved by $P=MAN$.
The quotient BMS measure is finite, with compact support in the convex cocompact setting.
The quotient BR measure is locally finite but has infinite total mass.

For the measure classes, the Busemann factors are strictly positive and therefore do not affect null sets.
Consequently, in each Hopf coordinate chart
\begin{align*}
 [d\widetilde{\mu}_{\BMS}]&=[d\nu(x)\,d\nu(y)\,dt\,dm_M],\\
 [d\widetilde{\mu}_{\BR}]&=[d\nu(x)\,d\omega(y)\,dt\,dm_M].
\end{align*}
In a local trivialization over the endpoint map $G\to G/P$, the BR class is $[d\nu(x)\,dm_P]$, where $dm_P$ is Haar measure on $P$.
Thus BR has a smooth measure class along each $P$-fiber and the Patterson--Sullivan class transversely.
However, the two measures on $G$ are not equivalent.
Indeed, $\widetilde{\mu}_{\BMS}$ is supported on $X_0$, whereas $\widetilde{\mu}_{\BR}(X_0)=0$ because $\omega(\Lambda_\Gamma)=0$.
Their projected classes on $\calC^*$ will be compared in Proposition~\ref{prop:BRprojection}.

Whenever we push forward a $\sigma$-finite measure class $[\sigma]$ under a measurable map $q$, we mean $q_*[\sigma]=[q_*(f\sigma)]$, where $f>0$ is integrable with respect to $\sigma$.
This class is independent of $f$ and is determined by the condition that a set is null precisely when its preimage is $\sigma$-null.
This convention applies even when the unweighted pushforward $q_*\sigma$ is not locally finite.

We record here the ergodicity statements for these measures that will be used in Section~\ref{sec:AREA}.
For the BMS measure we will need the mixing property, which is a theorem of Winter \cite{Winter}.

\begin{thm}[Winter]\label{thm:BMSmixing}
The right $A$-action on $(\Gamma\backslash X_0,\mu_{\BMS})$ is mixing, hence weakly mixing.
\end{thm}

\begin{proof}
The group $\Gamma$ is Zariski dense and its BMS measure is finite so the mixing assertion is precisely \cite[Thm~1.1]{Winter}.
\end{proof}

For the BR measure, Mohammadi and Oh \cite{MO} proved ergodicity of one-dimensional unipotent flows for convex cocompact hyperbolic $3$-manifolds when $\delta_\Gamma>1$.
The general criterion we use below, in which the critical exponent exceeds the codimension of the acting subgroup in $N$, is due to Maucourant and Schapira \cite{MS}.

\begin{thm}[Maucourant--Schapira]\label{thm:BRsubgroupergodicity}
Let $V\le N\cong\bbR^{n-1}$ be a nontrivial connected subgroup.
If the BMS measure is finite and $\delta_\Gamma>\operatorname{codim}_N V$, then the right $V$-action on $(\Gamma\backslash X_+,\mu_{\BR})$ is ergodic.
In particular, under our standing hypotheses every nontrivial connected subgroup of $N$ acts ergodically.
\end{thm}

\begin{proof}
This follows from \cite[Thm~1.3]{MS}.
In our setting $\delta_\Gamma>n-2$, whereas $\operatorname{codim}_N V\le n-2$ for every nontrivial connected $V\le N$, so the result follows.
\end{proof}

\subsection{Leafwise measures and entropy for compact boundary flows}\label{sec:leafwiseentropy}

For the surface case $m=2$ in arbitrary ambient dimension $n\ge3$, we will need some facts about leafwise measures and entropy for the frame flows over boundary components of $\core(Y_\infty)$.
We recall these here and fix the normalizations used in Section~\ref{sec:surfaceabsolute}.

Let $H=\SO_0(n-1,1)$ and let $\Delta<H$ be a cocompact lattice.
Put $X=\Delta\backslash H$ and let $A=\{a_t\}<H$ be the standard split subgroup, parametrized as in Section~\ref{sec:background}.
We also allow a compact parameter space $I$ on which $A$ acts trivially, which will record the angles of intersection in Section~\ref{sec:surfaceabsolute}.
Write $U^+,U^-\cong\bbR^{n-2}$ for the two full horospherical subgroups of $H$, with coordinates chosen so that
\begin{equation}\label{eqn:rootconjugation}
 a_tu^+(v)a_{-t}=u^+(e^{-t}v),\qquad
 a_tu^-(v)a_{-t}=u^-(e^tv).
\end{equation}
A \emph{$U^\pm$-plaque} through $x\in X$ is a local piece $xU^\pm_\epsilon$ of the $U^\pm$-orbit, where $U^\pm_\epsilon$ is a sufficiently small connected neighborhood of the identity on which $u\mapsto xu$ is injective.
On $X\times I$, we keep the parameter in $I$ fixed.
By Equation~\eqref{eqn:rootconjugation}, right translation by $a_t$ expands the $U^+$-plaques and contracts the $U^-$-plaques in positive time.
We refer to them as unstable and stable plaques, respectively.
When an adapted frame identifies this model with a boundary group, we use the same notation for the transported subgroups and coordinates.

We retain the compact normal-frame coordinate throughout.
Write $M_H\cong\SO(n-2)$ for the compact centralizer of $A$ in $H$ and choose local coordinates $\varpi\mapsto\sfm(\varpi)$ about its identity.
Near a frame $x_0\in X$, we use the chart
\begin{equation}\label{eqn:rootflowbox}
(v,\varpi,u,\tau)\longmapsto x_0u^-(v)\sfm(\varpi)u^+(u)a_\tau,
\end{equation}
where $u,v\in\bbR^{n-2}$, $\varpi\in\bbR^{(n-2)(n-3)/2}$, and $\tau\in\bbR$ range over sufficiently small sets.
We take the chart small enough to lift injectively to $H$.

Let $\sigma$ be a finite $A$-invariant measure on $X\times I$.
In such a box, $A$-invariance gives a transverse measure $\kappa$ with
$$d\sigma=d\kappa(v,\varpi,u,c)\,d\tau,$$
where $c\in I$ and $d\tau$ is Lebesgue measure.
Disintegrating over $\xi=(v,\varpi,c)$ gives
\begin{equation}\label{eqn:plaqueconditionals}
 d\kappa(\xi,u)=d\vartheta(\xi)\,r_\xi(du).
\end{equation}
The measures $r_\xi$ are the conditional measures on the unstable plaques, which we also call \emph{plaque conditionals}.
We take them to be probabilities on a fixed, slightly larger plaque, so that $\vartheta$ is the transverse measure obtained from this restriction.
The stable conditionals are defined in the same way, using a chart with the two root groups interchanged.

For a point $x$ with coordinates $(v,\varpi,u_0,\tau_0,c)$, the centered root parameter is $z=e^{\tau_0}(u-u_0)$.
Pushing $r_\xi$ forward by this map gives, up to normalization and restriction, the unstable leafwise measure at $x$.
Write $\sigma_x^+$ and $\sigma_x^-$ for the \emph{complete leafwise measures}, which are Radon measures on the whole root group $\bbR^{n-2}$, determined up to a positive scalar.
Their restrictions give the local conditionals just described.
We write $[\sigma_x^\pm]_{\mathrm{sc}}$ for their scalar rays, while brackets without this subscript continue to denote measure classes determined by null sets.

For $T_v(u)=u+v$ and $D_b(u)=bu$, the translation and flow rules are
\begin{equation}\label{eqn:leafwisecovariance}
\begin{aligned}
[\sigma_{xu^\pm(v)}^\pm]_{\mathrm{sc}}&=[(T_{-v})_*\sigma_x^\pm]_{\mathrm{sc}},\\
[\sigma_{xa_t}^+]_{\mathrm{sc}}&=[(D_{e^t})_*\sigma_x^+]_{\mathrm{sc}},\\
[\sigma_{xa_t}^-]_{\mathrm{sc}}&=[(D_{e^{-t}})_*\sigma_x^-]_{\mathrm{sc}}.
\end{aligned}
\end{equation}
The identities hold on the usual conull sets, with the parameter in $I$ included in $x$.
In particular, rescaling an unstable restriction of radius $e^{-t}$ at $x$ gives a bounded restriction of the complete measure at $xa_t$.
We also use compatibility with disintegration over an $A$-invariant factor, including the ergodic component and the angle parameter, as in \cite[Thm~6.3, Lem~7.16, Prop~7.22]{ELClay} and \cite[Lem~3.8]{ELRankOne}.
Passing to almost every such component leaves the complete leafwise measures unchanged up to scale.

For the unstable measure, define its \emph{projective stabilizer} to be the translation stabilizer of its scalar ray, given by
\begin{equation}\label{eqn:projectiveperiods}
\calL_x=\{v\in\bbR^{n-2}\mid (T_v)_*\sigma_x^+\text{ is proportional to }\sigma_x^+\}.
\end{equation}
Each $\calL_x$ is a closed additive subgroup, its identity component is a vector subspace, and Equation~\eqref{eqn:leafwisecovariance} gives $\calL_{xa_t}=e^t\calL_x$.
A nonzero element of $\calL_x$ will be called a \emph{projective period}.
The use of the complete measure is important here, since a translation must preserve the same ray on every bounded set with one common multiplier.
When we need a representative of this ray, choose a nonnegative $\chi\in C_c^1(\bbR^{n-2})$, equal to one near zero, and divide by $\sigma_x^+(\chi)$.
This denominator is positive and finite, since the origin belongs to the support almost everywhere.
This normalization does not restrict the complete measure.

We now recall the entropy and dimension statements used below.
Normalize $\sigma$ to probability and use base-$2$ logarithms.
For a finite measurable partition $\calP$, put
$$H_\sigma(\calP)=-\sum_{P\in\calP}\sigma(P)\log\sigma(P),$$
with $0\log 0=0$.
The measure-theoretic entropy of a $\sigma$-preserving map $T$ is
$$h_\sigma(T)=\sup_{\calP}\lim_{j\to\infty}\frac1j
H_\sigma\left(\bigvee_{i=0}^{j-1}T^{-i}\calP\right).$$
We refer to \cite[\S3]{ELClay} for these definitions and the entropy chain rule.
Write $h(\sigma)=h_\sigma(a_{\ln2})$ and let $h_\sigma(x)$ be the entropy of the ergodic component through $x$.
This time step doubles unstable distances and Haar probability has entropy $n-2$.

For a Radon measure $\beta$, write
$$d_\beta(z)=\lim_{r\downarrow0}\frac{\log\beta(B(z,r))}{\log r},$$
when the limit exists.
The Ledrappier--Young dimension formula \cite[Prop~7.3.1, Thm~C]{LedrappierYoungII} gives
$$d_{\sigma_x^+}(0)=d_{\sigma_x^-}(0)=h_\sigma(x),$$
for almost every $x$.
Equivalently, on almost every ergodic component of entropy $h$, the stable and unstable conditionals have pointwise dimension $h$.
In particular, on almost every ergodic component of positive entropy, both conditionals are non-atomic.
The compact isometric extension $X\to X/M_H$ preserves entropy \cite{LedrappierWalters}, without requiring an individual component to be $M_H$-invariant.

We also need the measures obtained by projecting to an endpoint.
Using the Hopf conventions of Section~\ref{sec:PSBMSBR} for $H$ in place of $G$, we retain the compact coordinate $\varpi$ and write
\begin{equation}\label{eqn:boundaryhopf}
 h\longmapsto\bigl(e^-(h),e^+(h),\varpi(h),t(h)\bigr).
\end{equation}
The endpoints and $\varpi$ are fixed by $A$, while right translation by $a_s$ sends $t$ to $t+s$.
Write $\widetilde{\sigma}$ for the local lift of $\sigma$.
On a bounded product Hopf box $\calB=Q\times I_t$, $A$-invariance gives
\begin{equation}\label{eqn:hopfdisintegration}
 d\widetilde{\sigma}\vert_{\calB}=d\kappa_Q(e^-,e^+,\varpi,c)\,dt\vert_{I_t}.
\end{equation}
Define the \emph{endpoint measures associated to $\calB$} by
\begin{equation}\label{eqn:localizedendpointmarginals}
 \beta_{\calB}^\pm=(e^\pm)_*(\widetilde{\sigma}\vert_{\calB})
 =|I_t|\,(\pr_{e^\pm})_*\kappa_Q.
\end{equation}
The restriction to $\calB$ precedes projection so these are finite measures.
After normalization, they average the corresponding plaque conditionals in endpoint coordinates.
For almost every $x\in\calB$, we have
\begin{equation}\label{eqn:endpointdimension}
 d_{\beta_{\calB}^\pm}(e^\pm(x))=h_\sigma(x).
\end{equation}
For ergodic probabilities this is \cite[Thm~2.4 ]{KaimanovichEntropy} and the remark which follows it, and the componentwise identity follows from the local entropy formula \cite{BrinKatokEntropy}.
Geodesics with the same forward endpoint are asymptotic after adjusting the $A$-coordinate, so their projected component entropies agree almost everywhere.
Reversing time gives the same assertion for the backward endpoint.

The comparison in Section~\ref{sec:surfaceboundary} is made on finite partitions of an endpoint chart.
Let $\kappa$ be a probability on bounded coordinates $(V,R)$, and let $\calD_j$ be the partition of the $V$-coordinate into half-open dyadic cubes of side $2^{-j}$.
Write $V_j$ for the cube containing $V$.
If $\beta$ is the distribution of $V$, its entropy at this scale is
$$H(V_j)=H_\beta(\calD_j)=-\sum_{P\in\calD_j}\beta(P)\log\beta(P).$$
Disintegrating over the transverse coordinate gives $d\kappa(V,R)=d\vartheta(R)\,\beta_R(dV)$.
Therefore we put
$$H(V_j\mid R)=\int H_{\beta_r}(\calD_j)\,d\vartheta(r).$$
Here we hold $R$ fixed in full, including any compact frame and angle coordinates, and partition only the endpoint coordinate.
All measures in these entropy formulas are normalized to probability and the entropies are not divided by $j$.
The comparison between pointwise dimension and these entropies is the one in \cite[Lem~3.3, Prop~3.4]{HochmanShmerkin}, which we apply below on bounded charts.

Finally, we record how a difference between conditional probabilities contributes to entropy.
Put $\calI_j=H(V_j)-H(V_j\mid R)$.
For $j'>j$, given a coarse cube $Q\in\calD_j$ and a fine cube $P\in\calD_{j'}$ which is contained in $Q$, write
$$\sfp_{r,Q}(P)=\frac{\beta_r(P)}{\beta_r(Q)},\qquad
\sfq_Q(P)=\frac{\beta(P)}{\beta(Q)},$$
whenever the denominators are nonzero.
For probability vectors indexed by the fine cubes $P\in\calD_{j'}$ contained in $Q$, write $H(\sfp)=-\sum_P\sfp(P)\log\sfp(P)$ and define the relative entropy (Kullback--Leibler divergence) \cite{KullbackLeibler} by
$$D_{\mathrm{KL}}(\sfp\Vert\sfq)=\sum_P\sfp(P)\log\frac{\sfp(P)}{\sfq(P)}.$$
Terms with $\sfp(P)=0$ are zero and the expression is infinite if $\sfp(P)>0=\sfq(P)$.
Expanding the conditional entropies gives
\begin{equation}\label{eqn:conditionalentropydifference}
\begin{aligned}
\calI_{j'}-\calI_j
&=H(V_{j'}\mid V_j)-H(V_{j'}\mid R,V_j),\\
&=\int D_{\mathrm{KL}}(\sfp_{R,V_j}\Vert\sfq_{V_j})\,d\kappa\ge0.
\end{aligned}
\end{equation}
Thus the entropy difference measures the separation between the conditional probability vector on a plaque and its average over the transverse coordinate in the coarse cube.
Zero-denominator cases have zero measure in this integral.

The elementary strict-concavity estimate we use is
\begin{equation}\label{eqn:strictconcavitygap}
\begin{aligned}
 D_{\mathrm{KL}}(\sfp\Vert\sfq)+D_{\mathrm{KL}}(\sfp'\Vert\sfq)
 &\ge 2H\left(\frac{\sfp+\sfp'}2\right)-H(\sfp)-H(\sfp').
\end{aligned}
\end{equation}
Indeed, the difference between the two sides is $2D_{\mathrm{KL}}((\sfp+\sfp')/2\Vert\sfq)\ge0$.
On a fixed finite set, the right side has a positive lower bound whenever $\|\sfp-\sfp'\|_1$ is bounded away from zero, by strict concavity and compactness.
In Section~\ref{sec:surfaceperiods}, we will apply this estimate to two separated plaque images in the same coarse cube, obtaining a definite entropy increase upon averaging over the transverse coordinate.

\section{Construction of the limiting measure and its transverse current}\label{sec:limitingmeasure}

Retain the notation of Section~\ref{sec:backgroundCC}, where $D_\Gamma=\Hull(\Lambda_\Gamma)$ is the closed convex hull in $\bbH^n$ and the convex core is $\core(Y_\infty)=\Gamma\backslash D_\Gamma$.
We now fix the given sequence of maximal properly immersed totally geodesic submanifolds as in Theorem~\ref{thm:measurerigid}, passing to a subsequence of constant dimension $m$ and retaining $W=W_m$ and its full stabilizer $S_m$.
We can and do discard the finitely many compact maximal members by \cite[Thm~1.7(3)]{LO}.
Similarly, a maximal plane meeting the core only in its boundary must be one of the finitely many boundary hypersurfaces, so we may discard these as well.
In particular, we may assume that every plane in our sequence has infinite volume and meets the interior of the convex core.

Choose lifts $C_i=g_iC_0\subset\bbH^n$, and put
$$\Delta_i=\Gamma\cap g_iS_mg_i^{-1},\qquad \Sigma_i=\Delta_i\backslash C_i,\qquad \calF_i=\Gamma g_iS_m\subset\calF.$$
Here the action of $\Delta_i$ on $C_i$ is faithful, since its kernel is discrete and contained in a compact group, and $\Gamma$ is torsion-free.
By the properness of the immersion and \cite[Prop~4.7]{LO}, $\Sigma_i$ is convex cocompact with Fuchsian ends and $C_i\cap D_\Gamma=\Hull_{C_i}(\Lambda_{\Delta_i})$.
Consequently
$$\core(\Sigma_i)=\Delta_i\backslash(C_i\cap D_\Gamma),$$
is compact with totally geodesic boundary.
Let $V_i$ be the volume of this convex core.
The geometry in a boundary collar gives a uniform comparison between the volume outside the convex core and the volume of the corresponding inward collars.
Summing these comparisons, we obtain
$$\vol(\Sigma_i\cap N_R)\le C(R)V_i,$$
where $N_R$ is the closed $R$-neighborhood of $\core(Y_\infty)$ and, crucially, $C(R)$ is independent of $i$.

To obtain measures on the frame bundle, equip the compact fiber of the adapted-frame bundle over $\Sigma_i$ with Haar probability measure.
Let $\mu'_i$ denote the resulting measure on $\calF_i$, which is right $S_m$-invariant and pushes forward to hyperbolic volume on $\Sigma_i$.
Finally, set $\mu_i=\mu'_i/V_i$.
Then if $\calN_R\subset\calF$ denotes the full frame bundle over $N_R$, we have
\begin{equation}\label{eqn:compactbounds}
\mu_i(\calN_0)=1,\qquad \mu_i(\calN_R)\le C(R).
\end{equation}
Since these compact sets exhaust $\calF$, the family of measures is uniformly bounded on compact sets.
After passing to a subsequence, the $\mu_i$ therefore weak-$*$ converge to a Radon measure which is right $S_m$-invariant, and we denote this limit by $\mu_\infty$.
Moreover, the limit is not the zero measure by Equation~\eqref{eqn:compactbounds}.

We next pass the limiting measure to the space of planes by taking its transverse current.
To this end, let $\pi_W:G\to\calC=G/W$ be the natural projection map and let $\widetilde{\mu}$ denote the $\Gamma$-invariant lift of a right $W$-invariant Radon measure $\mu$ on $\Gamma\backslash G$.
On a sufficiently small open set $\calO\subset\calC$, choose a smooth section $\sigma:\calO\to G$, which gives a local trivialization
\begin{align*}
\calO\times W&\longrightarrow\pi_W^{-1}(\calO),\\
(C,w)&\longmapsto\sigma(C)w.
\end{align*}
In these coordinates, right $W$-invariance gives the unique expression
$$d\widetilde{\mu}=d\eta^{\calC}(C)\,dm_W(w).$$
On overlaps these transverse measures agree and hence determine a $\Gamma$-invariant Radon current on $\calC$.
The construction is continuous for weak-$*$ convergence, so the convergence of our frame measures gives $\eta_i^{\calC}\to\eta_\infty^{\calC}$.
A useful consequence of the right $S_m$-invariance of the frame measures is that the resulting currents are invariant under the compact group $S_m/W$.
Writing $p:\calC\to\overline{\calC}=G/S_m$ for the map forgetting the normal frame, there are therefore currents $\overline{\eta}_i^{\calC}$ and $\overline{\eta}_\infty^{\calC}$ on $\overline{\calC}$ whose Haar probability lifts are the currents just constructed.
Naturally, for each $i$ the current $\overline{\eta}_i^{\calC}$ is supported on the discrete orbit $\Gamma C_i$ of underlying planes.

We can exclude atoms of these limit measures by studying the derivative action of the stabilizer of the corresponding plane.
The following general lemma gives the criterion we need.
For hypersurfaces this alone would be sufficient for our purposes.
However, in higher codimension we will also need to exclude concentration on entire families of planes contained in a fixed submanifold.
We return to this additional issue in Section~\ref{sec:incidencenonconcentration}.

\begin{lemma}\label{lem:noprob}
Let $\Gamma$ be a countable group acting by $C^1$ diffeomorphisms on a smooth manifold $Y$ and suppose that $q\in Y$ is fixed by $\Gamma$.
Let $\eta_i$ be $\Gamma$-invariant Radon measures weak-$*$ converging to $\eta$ with $\eta_i(\{q\})=0$ for every $i$.
If the projectivized derivative action of $\Gamma$ on $\bbP(T_qY)$ admits no invariant probability measure, then $\eta(\{q\})=0$.
\end{lemma}

\begin{proof}
Let $\beta:\Bl_q(Y)\to Y$ be the real projective blow-up with exceptional divisor $E=\beta^{-1}(q)=\bbP(T_qY)$.
The action of $\Gamma$ lifts canonically to the blow-up and each $\eta_i$ has a unique lift $\widehat{\eta}_i$ giving zero mass to $E$.
Since $\beta$ is proper, we may pass to a subsequence to obtain a limit $\widehat{\eta}$ satisfying $\beta_*\widehat{\eta}=\eta$.
If $\eta(\{q\})>0$, we may therefore normalize the restriction of $\widehat{\eta}$ to $E$ to obtain an invariant probability measure for the projectivized derivative action.
This contradicts the hypothesis of the lemma.
\end{proof}

To see how this applies to hypersurfaces, note that the tangent representation at an underlying hyperplane is the standard representation of its stabilizer.
For a closed orbit meeting the core, Zariski density and the usual Furstenberg lemma \cite{Furstenberg,Zimmer} show that there can be no invariant probability measures on $\bbP(T_qY)$.
Since the original hypersurfaces are distinct, Lemma~\ref{lem:noprob} therefore excludes limiting atoms on these hyperplanes.
The orbit closure theorem above excludes all other atoms on $\calC^*$.

For general $m$, we have a stronger avoidance property for the original currents coming from maximality.
To state it, we first fix some notation for hyperbolic subspaces.
For a hyperbolic subspace $F\subseteq\bbH^n$, write $V_F\subseteq\bbR^{n,1}$ for its linear span, so that $F=\bbH^n\cap V_F$.
In particular, if $\dim F=m$ then $\dim V_F=m+1$.
For a normally framed plane $C\in\calC$, we write $V_C$ for the linear span of its underlying plane $p(C)$.
Suppose now that $F\subset\bbH^n$ is a proper totally geodesic subspace projecting to a properly immersed geodesic submanifold and define
\begin{equation}\label{eqn:VFdefn}
\calV_F=\{C\in\overline{\calC}\mid C\subseteq F\},
\end{equation}
where $C\subseteq F$ is equivalently the linear containment $V_C\subseteq V_F$.
Then
\begin{equation}\label{eqn:incidenceavoidance}
\overline{\eta}_i^{\calC}(\calV_F)=0,
\end{equation}
for all sufficiently large $i$.
Unfortunately, this property need not pass to a weak-$*$ limit so we will need an additional argument to make use of maximality.

Finally, write $\eta_\infty$ and $\overline{\eta}_\infty$ for the restrictions of the limiting currents to $\calC^*$ and $\overline{\calC}^*$.
In Section~\ref{sec:incidencenonconcentration}, we will show that $\mu_\infty$ gives zero mass to frames based on the boundary of the convex core.
Since $\mu_\infty$ gives positive mass to frames based in the core, this will show that the restricted currents are nonzero.
We will also show that $\eta_\infty$-almost every plane crosses two distinct boundary lifts.
We use these properties to identify $\eta_\infty$ with a constant multiple of $\lambda$ in Section~\ref{sec:higherdimensionabsolute} for $m\ge3$ and in Section~\ref{sec:surfaceabsolute} for $m=2$.

\section{Absolute continuity for maximal submanifolds of dimension at least three}\label{sec:higherdimensionabsolute}

Throughout this section, we retain the notation of Section~\ref{sec:limitingmeasure} and assume that $m\ge3$, so in particular $n\ge4$.
The crossing coordinates, the Haar lift, and the compact nonconcentration argument below also apply to surfaces.
We will return to these constructions in Section~\ref{sec:surfaceabsolute}, where recurrence and entropy provide the additional input needed when $m=2$.
Our goal is to prove the following measure rigidity statement by studying intersections with the geodesic boundary and applying Ratner's theorem.

\begin{thm}\label{thm:higherdimensionhaar}
The current $\eta_\infty$ is nonzero and there exists a positive constant $c$ such that $\lambda=c\eta_\infty$.
\end{thm}

We quickly sketch the proof before giving the details.
First, fix $Z$ to be a lift of a geodesic boundary component and consider the intersections of our planes with this lift.
Along with the intersection, we record the angles between the chosen normal directions to the plane and the normal to the boundary lift.
These are the angle parameters that we introduce below.
We then transfer the current to the compact frame bundle over the boundary and disintegrate over this choice of angles.
The resulting conditional measures are finite, so Ratner's theorem shows that their ergodic components are homogeneous.
Some of these components may nevertheless be supported on smaller geodesic submanifolds of the boundary.
In the finite-volume setting, the corresponding step uses the support-containment theorem of Mozes and Shah \cite[Thm~1.1]{MozesShah} to rule out a proper homogeneous limit by maximality, as in \cite[Prop~3.1 and Thm~3.3]{BFMS}.
Here the intersections of our original submanifolds with the boundary need not themselves be maximal so this argument does not apply directly.

The way around this is to compare these disintegrations at two distinct boundary lifts.
If a smaller boundary submanifold contributes positive mass, then the current gives positive mass to a proper algebraic subset of the plane space.
Choose such a subset $\calA$ of least dimension.
Minimality lets us show that $\calA$ is preserved by the intrinsic isometry groups of the two boundary subspaces contributing positive mass.
Together, these groups force $\calA$ to be the family of all $m$-planes contained in a single proper geodesic subspace $F$.
The restricted current is also locally invariant under these groups, which makes it a multiple of intrinsic Haar measure on this family.

To bring maximality back into the argument, we still need $F$ to give a properly immersed submanifold of the original manifold.
Here local finiteness forces the translates of the subspace to be discrete and the orbit closure theorem of Lee and Oh then gives a proper immersion.
We are therefore left with concentration on the planes in a proper closed geodesic submanifold.
Maximality excludes this for the original sequence and another blow-up argument shows that it cannot arise in the limit either.
This time, the blow-up records the normal directions along which nearby planes approach the family, and recurrence at two boundary intersections rules out positive mass concentrating on those directions.
This contradiction rules out the smaller boundary submanifolds, hence the conditional measures on the frame bundles over the boundary components are Haar.
Comparing at two boundary lifts once more then gives the required multiple of Haar measure on $\calC^*$.
The argument is considerably shorter in codimension one, that is, when the geodesic submanifolds are genuine hypersurfaces.
We will explain these simplifications in Appendix~\ref{sec:hypersurfaceappendix}.

\subsection{Transferring the transverse current to the frame bundle over the boundary}\label{sec:crossinghyperplanes}

In this subsection, we transfer the transverse current to a space where we may apply Ratner's theorem.
The useful feature of the frame bundle over the boundary is that the transferred measure will be finite, even though our original measures are not.
To this end, let $Z$ be a lift of a component of the geodesic boundary with unit outward normal $n_Z$.
Throughout we use
$$H_Z=\SO_0(V_Z)\cong\SO_0(n-1,1),\qquad \Delta_Z=\Gamma\cap H_Z,$$
so that the corresponding frame bundle $\Delta_Z\backslash H_Z$ is compact.
We let $U_Z\subset\calC^*$ denote the set of normally framed planes crossing $Z$ and describe these planes by their angle parameters.
More precisely, if $(n_1,\ldots,n_{n-m})$ is the normal frame, define
$$a=(\langle n_Z,n_1\rangle,\ldots,\langle n_Z,n_{n-m}\rangle)\in\bbB^{n-m},$$
where $\bbB^{n-m}$ is the open unit ball.
A plane with this normal frame crosses $Z$ if and only if $|a|<1$.
Fix an $m$-dimensional indefinite linear subspace $V_0\subset V_Z$ and let $J=\SO_0(V_0)\cong\SO_0(m-1,1)$.
We choose an adapted frame identifying $H_Z$ with the model boundary group and $J$ with its standard $\SO_0(m-1,1)$ subgroup.
Choose an orthonormal basis $\{f_1,\ldots,f_{n-m}\}$ for the positive-definite orthogonal complement $D_0=V_Z\cap V_0^\perp$ and put
$$V_*=V_0\oplus\bbR n_Z,\qquad C_*=\bbH^n\cap V_*.$$
Then $V_Z=V_0\oplus D_0$ and $V_*^\perp=D_0$.

For an angle parameter $a=(a_1,\ldots,a_{n-m})\in\bbB^{n-m}$, write $\bfa=\sum_i a_i f_i\in D_0$ and $s_a=\sqrt{1-|a|^2}$.
Let $R_a$ be the orientation-preserving rotation of $D_0\oplus\bbR n_Z$ that fixes the vectors in $D_0$ perpendicular to $\bfa$ and sends $n_Z$ to $w_a=s_a n_Z-\bfa$.
Extend $R_a$ by the identity on $V_0$.
This choice varies smoothly with $a$, has $R_0=\Id$, and commutes with $J=\SO_0(V_0)$.
Define
\begin{equation}\label{eqn:referenceframe}
n_i(a)=R_af_i,\qquad V_{C_a}=R_aV_*=V_0\oplus\bbR w_a,\qquad C_a=\bbH^n\cap V_{C_a}.
\end{equation}
We give $C_a$ the ordered orthonormal normal frame $(n_1(a),\ldots,n_{n-m}(a))$, for which $\langle n_i(a),n_Z\rangle=a_i$ and $V_{C_a}\cap V_Z=V_0$.
Figure~\ref{fig:surface-rotation} illustrates this construction when $m=2$ and $n=3$, although we assume that $m\ge 3$ throughout this section unless explicitly stated otherwise.

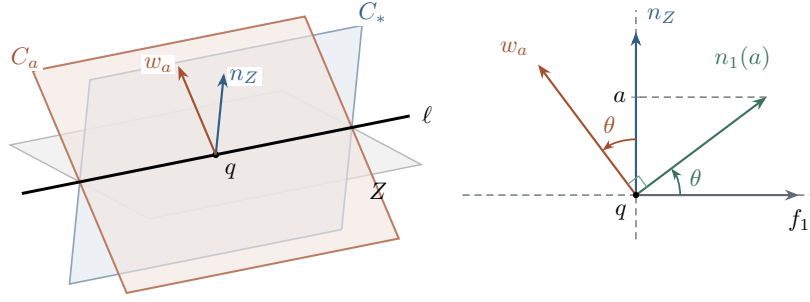
\begin{figure}[t]
\centering
\begingroup
\definecolor{rotationboundary}{HTML}{315E86}
\definecolor{rotationtilt}{HTML}{A45132}
\definecolor{rotationnormal}{HTML}{3D735C}
\definecolor{rotationmuted}{HTML}{626970}
\tikzset{every node/.style={font=\small},
  vec/.style={-{Stealth[length=2.1mm]},line width=.9pt},
  geodesic/.style={line width=1.25pt},
  guide/.style={densely dashed,line width=.55pt,color=rotationmuted}}
\scalebox{0.9}{\begin{tikzpicture}[baseline=(current bounding box.north),x=1cm,y=1cm]
  \coordinate (e) at (2.0,.4);
  \coordinate (f) at (1.02,-.544);
  \coordinate (n) at (.15,1.50);
  \coordinate (w) at (-.69,1.632);
  \filldraw[fill=black!5,draw=black!28,line width=.6pt]
    ($(e)+(f)$)--($(e)-(f)$)--($-1*(e)-(f)$)--($-1*(e)+(f)$)--cycle;
  \filldraw[fill=rotationboundary!11,draw=rotationboundary!55,line width=.7pt,fill opacity=.75]
    ($(e)+(n)$)--($(e)-(n)$)--($-1*(e)-(n)$)--($-1*(e)+(n)$)--cycle;
  \filldraw[fill=rotationtilt!15,draw=rotationtilt!75,line width=.85pt,fill opacity=.60]
    ($(e)+(w)$)--($(e)-(w)$)--($-1*(e)-(w)$)--($-1*(e)+(w)$)--cycle;
  \draw[geodesic] (-2.85,-.57)--(2.85,.57);
  \node[anchor=west] at (2.9,.59) {$\ell$};
  \fill (0,0) circle (1.4pt);
  \node[below right] at (0,0) {$q$};
  \node[rotationboundary,anchor=south west,fill=white,inner sep=1.2pt] at (2.05,1.89) {$C_*$};
  \node[rotationtilt,anchor=south east,fill=white,inner sep=1.2pt] at (-2.54,1.23) {$C_a$};
  \node[anchor=north west] at (2.12,-.27) {$Z$};
  \draw[vec,rotationboundary] (0,0)--(.12,1.20);
  \node[rotationboundary,anchor=west,fill=white,inner sep=1pt] at (.16,1.15) {$n_Z$};
  \draw[vec,rotationtilt] (0,0)--(-.552,1.306);
  \node[rotationtilt,anchor=east,fill=white,inner sep=1pt] at (-.59,1.32) {$w_a$};
\end{tikzpicture}
\hspace{.05cm}
\begin{tikzpicture}[baseline=(current bounding box.north),x=1cm,y=1cm]
  \draw[guide] (-2.55,0)--(2.55,0);
  \draw[guide] (0,-.65)--(0,2.72);
  \draw[vec,rotationmuted] (0,0)--(2.4,0);
  \node[anchor=north] at (2.40,-.08) {$f_1$};
  \draw[vec,rotationboundary] (0,0)--(0,2.4);
  \node[rotationboundary,anchor=south west] at (.03,2.4) {$n_Z$};
  \draw[vec,rotationtilt] (0,0)--(-1.44,1.92);
  \node[rotationtilt,anchor=south east] at (-1.44,1.92) {$w_a$};
  \draw[vec,rotationnormal] (0,0)--(1.92,1.44);
  \node[rotationnormal,anchor=south west] at (1.02,1.70) {$n_1(a)$};
  \draw[guide] (1.92,1.44)--(0,1.44);
  \draw[rotationmuted] (-.07,1.44)--(.07,1.44);
  \node[anchor=east,fill=white,inner sep=1pt] at (-.10,1.44) {$a$};
  \draw[-{Stealth[length=1.7mm]},rotationtilt,line width=.75pt] (0,.82) arc[start angle=90,end angle=126.8699,radius=.82];
  \node[rotationtilt] at (-.39,1.03) {$\theta$};
  \draw[-{Stealth[length=1.7mm]},rotationnormal,line width=.75pt] (.65,0) arc[start angle=0,end angle=36.8699,radius=.65];
  \node[rotationnormal] at (.88,.25) {$\theta$};
  \draw[rotationnormal!70,line width=.5pt] (-.12,.16)--(.04,.28)--(.16,.12);
  \fill (0,0) circle (1.3pt);
  \node[below left] at (0,0) {$q$};
\end{tikzpicture}
}\endgroup
\caption{The reference surface $C_*$ and the rotated surface $C_a$ about $\ell=C_*\cap Z$ when $n=3$, with $a=\sin\theta$ and $0<\theta<\pi/2$.
On the right, at $q\in\ell$, the rotation $R_a$ in the positive-definite plane perpendicular to $\ell$ sends $n_Z$ to $w_a$ and $f_1$ to $n_1(a)$.}
\label{fig:surface-rotation}
\end{figure}

There is an $H_Z$-equivariant identification
\begin{align}
\label{eqn:equivariant} H_Z/J\times\bbB^{n-m}&\to U_Z,\\
 (hJ,a)&\mapsto hC_a.\nonumber
\end{align}
In these coordinates, every normally framed plane $C$ crossing $Z$ has the form $C=hC_a$ and ordered normal frame $(hn_1(a),\ldots,hn_{n-m}(a))$.
The angle parameter $a$ is determined by $C$, while $h$ is determined up to right multiplication by $J$.
Choosing $h$ records an oriented orthonormal tangent frame based at a point of $C\cap Z$, and $hR_a$ carries the reference frame of $C_*$ to the corresponding adapted frame of $C$.
When $m=2$, we have $J=A$, and right multiplication by $a_t$ moves the basepoint along the geodesic $C\cap Z$.

We will also use these adapted frames to express tangent vectors in a fixed reference space.
For a normally framed plane $C$, we write $T_C\overline{\calC}$ for $T_{p(C)}\overline{\calC}$, so that
$$T_C\overline{\calC}=\Hom(V_C,V_C^\perp).$$
Given $C=hC_a$ and $\Theta\in T_C\overline{\calC}$, pulling $\Theta$ back to the reference plane gives
\begin{equation}\label{eqn:pullback}
T=(hR_a)^{-1}\Theta(hR_a)\vert_{V_*}\in T_{C_*}\overline{\calC}=\Hom(V_*,V_*^\perp).
\end{equation}
For $g\in J$, replacing $h$ by $hg$ changes this pullback to $g^{-1}Tg=T\circ g$, since $R_a$ commutes with $J$ and $J$ fixes $V_*^\perp$ pointwise.
Thus on the projective tangent bundle, the action is given in coordinates by
\begin{equation}\label{eqn:Jinvariant}
(h,a,[T])\cdot g=(hg,a,[T\circ g]).
\end{equation}

For a compact set $I\subset\bbB^{n-m}$, let $U_{Z,I}$ consist of the planes in $U_Z$ whose angle parameter belongs to $I$.
We then transfer the current to a finite measure on $(\Delta_Z\backslash H_Z)\times I$.
This is the usual correspondence between transverse measures and measures obtained by adjoining Haar measure along the fibers.
We include the construction because we will use it several times and because compactness of the frame bundle over the boundary is what makes the resulting measure finite.

\begin{lemma}\label{lem:banana}
For every $2\le m<n$, the restriction to $U_{Z,I}$ of a $\Gamma$-invariant Radon current on $\calC^*$ has a canonical Haar lift to $H_Z\times I$.
This descends to a finite right $J$-invariant measure $\mu_I$ on $(\Delta_Z\backslash H_Z)\times I$.
The conditional probability measures $\mu_a$ obtained by disintegrating over $I$ are $J$-invariant for almost every $a$.
The construction is positive and linear in the current.
\end{lemma}

\begin{proof}
Let $\eta_{U_{Z,I}}$ denote the given current.
The construction uses the following surjective maps
$$\xymatrix@C=3em@R=2em{
& H_Z\times I \ar[dl]_-{\pi_{\Delta_Z}} \ar[dr]^-{\alpha} & \\
\makebox[8em][c]{$(\Delta_Z\backslash H_Z)\times I$}
&&
\makebox[8em][c]{$U_{Z,I},$}
}$$
where $\pi_{\Delta_Z}(h,a)=(\Delta_Zh,a)$ and $\alpha(h,a)=hC_a$.
We first lift the current along the right-hand map and then descend the resulting measure along the left-hand map.

More precisely, for any $\varphi\in C_c(H_Z\times I)$, define the Haar lift $\widetilde{\eta}_{U_{Z,I}}$ by
$$\int_{H_Z\times I}\varphi\,d\widetilde{\eta}_{U_{Z,I}}
=\int_{U_{Z,I}}\int_J\varphi(hj,a)\,dm_J(j)\,d\eta_{U_{Z,I}}(hC_a),$$
where $(h,a)$ is a local representative of the plane $hC_a$ and $m_J$ is Haar measure on $J$.
Changing the representative amounts to left translation in the inner integral hence the formula is independent of this choice.
In particular, in local sections this is simply the product of the transverse current and Haar measure along the $J$-fibers.
The resulting measure is Radon and left $\Delta_Z$-invariant, and unimodularity of $J$ gives right $J$-invariance as well.

We then descend this measure to the quotient by choosing a measurable fundamental domain $\frakD$ for the left $\Delta_Z$-action on $H_Z\times I$ and defining
$$\mu_I(A)=\widetilde{\eta}_{U_{Z,I}}\bigl(\pi_{\Delta_Z}^{-1}(A)\cap \frakD\bigr),$$
for Borel sets $A\subseteq(\Delta_Z\backslash H_Z)\times I$.
The left invariance makes this independent of the choice of fundamental domain, while the commuting right $J$-action preserves $\mu_I$.
Moreover, the measure is locally finite in quotient charts and, since $\Delta_Z\backslash H_Z$ and $I$ are compact, finitely many such charts suffice to show that $\mu_I$ is finite.

Finally, let $\sigma_I=(\pr_I)_*\mu_I$ be the measure recording the angle parameters.
Disintegrating over projection to the second factor gives
$$\mu_I=\int_I\mu_a\otimes\delta_a\,d\sigma_I(a),$$
where $\mu_a$ is a probability measure on $\Delta_Z\backslash H_Z$ for $\sigma_I$-almost every $a$.
Thus, for any continuous function $f$ on $(\Delta_Z\backslash H_Z)\times I$, we obtain
$$\int f\,d\mu_I=\int_I\left(\int_{\Delta_Z\backslash H_Z}f(x,a)\,d\mu_a(x)\right)d\sigma_I(a).$$
By the usual properties of disintegration, the conditional measures $\mu_a$ are $J$-invariant.
\end{proof}

\begin{rem}
The same construction applies on the full angle ball $\bbB^{n-m}$ and gives a possibly infinite Radon measure on $(\Delta_Z\backslash H_Z)\times\bbB^{n-m}$.
Since $\Delta_Z\backslash H_Z$ is compact, the angle marginal is locally finite, so we may disintegrate this measure directly into $J$-invariant conditional probability measures.
Thus the restriction to a compact set $I$ is not essential to the construction.
We make this restriction for the convenience of working with finite measures throughout, recovering the full current by compact exhaustion when needed.
\end{rem}

Since $m\ge3$, the group $J$ is generated by one-parameter unipotent subgroups, so Ratner's theorem \cite{Ratner2} shows that the $J$-ergodic components of $\mu_a$ are homogeneous.
Put $X_Z=\Delta_Z\backslash H_Z$.
Separating the Haar component gives the decomposition
\begin{equation}\label{eqn:ratnerdecomp}
\mu_a=c(a)m_{X_Z}+\mu_a^{\mathrm{Ex}},
\end{equation}
for $\sigma_I$-almost every $a$.
Here $m_{X_Z}$ is Haar probability measure, and the exceptional term is the ergodic integral
$$\mu_a^{\mathrm{Ex}}=\int_{\Omega_a^{\mathrm{Ex}}}m_\upsilon\,d\tau_a(\upsilon),$$
where each $m_\upsilon$ is a homogeneous $J$-ergodic probability on a proper closed orbit $\Delta_ZhL$ with $J\le L$, and $\tau_a$ has total mass $1-c(a)$.
We do not assume that $\tau_a$ is atomic.
The countable grouping of invariant measures in \cite[Thm~2.2]{MozesShah} likewise allows each grouped measure to have a continuous ergodic decomposition.
In particular, $0\le c(a)\le1$ and $\mu_a^{\mathrm{Ex}}(X_Z)=1-c(a)$ and our goal is to show that $c(a)=1$ for $\sigma_I$-almost every $a$.

To describe the support of these exceptional measures, we introduce the following notation.
For a totally geodesic subspace $B\subseteq Z$, let
$$\calE_{Z,B}=\{C\in\overline{\calC}\mid \dim(V_C\cap V_B)\ge m\},$$
the set of planes whose linear spans meet $V_B$ in dimension at least $m$.
For a plane $C$ crossing $Z$, membership in $\calE_{Z,B}$ is equivalent to $C\cap Z\subseteq B$.
If $B\ne Z$, then $\calE_{Z,B}$ is a proper algebraic subset of the relevant Grassmannian.
Using this notation, we show that $\eta_\infty\vert_{U_Z}$ decomposes similarly to Ratner's theorem and study the support of the corresponding exceptional components.

\begin{prop}\label{prop:boundaryhomogeneous}
Let $\eta$ be a $\Gamma$-invariant Radon current.
The decomposition in Equation~\eqref{eqn:ratnerdecomp} induces a decomposition
$$\eta\vert_{U_Z}=\eta_{Z,Z}+\eta_Z^{\mathrm{Ex}},\qquad
\eta_Z^{\mathrm{Ex}}=\sum_{B\subset Z}\eta_{Z,B},$$
where $B$ ranges over the countable collection of lifts of proper closed geodesic submanifolds of the boundary component covered by $Z$ that have dimension at least $m-1$.
Each $\eta_{Z,B}$ is locally $\SO_0(V_B)$-invariant and satisfies
$$\eta_{Z,B}\bigl(U_Z\ssm p^{-1}(\calE_{Z,B})\bigr)=0.$$
After disintegrating over the angle parameters, its conditional measures on the $\SO_0(V_B)$-orbits are smooth homogeneous measures.
Finally, the current $\eta_{Z,Z}$ corresponds to $c(a)m_{X_Z}$ and is $H_Z$-invariant on $U_Z$.
\end{prop}

\begin{proof}
Fix an angle parameter $a$ for which the preceding decomposition is defined and let $m_{\Delta_ZhL}$ be a homogeneous component in its exceptional ergodic integral, with $J\le L$.
Recalling that $V_0$ was the linear subspace defining $J$, the subgroup description in \cite[Lem~3.2]{BFMS} then supplies a totally geodesic subspace $B_0\subset Z$ with $V_0\subseteq V_{B_0}$ for which the noncompact factor of $L$ is $\SO_0(V_{B_0})$ and the remaining factor is contained in its compact centralizer.
Writing $B=hB_0$, the corresponding geodesic submanifold in $\Delta_Z\backslash Z$ is closed and there are only countably many such submanifolds and lifts by the finite-volume case of \cite[Thm~1.7(2)]{LO}.

It remains to interpret these homogeneous measures as transverse currents.
Lifting each one back to $H_Z$ gives, locally, Haar measure on a coset of $L$, which is invariant under the left action of $\SO_0(V_B)$.
Moreover, the compact factor commutes with $\SO_0(V_B)$, so we may disintegrate over that factor as well and then pass to the quotient by the right $J$-action.
We integrate these transverse currents over the ergodic components and the angle parameters, grouping together those assigned to the same lift $B$.
All terms in a group have the same local $\SO_0(V_B)$-invariance so integration preserves this invariance and gives $\eta_{Z,B}$ with the claimed conditional measures.
The resulting sum is countable because the collection of subspaces $B$ is countable.
The angle and compact parameters may still vary continuously and these parameters remain inside the integrals defining each $\eta_{Z,B}$.
Moreover, for $\ell\in L$, $h\ell V_{C_a}\cap V_Z=h\ell V_0\subseteq V_B$, giving
$$\eta_{Z,B}\bigl(U_{Z,I}\ssm p^{-1}(\calE_{Z,B})\bigr)=0,$$
as desired.
Taking a compact exhaustion of $\bbB^{n-m}$ by compact balls gives the decomposition on all of $U_Z$, which completes the proof.
\end{proof}

\subsection{Nonconcentration on closed geodesic submanifolds}\label{sec:incidencenonconcentration}

Recall from Section~\ref{sec:limitingmeasure} that maximality implies that the original currents give zero mass to $\calV_F$ from Equation~\eqref{eqn:VFdefn}, the set of planes contained in a fixed proper closed geodesic submanifold.
What we need, however, is the corresponding statement for the limiting current.
As in Lemma~\ref{lem:noprob}, we will obtain this by blowing up the set on which concentration might occur and studying the resulting projective action.
Here the normal space to $\calV_F$ at a plane $C\subseteq F$ is given by $N_C\calV_F=\Hom(V_C,V_F^\perp)$.
The following well-known consequence of Poincar\'e recurrence will allow us to control invariant measures on the projectivization of this normal space.
We add a proof for completeness.

\begin{lemma}\label{lem:normalrecurrence}
Let $V=V_E\oplus\bbR t$, where $V_E$ is indefinite of dimension $m\ge2$, $\langle t,t\rangle>0$, and let $V_D$ be positive definite.
If $m\ge3$ and $\mu$ is an $\SO_0(V_E)$-invariant probability measure on $\bbP\Hom(V,V_D)$, then
$$\mu\bigl(\{[T]\in\bbP\Hom(V,V_D)\mid T\vert_{V_E}=0\}\bigr)=1.$$
For every $m\ge2$, the action of $\SO_0(V)$ on the entire domain admits no invariant projective probability measure.
\end{lemma}

\begin{proof}
By Poincar\'e recurrence, any invariant probability measure is supported on the projectivized fixed space of each unipotent one-parameter subgroup.
For $m\ge3$, the group $\SO_0(V_E)$ is generated by unipotents, so $\mu$ is supported on their common fixed space in $\Hom(V,V_D)$, which is precisely $\{T\mid T\vert_{V_E}=0\}$.
Since $\SO_0(V)$ is generated by unipotents for every $m\ge2$ and its only fixed map is zero, there can be no invariant probability measure on the projective space.
\end{proof}

We use this to generalize the argument of non-atomicity in Lemma~\ref{lem:noprob} to nonconcentration on the families $\calV_F$.

\begin{prop}\label{prop:incidencenonconcentration}
Suppose that $F\subset\bbH^n$ is a totally geodesic subspace of dimension at least $m$ that projects to a properly immersed closed geodesic submanifold.
Let $\overline{\eta}_i^{\calC}\to\overline{\eta}^{\calC}$ be $\Gamma$-invariant Radon currents satisfying $\overline{\eta}_i^{\calC}(\calV_F)=0$.
Then for any two distinct boundary lifts $B,B'$ of the geodesic boundary of the intrinsic core of $F$, we have
$$\overline{\eta}^{\calC}\bigl(\{C\in\calV_F\mid C\text{ crosses both }B\text{ and }B'\}\bigr)=0.$$
If the quotient of $F$ is compact, then $\overline{\eta}^{\calC}(\calV_F)=0$.
This conclusion in the compact case also holds for $m=2$.
\end{prop}

\begin{proof}
We blow up $\overline{\calC}$ along the smooth submanifold $\calV_F$ to obtain
$$\beta:\Bl_{\calV_F}(\overline{\calC})\to\overline{\calC}.$$
This map is proper and the exceptional divisor $E_F=\beta^{-1}(\calV_F)$ has fiber $\bbP\Hom(V_C,V_F^\perp)$ at $C\subseteq F$.
As in Lemma~\ref{lem:noprob}, each original current has a unique lift $\widehat{\eta}_i$ satisfying
$$\beta_*\widehat{\eta}_i=\overline{\eta}_i^{\calC},\qquad \widehat{\eta}_i(E_F)=0.$$
After passing to a subsequence, write $\widehat{\eta}$ for the weak-$*$ limit and set $\widehat{\eta}_{E_F}=\widehat{\eta}\vert_{E_F}$.
Properness of $\beta$ again gives that $\beta_*\widehat{\eta}=\overline{\eta}^{\calC}$ and $\beta_*\widehat{\eta}_{E_F}=\overline{\eta}^{\calC}\vert_{\calV_F}$.
Thus $\widehat{\eta}_{E_F}$ records precisely the possible limiting mass on $\calV_F$ together with the normal directions along which that mass approaches.
It is invariant under $\Gamma_F=\Gamma\cap\Stab_G(F)$ and $\beta$ is equivariant for the action of this group.

To use recurrence, we again work over a frame bundle of a boundary component.
Let $B$ be a lift of a boundary component of the convex core in $F$ and let $K_F$ be the compact group acting trivially on $B$ and by rotations on the normal directions.
Put $\Gamma_{F,B}=\Gamma\cap\Stab_G(F)\cap\Stab_G(B)$.
Then, passing to finite orientation covers if necessary, we have that $\Gamma_{F,B}$ embeds in $\SO_0(V_B)\times K_F$ and the bundle of adapted frames $\Gamma_{F,B}\backslash(\SO_0(V_B)\times K_F)$ is compact.

We use the reference planes and adapted frames of Section~\ref{sec:crossinghyperplanes} inside $F$, with $B$ as the boundary lift.
Thus $C_a\subset F$, $V_{C_a}\subset V_F$, and the angle parameter $a$ belongs to $\bbB^{\dim(F)-m}$.
We lift $\widehat{\eta}_{E_F}$ along the $J=\SO_0(m-1,1)$-fibers with Haar measure.
We extend the intrinsic adapted frames by the normal frame of $F$.
In the coordinates of Equation~\eqref{eqn:pullback}, a projective normal direction $[\Theta]$, with $\Theta:V_C\to V_F^\perp$, is represented by a nonzero map $T:V_*\to V_F^\perp$.
Equation~\eqref{eqn:Jinvariant} gives $T\mapsto T\circ j$ for $j\in J$, while leaving the normal frame fixed.
These coordinates identify the reference domains with $V_*$ so the resulting action on the projective space is independent of the angle parameter.

For a compact angle set $I$, denote the resulting measure on the quotient bundle by $\widehat{\eta}_{E_F,I}$ and let $q$ be the projection map to the projective space in the above coordinates.
If $\widehat{\eta}_{E_F,I}\ne0$, then the normalized pushforward
$$\zeta_I=\frac{q_*\widehat{\eta}_{E_F,I}}{\|\widehat{\eta}_{E_F,I}\|},$$
is a $J$-invariant probability measure, hence Lemma~\ref{lem:normalrecurrence} implies that $T\vert_{V_0}=0$.
Since $hR_aV_0=V_C\cap V_B$, this says that $\Theta\vert_{V_C\cap V_B}=0$ for $\widehat{\eta}_{E_F}$-almost every point over a plane crossing $B$.
There are only countably many boundary lifts, so for almost every $(C,[\Theta])$ the map $\Theta$ vanishes on $V_C\cap V_B$ for every lift $B$ crossed by $C$.
If $C$ crosses two distinct lifts $B,B'$, these intersections are distinct hyperplanes in $V_C$ and hence span $V_C$.
It follows that $\Theta=0$, contradicting the fact that $[\Theta]$ is represented by a nonzero map.

When the quotient of $F$ is compact, the same construction can be carried out on its entire adapted-frame bundle.
The resulting measure is again finite but its pushforward to the projective space is now invariant under the full group $\SO_0(V_C)$, identified with $W_m$ in adapted frames.
The last assertion of Lemma~\ref{lem:normalrecurrence} rules out such a measure directly, giving the compact case for every $m\ge2$.
\end{proof}

Applying this proposition to the currents satisfying Equation~\eqref{eqn:incidenceavoidance}, we obtain the two properties of the limiting current promised at the end of the preceding section.

\begin{cor}\label{cor:twocrossings}
The transverse current $\eta_\infty$ is nonzero and $\eta_\infty(\calC^*\ssm\calU_\partial)=0$, where $\calU_\partial$ is the set of normally framed planes crossing two distinct boundary lifts.
These assertions also hold for $m=2$ in every ambient dimension $n\ge3$.
\end{cor}

\begin{proof}
By the compact case of Proposition~\ref{prop:incidencenonconcentration}, $\overline{\eta}_\infty^{\calC}(\calV_F)=0$ for every lift $F$ of a compact closed geodesic submanifold.
There are only countably many such submanifolds and lifts by Corollary~\ref{cor:planeclosures}, so their union also has measure zero.
In particular, $\overline{\eta}_\infty^{\calC}$-almost no plane is itself contained in a boundary lift, while every other plane meets the boundary of the convex core in a set of $m$-dimensional volume zero.
Disintegrating $\mu_\infty$ over its transverse current therefore shows that $\mu_\infty$ gives zero mass to frames based on the boundary of the core.
Since $\mu_\infty(\calN_0)\ge1$, it follows that the frames based in the interior have positive mass, and hence $\eta_\infty$ is nonzero.

For the second assertion, Lemma~\ref{lem:aeboundarylift} shows that a plane meeting the interior of the core and crossing at most one boundary lift must lie in a lift of a compact closed geodesic submanifold.
These planes have measure zero by the first paragraph, which proves the assertion.
\end{proof}

\subsection{Intermediate boundary orbit closures}\label{sec:intermediateclosures}

We return to the boundary decomposition of $\eta_\infty$, with the goal of proving $\eta_Z^{\mathrm{Ex}}=0$.
The idea is to show that if $\eta_{Z,B}\ne0$ for some proper $B\subset Z$ then there would be a proper closed geodesic submanifold $F$ with $\overline{\eta}_\infty(\calV_F)>0$, contradicting the results of the previous subsection.
Throughout this subsection we consider the space of $m$-dimensional planes $\overline{\calC}$ and write $\overline{\eta}=\overline{\eta}_\infty$.

Recall that a totally geodesic $m$-plane $C\subset\bbH^n$ is determined by its linear span $V_C\subset\bbR^{n,1}$, with $C=\bbH^n\cap V_C$.
The correspondence $C\mapsto V_C$ identifies $\overline{\calC}$ with the open subset of $\Gr_{m+1}(\bbR^{n,1})$ consisting of subspaces on which the form has signature $(m,1)$.
We therefore regard algebraic subsets of $\overline{\calC}$ as intersections of this open subset with real algebraic subsets of the Grassmannian.
In particular, $\calV_F$ is defined by the algebraic incidence condition $V_C\subseteq V_F$.

Suppose that $\overline{\eta}$ gives positive measure to a proper algebraic subset of $\overline{\calC}$ and choose $d$ to be the least dimension of such a subset.
Two distinct irreducible subsets of dimension $d$ intersect in smaller dimension and hence in a set of measure zero.
It follows that there are only countably many irreducible subsets of dimension $d$ with positive measure.
We call these the \emph{minimal algebraic carriers} of $\overline{\eta}$.
Their union is $\Gamma$-invariant, so we apply the boundary decomposition to the restriction of $\overline{\eta}$ to this union and write $\overline{\eta}_{Z,B}$ for its components on the underlying space $\overline{\calC}$.
The point of this minimality assumption is that every positive-measure subset of a minimal algebraic carrier is Zariski dense in it.
This will allow us to extend algebraic conditions holding on a set of positive measure to the entire carrier.

Before proceeding, we record the Lie algebra calculation that will let us combine the invariances obtained at two boundary lifts.
If the linear spans $V_B,V_{B'}$ share a common vector $e$ for which $\langle e,e\rangle>0$, then
\begin{equation}\label{eqn:boundarygeneration}
\langle\frakso(V_B),\frakso(V_{B'})\rangle=\frakso(V_B+V_{B'}).
\end{equation}
This is a simple fact but it will be used several times in the sequel.
In particular, one of its uses will be to prove the following.

\begin{lemma}\label{lem:algebraiccarrier}
Every minimal algebraic carrier of $\overline{\eta}$ is $\calV_F$ for a proper totally geodesic subspace $F\subset\bbH^n$.
\end{lemma}

\begin{proof}
Let $\calA$ be a minimal algebraic carrier.
By Corollary~\ref{cor:twocrossings} and countability of the boundary lifts, we can choose distinct lifts $Z,Z'$ such that the planes in $\calA$ crossing both $Z$ and $Z'$ have positive measure.
Write $\overline{\eta}^{\calA}$ for the restriction of $\overline{\eta}$ to these planes.
We will then compare the two boundary decompositions on this common part of the current.

Let $\overline{\eta}_B$ and $\overline{\eta}'_{B'}$ be the restrictions of $\overline{\eta}_{Z,B}$ and $\overline{\eta}_{Z',B'}$ to this same set, so that
$$\overline{\eta}^{\calA}=\sum_{B\subseteq Z}\overline{\eta}_B=\sum_{B'\subseteq Z'}\overline{\eta}'_{B'}.$$
To isolate components that occur together, form the measure
$$d\overline{\eta}_{B,B'}=\frac{d\overline{\eta}_B}{d\overline{\eta}^{\calA}}\frac{d\overline{\eta}'_{B'}}{d\overline{\eta}^{\calA}}\,d\overline{\eta}^{\calA},$$
and note that
$$\sum_{B,B'}\overline{\eta}_{B,B'}=\overline{\eta}^{\calA}.$$
Therefore we may choose a pair $(B,B')$ with $\overline{\eta}_{B,B'}\ne0$.
Since $\overline{\eta}_{B,B'}\ll\overline{\eta}_B,\overline{\eta}'_{B'}$, both boundary descriptions hold on a common set of positive $\overline{\eta}$-measure in $\calA$.

We first show that both boundary groups preserve $\calA$.
By Proposition~\ref{prop:boundaryhomogeneous}, the conditionals of $\overline{\eta}_{Z,B}$ on the $\SO_0(V_B)$-orbits are smooth homogeneous measures.
An orbit is analytic and $\calA$ is algebraic, so their intersection has measure zero unless the entire connected orbit lies in $\calA$.
Our restricted current is carried by the countable union of the minimal algebraic carriers, whose pairwise intersections have measure zero.
It follows that almost every orbit contributing to the measure on $\calA$ lies entirely in $\calA$.
The vector fields from $\frakso(V_B)$ are therefore tangent to $\calA$ on a positive-measure subset of its smooth locus.
By minimality this subset is Zariski dense so the algebraic tangency condition holds throughout and $\SO_0(V_B)$ preserves $\calA$.
The same argument applies to $B'$, giving
$$\SO_0(V_B)\calA=\calA=\SO_0(V_{B'})\calA.$$

For $\overline{\eta}_{B,B'}$-almost every $C$, the support conditions in Proposition~\ref{prop:boundaryhomogeneous} give
$$V_C=(V_C\cap V_Z)+(V_C\cap V_{Z'})\subseteq V_B+V_{B'}=V_F,$$
where $F=\bbH^n\cap(V_B+V_{B'})$ (see Figure~\ref{fig:geodesic-carrier}).
Since $V_C\subseteq V_F$ is an algebraic condition holding on a positive-measure subset of $\calA$, minimality gives $\calA\subseteq\calV_F$.

\begin{figure}[t]
\centering
\begingroup
\definecolor{carrierfirst}{HTML}{315E86}
\definecolor{carriersecond}{HTML}{A45132}
\definecolor{carriersurface}{HTML}{3D735C}
\definecolor{carriermuted}{HTML}{626970}
\scalebox{0.81}{\begin{tikzpicture}[x={(3.7cm,-.35cm)},y={(1.15cm,.70cm)},z={(0cm,2.35cm)},
  every node/.style={font=\small}]
  \coordinate (Lbf) at (-.68,-.83,-.70);
  \coordinate (Lbb) at (-.68,.83,-.70);
  \coordinate (Ltb) at (-.68,.83,.70);
  \coordinate (Ltf) at (-.68,-.83,.70);
  \coordinate (Rbf) at (.68,-.83,-.70);
  \coordinate (Rbb) at (.68,.83,-.70);
  \coordinate (Rtb) at (.68,.83,.70);
  \coordinate (Rtf) at (.68,-.83,.70);
  \draw[black!15,densely dashed] (-1,-1,-.85)--(1,-1,-.85)--(1,1,-.85)--(-1,1,-.85)--cycle;
  \draw[black!15,densely dashed] (-1,-1,.85)--(1,-1,.85)--(1,1,.85)--(-1,1,.85)--cycle;
  \foreach \x in {-1,1}\foreach \y in {-1,1}
    \draw[black!15,densely dashed] (\x,\y,-.85)--(\x,\y,.85);
  \filldraw[fill=carrierfirst!13,draw=carrierfirst!60,line width=.65pt,fill opacity=.65]
    (Lbf)--(Lbb)--(Ltb)--(Ltf)--cycle;
  \filldraw[fill=carriersecond!13,draw=carriersecond!65,line width=.65pt,fill opacity=.65]
    (Rbf)--(Rbb)--(Rtb)--(Rtf)--cycle;
  \filldraw[fill=carriersurface!15,draw=carriersurface!65,line width=.8pt,fill opacity=.65]
    (-.96,-.64,0)--(.96,-.64,0)--(.96,.64,0)--(-.96,.64,0)--cycle;
  \draw[carrierfirst,line width=1.6pt] (-.68,-.64,0)--(-.68,.64,0);
  \draw[carriersecond,line width=1.6pt] (.68,-.64,0)--(.68,.64,0);
  \node[carrierfirst,anchor=south] at (-.68,0,.78) {$B\subset Z$};
  \node[carriersecond,anchor=south] at (.68,0,.78) {$B'\subset Z'$};
  \node[carriersurface,fill=white,inner sep=1.5pt] at (.08,0,0) {$C$};
  \node[anchor=west,font=\large] at (1.03,1,.66) {$F$};
  \draw[carrierfirst!70,line width=.6pt] (-.68,-.24,0)--(-.78,-.65,-.57);
  \node[carrierfirst,anchor=north,fill=white,inner sep=2pt] at (-.78,-.65,-.57) {$C\cap Z\subset B$};
  \draw[carriersecond!70,line width=.6pt] (.68,-.24,0)--(.77,-.65,-.57);
  \node[carriersecond,anchor=north,fill=white,inner sep=2pt] at (.77,-.65,-.57) {$C\cap Z'\subset B'$};
\end{tikzpicture}
}\endgroup
\caption{The boundary support conditions force $C\subset F$, where $V_F=V_B+V_{B'}$.
The picture shows a three-dimensional schematic.}
\label{fig:geodesic-carrier}
\end{figure}
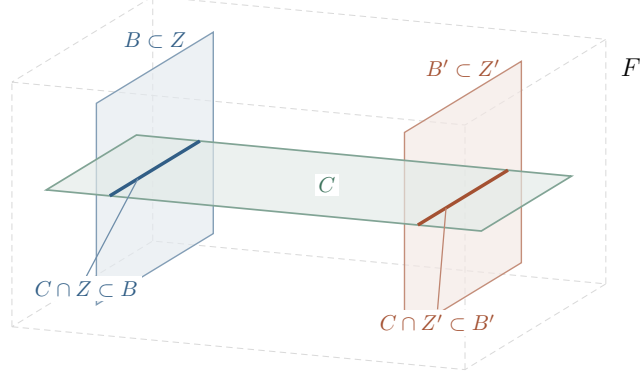

It remains to show that this inclusion is an equality.
The space $V_Z\cap V_{Z'}$ is positive definite and for any of the contributing planes $C$ the $(m-1)$-dimensional space
$$V_C\cap V_Z\cap V_{Z'}\subseteq V_B\cap V_{B'},$$
therefore contains a vector of positive norm.
Equation~\eqref{eqn:boundarygeneration} now shows that the two boundary groups generate $\SO_0(V_F)$, so this group preserves $\calA$.
Its transitivity on $\calV_F$ gives $\calA=\calV_F$.
Since $\calA$ is proper, $F$ is also proper, proving the lemma.
\end{proof}

At this point we have found the appropriate subspace but to apply the preceding subsection we still need to verify that $F$ projects to a properly immersed submanifold.
This additional information will come from the measures.
In particular, the two boundary groups preserving $\calV_F$ must also locally preserve the restricted current.
We use this to identify the measure on $\calV_F$, after which local finiteness will force the translates of $F$ to be discrete.

\begin{prop}\label{prop:closedcarrier}
If $\calV_F$ is a minimal algebraic carrier of positive $\overline{\eta}$-mass, then $F$ projects to a properly immersed closed geodesic submanifold with nonempty Fuchsian ends.
Moreover, there exists a positive constant $b$ such that
$$\overline{\eta}\vert_{\calV_F}=b\lambda_F\vert_{\calV_F\cap\overline{\calC}^*},$$
where $\lambda_F$ is the intrinsic smooth invariant measure on the space of $m$-planes in $F$.
\end{prop}

\begin{proof}
We first identify $\overline{\eta}\vert_{\calV_F}$.
Choose $B\subseteq Z$ with $\overline{\eta}_{Z,B}(\calV_F)>0$.
The proof of Lemma~\ref{lem:algebraiccarrier} shows that $\SO_0(V_B)$ preserves $\calV_F$ and hence also the linear span $V_F$ of its planes.
Its action on $V_B$ is the standard one, while $V_B^\perp$ is positive definite and fixed pointwise.
Thus every invariant indefinite linear subspace contains $V_B$, giving $V_B\subseteq V_F\cap V_Z$.
Conversely, the support condition $V_C\cap V_Z\subseteq V_B$ holds on a Zariski dense subset of $\calV_F$ and hence on every plane in $\calV_F$ crossing $Z$.
These intersections span $V_F\cap V_Z$, so
$$V_B=V_F\cap V_Z,\qquad B=F\cap Z,$$
as one expects.
Thus every boundary component contributing to $\overline{\eta}\vert_{\calV_F}$ over $Z$ comes from $F\cap Z$.
The homogeneous disintegration therefore gives local $\SO_0(V_F\cap V_Z)$-invariance on the set of planes crossing $Z$.
We may restrict the current to $\calV_F$ here because almost every contributing orbit lies in $\calV_F$, as in the preceding proof.

If a plane crosses two distinct lifts $Z,Z'$, we obtain both boundary invariances near that plane.
By Equation~\eqref{eqn:boundarygeneration}, the two Lie algebras generate $\frakso(V_F)$.
Conjugates of the two boundary Lie algebras by small products of elements of the boundary subgroups span $\frakso(V_F)$.
The inverse function theorem applied to the corresponding product map therefore gives a neighborhood of the identity in $\SO_0(V_F)$.
We may choose these products and shrink the neighborhood of the plane so that all intermediate images still cross both lifts and hence each factor preserves the restricted current.
It follows that the current is locally $\SO_0(V_F)$-invariant and, by transitivity and local uniqueness of invariant measure, locally a constant multiple of $\lambda_F$.

We next show that these local constants agree.
For a fixed pair $Z,Z'$, parametrize the planes in $F$ crossing both lifts by a point on each intersection with $F$ and an $m$-plane containing their joining geodesic.
The base and the fiber of this parameter space are connected, the latter being an unoriented Grassmannian, so its image is connected as well.
Pairs sharing a boundary lift give overlapping sets of planes, since three boundary points lie in a geodesic subspace of dimension at most two and can therefore be included in an $m$-plane.
The local constants consequently agree on the union and their common value $b$ is positive because this union carries all of $\overline{\eta}\vert_{\calV_F}$ by Corollary~\ref{cor:twocrossings}.

To obtain the stated equality of measures, we also need the complement of this union to be $\lambda_F$-null.
By Lemma~\ref{lem:aeboundarylift}, every remaining plane lies in a lift of one of countably many compact closed geodesic submanifolds.
The planes that such a lift contains in $F$ form a proper algebraic subset of $\calV_F$ unless the lift contains $F$ itself.
The latter cannot happen, since $F$ crosses an ambient boundary lift whereas a compact geodesic submanifold is contained in the core.
Thus the complement is null for both measures, and we have
$$\overline{\eta}\vert_{\calV_F}=b\lambda_F\vert_{\calV_F\cap\overline{\calC}^*}.$$

We now use this equality to show that the translates of $F$ are discrete.
Normalize the intrinsic measures equivariantly under $\Gamma$ so that the same constant $b$ occurs on every translate $\gamma F$.
The measures on distinct translates are mutually singular, since a plane common to two translates lies in their intersection, which has smaller dimension.
Suppose that distinct translates $F_j$ converge to a totally geodesic subspace $F_0$ meeting the core interior.

Choose a nonnegative $f\in C_c(\overline{\calC}^*)$ supported near an $m$-plane in $F_0$ and satisfying $\int f\,d\lambda_{F_0}>0$.
The intrinsic measures vary smoothly with the subspace so $\int f\,d\lambda_{F_j}\ge\epsilon>0$ for all sufficiently large $j$.
Since the restrictions to distinct translates are mutually singular, for every $N\ge j_0$ we have
$$
\int f\,d\overline{\eta}
\ge \sum_{j=j_0}^{N}\int_{\calV_{F_j}}f\,d\overline{\eta}
= b\sum_{j=j_0}^{N}\int f\,d\lambda_{F_j}
\ge b\epsilon(N-j_0+1),
$$
where $j_0$ is chosen so that the preceding lower bound holds for all $j\ge j_0$.
Letting $N\to\infty$ contradicts local finiteness.
Thus $\Gamma F$ has no accumulation among the subspaces meeting the core interior.

Applying Corollary~\ref{cor:planeclosures}, the plane-space form of \cite[Cor~1.9(1)]{LO}, we find that $F$ is contained in a lift of a properly immersed closed geodesic submanifold.
If that lift were larger than $F$, the closure of $\Gamma F$ would contain a positive-dimensional family of planes of dimension $\dim( F)$ meeting the core interior, contrary to the discreteness just proved.
Hence $F$ itself projects to a properly immersed closed geodesic submanifold.
Its core is compact with totally geodesic boundary by \cite[Prop~4.7]{LO}, and it has nonempty ends because it crosses an ambient boundary lift.
\end{proof}

\begin{cor}\label{cor:algebraicnull}
For every proper algebraic subset $\calA\subset\overline{\calC}$, we have $\overline{\eta}_\infty(\calA)=0$.
\end{cor}

\begin{proof}
Indeed, if not then Lemma~\ref{lem:algebraiccarrier} and Proposition~\ref{prop:closedcarrier} give a proper subspace $F$ projecting to a properly immersed closed geodesic submanifold with $\overline{\eta}(\calV_F)>0$.
By maximality, the original currents satisfy Equation~\eqref{eqn:incidenceavoidance} hence Proposition~\ref{prop:incidencenonconcentration} gives measure zero to the planes in $F$ crossing two boundary lifts of its convex core.
These boundary lifts are precisely the intersections $F\cap Z$ with ambient boundary lifts $Z$ crossed by $F$.
However, almost every plane in $\calV_F$ crosses two of them by Corollary~\ref{cor:twocrossings}, contradicting $\overline{\eta}(\calV_F)>0$.
\end{proof}

\subsection{Absolute continuity}\label{sec:absolutecontinuityhigher}

Now that we have the boundary coordinates and the local invariance argument in place, we prove that the $\Gamma$-action on $(\calC^*,\lambda)$ is ergodic.
We will also need this in the surface argument and in Section~\ref{sec:measureclasscomparison}.

\begin{prop}\label{prop:lambdaergodicity}
For every $2\le m<n$, the $\Gamma$-action on $(\calC^*,\lambda)$ is ergodic.
\end{prop}

\begin{proof}
By Lemma~\ref{lem:aeboundarylift}, the set
$$\calU_\partial=\bigcup_{Z\ne Z'}(U_Z\cap U_{Z'}),$$
is $\lambda$-conull in $\calC^*$.
Let $f$ be a bounded $\Gamma$-invariant measurable function.
We will show that $f$ is essentially constant on this union by comparing its descriptions at two boundary lifts.

Fix $Z$ and use the coordinates $H_Z/J\times\bbB^{n-m}$ from Equation~\eqref{eqn:equivariant}.
The smooth measure class is the product of Haar measure on $H_Z/J$ and Lebesgue measure on the ball of angle parameters.
Since $J=\SO_0(m-1,1)$ is noncompact for $m\ge2$, Howe--Moore \cite{HoweMoore} gives ergodicity of its action on the finite-volume quotient $\Delta_Z\backslash H_Z$.
Equivalently, the $\Delta_Z$-action on $H_Z/J$ is ergodic.
Disintegrating over the angle parameter, we therefore obtain that $f(hC_a)=f_Z(a)$ almost everywhere on $U_Z$ for a measurable function $f_Z$.
In particular, $f$ is $H_Z$-invariant there.

Now let $Z'$ be a second boundary lift and repeat the argument.
Writing $a,a'$ for the two angle parameters, we have
$$f(C)=f_Z(a)=f_{Z'}(a'),$$
for almost every $C\in U_Z\cap U_{Z'}$.
Thus $f$ is invariant under both boundary groups on this intersection.
By Equation~\eqref{eqn:boundarygeneration} their Lie algebras generate the Lie algebra of $G$, so the local invariance argument in Proposition~\ref{prop:closedcarrier} shows that $f$ is locally constant almost everywhere.
As $U_Z\cap U_{Z'}$ is connected, $f$ takes a single value there.

Finally, the intersections corresponding to pairs sharing a boundary lift have nonempty open overlap by the geometric argument in Proposition~\ref{prop:closedcarrier}.
Their constants therefore agree as we vary $Z,Z'$, and $f$ is essentially constant on $\calU_\partial$.
\end{proof}

\begin{rem}\label{rem:metricergodicity}
The same proof also gives metric ergodicity of the $\Gamma$-action on $(\calC^*,\lambda)$ in the sense of Bader and Gelander \cite{BG}.
We use their Corollary~6.7 in the boundary step, after which the local invariance and overlap arguments apply to metric-valued maps.
\end{rem}

\begin{proof}[Proof of Theorem~\ref{thm:higherdimensionhaar}]
Fix a boundary lift $Z$ and apply Proposition~\ref{prop:boundaryhomogeneous} to $\eta_\infty$.
The support of $\eta_{Z,B}$ lies in $p^{-1}(\calE_{Z,B})$, which is $\eta_\infty$-null for $B\subset Z$ by Corollary~\ref{cor:algebraicnull}.
Since $\eta_{Z,B}\le\eta_\infty\vert_{U_Z}$, we obtain $\eta_{Z,B}=0$ for every such $B$ and hence $\eta_Z^{\mathrm{Ex}}=0$.
Equivalently, $\mu_a^{\mathrm{Ex}}=0$ for $\sigma_I$-almost every $a$ on each compact angle set $I$.
Since $\mu_a$ is a probability measure, we have $c(a)=1$ and therefore $\mu_a=m_{X_Z}$ for $\sigma_I$-almost every $a$ and $\mu_I=m_{X_Z}\otimes\sigma_I$.
Returning to the transverse current and exhausting the angle ball, we conclude that $\eta_\infty\vert_{U_Z}$ is $H_Z$-invariant.

Applying this at a second boundary lift $Z'$, Equation~\eqref{eqn:boundarygeneration} gives local $G$-invariance on $U_Z\cap U_{Z'}$ by the argument in Proposition~\ref{prop:closedcarrier}.
Thus the current is locally a constant multiple of $\lambda$ on each such intersection and the connectedness argument in Proposition~\ref{prop:lambdaergodicity} shows that the constants agree.
Since the union of these intersections is conull for both measures by Lemma~\ref{lem:aeboundarylift} and Corollary~\ref{cor:twocrossings}, this gives the required equality of measures.
Finally, $\eta_\infty$ is nonzero by Corollary~\ref{cor:twocrossings} so the common constant is positive.
\end{proof}

The assumption $m\ge3$ enters the boundary decomposition in Proposition~\ref{prop:boundaryhomogeneous}, where $J=\SO_0(m-1,1)$ must be generated by unipotents to apply Ratner's theorem, and the noncompact-carrier argument in Proposition~\ref{prop:incidencenonconcentration}, which uses the first assertion of Lemma~\ref{lem:normalrecurrence}.
For maximal geodesic surfaces this group is $\SO_0(1,1)=A$ in every ambient dimension.
We therefore need different arguments for nonconcentration and for obtaining homogeneous boundary measures.
We take up these two points in the next section.

\section{Absolute continuity for maximal geodesic surfaces}\label{sec:surfaceabsolute}

Throughout this section we fix $m=2$ and $n\ge3$.
We retain the notation of Sections~\ref{sec:limitingmeasure} and \ref{sec:higherdimensionabsolute}, including the limiting currents $\eta_\infty$ and $\overline{\eta}_\infty$.
Our goal is to prove the following.

\begin{thm}\label{thm:surfacehaar}
The current $\eta_\infty$ is nonzero, and there exists a positive constant $c$ such that $\lambda=c\eta_\infty$.
\end{thm}

The boundary intersection group of the previous section is now $J=\SO_0(1,1)=A$ so the transferred measures are invariant under the frame flow.
Since this group is not generated by unipotents, we cannot apply Ratner's theorem directly.
We therefore seek additional invariance under unipotent subgroups, after which the carrier argument of Section~\ref{sec:intermediateclosures} will identify the current with Haar measure.

There are two parts to the argument.
First, for every nonzero invariant $\varrho\le\overline{\eta}_\infty$, we obtain a boundary measure $\mu^\varrho_{Z,I}$ whose probability normalization $\sigma^\varrho_{Z,I}$ has positive entropy.
If all of these normalized measures had zero entropy, the small-scale behavior of the current would produce a $\Gamma$-equivariant probability field on the fibers of $\bbP T\overline{\calC}$.
Section~\ref{sec:surfacenonconcentration} rules out such a field and Section~\ref{sec:surfaceboundary} carries out this zero-entropy argument after comparing the entropy functions at distinct boundary lifts.

On the positive-entropy part, we show that the complete unstable leafwise measures have nontrivial projective stabilizers.
The comparison at two boundary lifts is again essential.
Two points on the same stable plaque have close unstable measures after flowing but their target endpoint descriptions differ by a translation.
Unless that translation preserves the complete ray, the two plaque images remain separated and strict concavity forces a positive entropy gap.
The entropy comparison already established rules this out.
We carry this out in Section~\ref{sec:surfaceperiods} and use the resulting additional invariance in Section~\ref{sec:surfaceidentification} to return to the carrier argument of the previous section.

\subsection{Nonconcentration and measures on the projectivized tangent bundle}\label{sec:surfacenonconcentration}

In this subsection, we adapt the recurrence arguments of Section~\ref{sec:incidencenonconcentration} to show that there is no $\Gamma$-equivariant probability field on the projective tangent bundle $\bbP T\overline{\calC}$.
Before proving this, we recall the reference planes and tangent coordinates in the surface case.

Recall from Section~\ref{sec:crossinghyperplanes} that for an underlying plane $C\in\overline{\calC}$, the fiber of the tangent bundle over $C$ is given by
$$T_C\overline{\calC}=\Hom(V_C,V_C^\perp),$$
where, as in Section~\ref{sec:limitingmeasure}, $V_C$ is the three-dimensional indefinite subspace giving rise to $C$ and $V_C^\perp$ is the $(n-2)$-dimensional orthogonal complement with respect to our form.
As before, fix a lift of a geodesic boundary component $Z$ with outward unit normal $n_Z$.
We use the construction of Section~\ref{sec:crossinghyperplanes} with $m=2$, so $V_0$ is two-dimensional and $D_0=V_*^\perp$ has dimension $n-2$.
Here $V_*=V_0\oplus\bbR n_Z$ is the reference three-dimensional vector space giving rise to the surface $C_*=V_*\cap\bbH^n$, which intersects $Z$ in the geodesic $V_0\cap\bbH^n$.
Using the adapted-frame convention of Section~\ref{sec:crossinghyperplanes}, we write $J=A=\{a_t\}=\SO_0(V_0)$ for the corresponding split subgroup of $H_Z$.
It acts loxodromically on this geodesic and fixes $D_0+\bbR n_Z$ pointwise.
Equation~\eqref{eqn:referenceframe} gives the reference surface $C_a$ with linear span $V_{C_a}=R_aV_*$ and ordered normal frame $(n_1(a),\ldots,n_{n-2}(a))$ for $a\in\bbB^{n-2}$.
The rotations $R_a$ commute with $A$.
For $C=hC_a$ and $\Theta\in T_C\overline{\calC}$, we use Equation~\eqref{eqn:pullback} to define $T\in\Hom(V_*,V_*^\perp)$.
Since $J=A$, Equation~\eqref{eqn:Jinvariant} gives the right $A$-action on these coordinates.

As in Proposition~\ref{prop:incidencenonconcentration}, the idea is to compare the support restrictions imposed by recurrence at two boundary components.
Each intersection gives three weight spaces in $\Hom(V_C,V_C^\perp)$.
The Haar lift along the $A$-orbits descends to a finite invariant measure on the frame bundle over the boundary, so recurrence forces the probability field to be supported on the union of these projectivized weight spaces \cite{Furstenberg}.
Almost every plane crosses two distinct boundary lifts and the corresponding unions are disjoint, giving the contradiction.

\begin{prop}\label{prop:noequivprob}
Let $\varrho\le\overline{\eta}_\infty$ be a nonzero $\Gamma$-invariant current.
Then there does not exist a $\varrho$-measurable field of probabilities $\zeta$ on $\bbP T\overline{\calC}$ such that
\begin{enumerate}[label=\textup{(\roman*)}]
\item $\zeta_C$ is a probability measure supported on $\bbP T_C\overline{\calC}$ for $\varrho$-almost every $C$, and
\item $\zeta_{\gamma C}=(D\gamma_C)_*\zeta_C$ for every $\gamma\in\Gamma$ and $\varrho$-almost every $C$.
\end{enumerate}
In particular, the current $\overline{\eta}_\infty$ is non-atomic.
\end{prop}

\begin{proof}
Fix a lift $Z$ of a geodesic boundary component and a surface $C$ such that $\ell=Z\cap C$ is a geodesic.
Then $V_C=V_\ell\oplus\bbR v$ for some vector $v\in V_C$ of positive norm.
The group $\SO_0(V_\ell)\cong\SO_0(1,1)$ decomposes $V_C^*$ into one-dimensional weight spaces which we denote $\ell_+$, $\ell_-$, and $\ell_0$.
This group exponentially expands and contracts the first two spaces, respectively, and acts trivially on the third.
Consequently
\begin{equation}\label{eqn:tangentdecomposition}
T_C\overline{\calC}\cong\left(V_C^\perp\otimes\ell_-\right)\oplus\left(V_C^\perp\otimes\ell_0\right)\oplus\left(V_C^\perp\otimes\ell_+\right).
\end{equation}
A projective point with nonzero components in more than one of these summands tends to the projectivization of its highest occupied weight.
Poincar\'e recurrence therefore forces an invariant probability to be supported on
\begin{equation}\label{eqn:support}
\bbP(V_C^\perp\otimes\ell_-)\cup\bbP(V_C^\perp\otimes\ell_0)\cup\bbP(V_C^\perp\otimes\ell_+).
\end{equation}
When $C$ intersects a second boundary lift $Z'$ in a geodesic $\ell'$, we get another decomposition of the form in Equation~\eqref{eqn:tangentdecomposition}.
The two intersection geodesics are ultraparallel so their corresponding unions in Equation~\eqref{eqn:support} are disjoint.
It remains to obtain this support restriction from an equivariant field.

Suppose that such a field exists.
Lift $\varrho$ to $\varrho_{\mathrm{fr}}$ on $\calC$ by adding Haar probability in every compact normal-frame fiber then pull the bundle and field back along $p:\calC\to\overline{\calC}$.
Thus, for a framed plane $C$, the notation $\zeta_C$ means $\zeta_{p(C)}$.
For a compact set $I\subset\bbB^{n-2}$, apply the Haar-lift construction of Lemma~\ref{lem:banana} to the joint measure $d\zeta_C\,d\varrho_{\mathrm{fr}}(C)$.
Using Equation~\eqref{eqn:pullback}, its defining formula is
\begin{equation}\label{eqn:projectivelift}
\int\phi\,d\widetilde{\varrho}_{Z,I}^{\mathrm{proj}}=\int_{U_{Z,I}}\int_{\bbP\Hom(V_C,V_C^\perp)}\int_A\phi(hg,a,[T\circ g])\,dm_A(g) d\zeta_C([\Theta])\,d\varrho_{\mathrm{fr}}(C),
\end{equation}
where $C=hC_a$ and $T$ is as in Equation~\eqref{eqn:pullback}.
Since $\varrho$ is $\Gamma$-invariant and the field is $\Gamma$-equivariant, this lift descends to a finite measure on
$$\Delta_Z\backslash H_Z\times I\times\bbP\Hom(V_*,V_*^\perp),$$
invariant under the right $A$-action of Equation~\eqref{eqn:Jinvariant}.
Disintegrating over $I$ and normalizing, its projective marginal is an $A$-invariant probability.
Recurrence gives the support restriction in Equation~\eqref{eqn:support} and Equation~\eqref{eqn:projectivelift} gives the same restriction for $\zeta_C$ almost everywhere.
A countable exhaustion of the angle parameters and comparison at a second boundary lift give the contradiction.

Finally, suppose that $C$ is an atom of $\overline{\eta}_\infty$.
Its $\Gamma$-orbit is locally finite, since all orbit points have the same positive mass.
After removing at most one of the distinct original surfaces, their currents give $C$ zero mass.
The blow-up construction of Lemma~\ref{lem:noprob} then gives a $\Gamma_C$-invariant measure of mass $\overline{\eta}_\infty(\{C\})$ on $\bbP T_C\overline{\calC}$.
Normalize it and transport it by $D\gamma_C$ around $\Gamma C$.
Stabilizer invariance makes this well defined, yielding the field just excluded over $\overline{\eta}_\infty\vert_{\Gamma C}$.
Therefore no such atom exists.
\end{proof}

\subsection{Entropy of the boundary measures}\label{sec:surfaceboundary}

In this subsection, we show that for every nonzero $\Gamma$-invariant $\varrho\le\overline{\eta}_\infty$ there is some boundary measure $\mu^\varrho_{Z,I}$ whose probability normalization $\sigma^\varrho_{Z,I}$ has nonzero entropy.
As before, these measures come from the construction of Lemma~\ref{lem:banana}.
In Proposition~\ref{prop:entropyindependence}, we will use the fact that the transition between endpoint coordinates is locally bi-Lipschitz in either root direction to show that the component entropies agree on overlaps and define a $\Gamma$-invariant function on the current.
The entropy--dimension formulas of Ledrappier--Young \cite[Prop~7.3.1, Thm~C]{LedrappierYoungII} and Kaimanovich \cite[\S2.3, Thm~2.4 and the following remark]{KaimanovichEntropy}, recalled in Section~\ref{sec:leafwiseentropy}, allow us to make this comparison at the level of endpoint measures.
The geometric input is that changing the boundary lift preserves dimension along each root plaque.
An inequality between the two entropies would then force a plaque image to be carried by a set of smaller dimension.
Proposition~\ref{prop:dyadiccompute} records the resulting comparison between the target entropy and its conditional entropy, which we use in the next subsection.

Finally, for the positive-entropy assertion, we argue by contradiction in Proposition~\ref{prop:positiveentropycomponent}.
If all normalized boundary measures had zero entropy, their endpoint measures, and hence localizations of the current, would have sublinear small-scale entropy.
For each $\gamma\in\Gamma$, the normalized restrictions to balls centered at $\gamma C$ and the $\gamma$-images of those centered at $C$ would, after averaging over scales, differ by a quantity tending to zero in total variation.
Pushing these restrictions forward to the projective classes of the linear maps describing nearby planes as graphs and taking a subsequential limit would give the equivariant probability field excluded in the previous subsection.
We now fix the notation for both arguments.

Fix a lift $Z$ of a geodesic boundary component and let $U_Z\cong H_Z/A\times\bbB^{n-2}$ as in Lemma~\ref{lem:banana}.
For an invariant $\varrho\le\overline{\eta}_\infty$, apply that construction over a compact angle set $I$ with Haar measure in the compact fibers and denote the resulting finite measure on $X_Z\times I$ by $\mu^\varrho_{Z,I}$.
Since $A$ acts trivially on $I$, the angle parameter is constant on almost every ergodic component.
When $\mu^\varrho_{Z,I}$ has positive mass, put
$$\sigma^\varrho_{Z,I}=\frac{\mu^\varrho_{Z,I}}{\mu^\varrho_{Z,I}(X_Z\times I)}.$$
Using the component entropy notation of Section~\ref{sec:leafwiseentropy}, define
\begin{equation}\label{eqn:boundarycomponententropy}
h^\varrho_{Z,I}(x)=h_{\sigma^\varrho_{Z,I}}(x).
\end{equation}
If $\mu^\varrho_{Z,I}=0$, this function is left undefined.
We write $h_Z(x)$ when $\varrho$ and $I$ are fixed by context.

For the comparison, take a small overlap box lifted to $H_Z$.
In the chart of Equation~\eqref{eqn:rootflowbox}, let $\sfX,\sfY\in\bbR^{n-2}$ be Euclidean coordinates for the forward and backward endpoints and retain the compact frame coordinate $\varpi$.
Suppressing the $A$-coordinate gives
\begin{equation}\label{eq:overlapcoordinates}
(hA,a)\longmapsto(\sfX,\sfY,\varpi,a)=(\sfX,R),
\end{equation}
where $R=(\sfY,\varpi,a)$.
On a single ergodic component the angle is fixed so we may then suppress it from $R$.
Fixing the transverse coordinates $R$ specifies a local $U^+$-plaque.
For a product Hopf box $\calB=Q\times I_t$ of positive measure, write $\kappa_{Q,R}$ for the plaque conditionals of the probability $\kappa_Q/\kappa_Q(Q)$, obtained by normalizing the transverse measure in Equations~\eqref{eqn:plaqueconditionals} and \eqref{eqn:hopfdisintegration}.
Their average is the probability normalization of the endpoint measure $\beta^+_{\calB}$ from Equation~\eqref{eqn:localizedendpointmarginals}.
The dimension identity in Equation~\eqref{eqn:endpointdimension} applies to both endpoint measures.

On a smaller overlap with a second boundary lift $Z'$, choose either endpoint of $C\cap Z'$.
Its coordinate in a fixed chart is a smooth function $f=f_{Z,Z'}(\sfX,R)\in\bbR^{n-2}$ of the underlying plane and its image angle may vary.
For a boundary frame $x$, we write
$$f(x)=f_{Z,Z'}\bigl(\sfX(x),R(x)\bigr),$$
and let $D_{U^\pm}f(x)$ denote the derivative at zero of $z\mapsto f(xu^\pm(z))$.
We use the dyadic entropy notation of Section~\ref{sec:leafwiseentropy} with scale index $j$ throughout.
Write $f_j(x)$ for the dyadic cube of side length $2^{-j}$ containing $f(x)$.

\begin{prop}\label{prop:entropyindependence}
For every nonzero $\Gamma$-invariant $\varrho\le\overline{\eta}_\infty$, the function $h^\varrho_{Z,I}$ of Equation~\eqref{eqn:boundarycomponententropy} is almost everywhere constant along the fibers that forget the $A$-coordinate and compact normal frame.
Its resulting value on the underlying current is independent of $I$ and agrees on every positive-mass overlap of two boundary lifts $Z$ and $Z'$.
In particular, these values define a measurable $\Gamma$-invariant function $h^\varrho(C)$ for $\varrho$-almost every plane $C$.
On every sufficiently small overlap, both $D_{U^+}f(x)$ and $D_{U^-}f(x)$ are invertible.
\end{prop}

\begin{proof}
We first compute the derivatives $D_{U^\pm}$ of the endpoint map.
Fix a surface $C$ crossing $Z$ and $Z'$.
Choose vectors $v_-$, $v_+$ representing the backward and forward endpoints of $C\cap Z$, normalized by the condition that $\langle v_-,v_+\rangle=-1$.
Let $w_C$ be the unit vector in $V_C$ orthogonal to both, with positive $n_Z$-component.
Then we may write $V_C=\Span(v_-,v_+,w_C)$, where $w_C=sn_Z+b$, $s$ is positive, and
$$b\in N_Z=V_Z\cap\Span(v_-,v_+)^\perp.$$
Both $N_Z$ and $V_C^\perp$ clearly have dimension $n-2$.

Consider the horospherical subgroup of $H_Z$ fixing $v_-$ and identify its Lie algebra with $N_Z$.
For $v\in N_Z$, write $\delta$ for differentiation at $t=0$ along the corresponding one-parameter subgroup.
One computes that its action gives
$$\delta v_-=0,\qquad\delta v_+=v,\qquad\delta q=\langle q,v\rangle v_-,\qquad\delta n_Z=0,$$
where $q\in N_Z$.
In particular,
$$\delta w_C=\langle b,v\rangle v_-\in V_C.$$
As $C$ varies, choose a smooth representative $\xi'$ for one endpoint of $C\cap Z'$ and write
$$\xi'=b_1v_-+b_2v_++b_3w_C\in V_C\cap V_{Z'}.$$
Differentiating and projecting to $V_C^\perp$ gives
\begin{equation}\label{eqn:projection}
\pr_{V_C^\perp}(\delta\xi')=b_2\pr_{V_C^\perp}(v).
\end{equation}
We extrapolate three facts from this calculation.
First, a vector in the kernel of $\pr_{V_C^\perp}\vert_{N_Z}$ lies in $N_Z\cap V_C$ and must be a multiple of $w_C$.
Its $n_Z$ component is zero so this vector vanishes and the projection is an isomorphism.
Second, the same projection identifies the tangent space
$$T_{[\xi']}\partial Z'\cong\bigl(V_{Z'}\cap(\bbR\xi')^\perp\bigr)/(\bbR\xi'),$$
with $V_C^\perp$.
Indeed, a representative in its kernel lies in the indefinite $2$-plane $V_C\cap V_{Z'}$ and is orthogonal to $\xi'$, so is a multiple of $\xi'$.
Finally, we must have $b_2\ne0$, since otherwise $\xi'$ would be an isotropic vector in $\Span(v_-,w_C)$ and therefore a multiple of $v_-$, whereas the endpoints of $C\cap Z$ do not lie in $\partial Z'$.
The calculation for the other root is the same.
After shrinking the chart if necessary, the derivatives and their inverses are uniformly bounded on compact sets.

We now show that the entropy functions agree.
If they did not then, after restricting to a set $E$ of positive measure in an overlap and interchanging $Z$, $Z'$ if necessary, we may assume that
$$h_Z\ge h_1>h_0\ge h_{Z'},$$
on $E$, where $h_0,h_1$ are rational.
Disintegrate the restricted measure along the local $U^+$-plaques, keeping the opposite endpoint, compact frame, and angle fixed.
By Section~\ref{sec:leafwiseentropy}, almost every resulting conditional has lower Hausdorff dimension at least $h_1$.
The same holds for its image under $f$, since $f$ is locally bi-Lipschitz on each plaque.
Their average is the image of the restricted overlap measure, which the Haar-lift construction makes absolutely continuous with respect to an endpoint measure $\beta'$ for $Z'$.
By Equation~\eqref{eqn:endpointdimension}, this image is carried by
$$E_{h_0}(\beta')=\{z\mid\overline{d}_{\beta'}(z)\le h_0\},$$
where $\overline{d}_{\beta'}$ denotes upper pointwise dimension.
For every $\epsilon>0$, this is a countable union of bounded sets such that $\beta'(B(z,r))\ge r^{h_0+\epsilon}$ below a common positive radius on each.
Counting disjoint balls shows that each such set meets $O(2^{(h_0+\epsilon)j})$ dyadic cubes of side length $2^{-j}$.
Thus $E_{h_0}(\beta')$ has Hausdorff dimension at most $h_0$.
Almost every plaque image gives this set zero mass whereas their average is carried by it, a contradiction.

Taking countably many overlap charts and rational thresholds makes the equalities $h_Z=h_{Z'}$ simultaneous.
Invariance under $A$ and Equation~\eqref{eqn:endpointdimension} show that the common value is independent of the $A$-coordinate and compact normal frame.
The Haar lifts and their disintegrations commute with restriction to the angle parameters so the value is also independent of $I$.
The construction is $\Gamma$-equivariant and hence defines the claimed function $h^\varrho(C)$.
\end{proof}

A straightforward consequence is the following entropy comparison.
Only the target endpoint coordinate is needed here and throughout we condition on the full value of $R$.

\begin{prop}\label{prop:dyadiccompute}
Let $\varrho$ and $h^\varrho$ be as in Proposition~\ref{prop:entropyindependence}.
For almost every ergodic component of entropy $h$, normalize its transverse measure on a bounded overlap chart of positive measure to a probability $\kappa$.
Then
\begin{equation}\label{eqn:surfaceentropyidentities}
H(f_j\mid R)=hj+o(j),\qquad
\calI_j=H(f_j)-H(f_j\mid R)=o(j),
\end{equation}
where $\calI_j$ denotes the displayed entropy difference.
The assertion holds for either target endpoint and after restriction to smaller bounded charts of positive measure.
\end{prop}

\begin{proof}
By Section~\ref{sec:leafwiseentropy}, almost every source plaque conditional has pointwise dimension $h$.
By Proposition~\ref{prop:entropyindependence}, the same is true of its image under $f$.
The dimension--entropy comparison \cite[Lem~3.3, Prop~3.4]{HochmanShmerkin} and dominated convergence on the bounded chart give
$$H(f_j\mid R)=hj+o(j).$$
For the other equation, let $\beta'$ be a target endpoint measure.
Equation~\eqref{eqn:endpointdimension} and Proposition~\ref{prop:entropyindependence}, followed by disintegration over the source components, show that $f_*\kappa$ is carried by $E_h(\beta')$ for almost every such component.
By the covering argument above, for every $\epsilon,\delta>0$, finitely many sets meeting $O(2^{(h+\epsilon)j})$ dyadic cubes carry mass at least $1-\delta$.
Their complement lies in a bounded chart meeting $O(2^{(n-2)j})$ cubes so
$$\limsup_{j\to\infty}\frac{H(f_j)}j\le h+\epsilon+\delta(n-2).$$
Letting $\epsilon,\delta\to0$ and using $H(f_j)\ge H(f_j\mid R)$ proves Equation~\eqref{eqn:surfaceentropyidentities}.
Restriction to a smaller chart preserves the dimension statements by differentiation and the argument treats either endpoint.
\end{proof}

We now prove the positive-entropy assertion.

\begin{prop}\label{prop:positiveentropycomponent}
For every nonzero $\Gamma$-invariant $\varrho\le\overline{\eta}_\infty$, there is a boundary lift $Z$ and a compact angle set $I$ for which $\mu^\varrho_{Z,I}$ has positive mass and its normalized $A$-invariant probability has positive entropy, that is to say, $h(\sigma^\varrho_{Z,I})>0$.
\end{prop}

\begin{proof}
Fix $\varrho$ as above and suppose for contradiction that every positive-mass boundary measure has probability normalization of entropy zero.
We will use this assumption to construct the equivariant probability field excluded by Proposition~\ref{prop:noequivprob}.

For $C'\ne C$ sufficiently close to $C$, write $V_{C'}$ as the graph of $A_{C,C'}:V_C\to V_C^\perp$ and put
$$\fraka(C,C')=[A_{C,C'}]\in\bbP T_C\overline{\calC}.$$
The orthogonal splittings are $\Gamma$-equivariant, hence
$$\fraka(\gamma C,\gamma C')=D\gamma_C\fraka(C,C').$$
Choose a complete smooth auxiliary metric on $\overline{\calC}$ and, for sufficiently large $t$, define
$$\zeta_{C,t}=\fraka(C,\cdot)_*\left(\frac{\varrho\vert_{B(C,e^{-t})\ssm\{C\}}}{\varrho(B(C,e^{-t}))}\right).$$
Proposition~\ref{prop:noequivprob} makes $\varrho$ non-atomic so this is a probability for $\varrho$-almost every $C$.
The initial range of $t$, before the ball lies in a graph chart, is locally bounded in $C$ and may be filled with measurably chosen probabilities.

The balls need not be preserved by $\Gamma$ so we average over scales by letting
$$\overline{\zeta}_{C,s}=\frac1s\int_0^s\zeta_{C,t}\,dt.$$
Fix $\gamma\in\Gamma$ and a compact set $E$ in the locus where planes cross two distinct boundary lifts, which is $\varrho$-conull by Corollary~\ref{cor:twocrossings}.
Bounded local distortion gives $b>0$ and $t_0>b$ such that
$$B(\gamma C,e^{-(t+b)})\subseteq\gamma B(C,e^{-t})\subseteq B(\gamma C,e^{-(t-b)}),$$
for $C\in E$ and $t\ge t_0$.
The corresponding normalized restrictions have a common part of mass at least the ratio of the inner and outer ball masses.
By $\Gamma$-invariance and covariance of $\fraka$, this gives
$$\begin{aligned}
\|\zeta_{\gamma C,t}-(D\gamma_C)_*\zeta_{C,t}\|_{\mathrm{TV}}
&\le 2\left(1-\frac{\varrho(B(\gamma C,e^{-(t+b)}))}{\varrho(B(\gamma C,e^{-(t-b)}))}\right),\\
&\le 2\ln\frac{\varrho(B(\gamma C,e^{-(t-b)}))}{\varrho(B(\gamma C,e^{-(t+b)}))},
\end{aligned}$$
where $\|\alpha-\beta\|_{\mathrm{TV}}=\sup_{\|\phi\|_\infty\le1}|\alpha(\phi)-\beta(\phi)|$ denotes total variation.

The scale integral cancels at intermediate scales, since for $s>t_0+2b$ we have
$$\int_{t_0}^s\ln\frac{\varrho(B(C,e^{-(t-b)}))}{\varrho(B(C,e^{-(t+b)}))}\,dt=\int_{s-b}^{s+b}\ln\frac1{\varrho(B(C,e^{-u}))}\,du-\int_{t_0-b}^{t_0+b}\ln\frac1{\varrho(B(C,e^{-u}))}\,du.$$
The second term is fixed and finite for $C\in\supp\varrho$, while the first is at most $2b\ln(1/\varrho(B(C,e^{-(s+b)})))$.
Thus the pointwise estimate we need is
\begin{equation}\label{eqn:currentzerodimension}
\lim_{t\to\infty}\frac1t\ln\varrho(B(C,e^{-t}))=0,
\end{equation}
for $\varrho$-almost every $C$.
Assuming it for now, convexity of total variation and $\Gamma$-invariance give
$$\|\overline{\zeta}_{\gamma C,s}-(D\gamma_C)_*\overline{\zeta}_{C,s}\|_{\mathrm{TV}}\longrightarrow0,$$
for almost every $C$, with the initial interval contributing at most $2t_0/s$.
This defect is bounded by $2$, so bounded convergence on $E$ gives
\begin{equation}\label{eqn:averagedfieldequivariance}
\int_E\|\overline{\zeta}_{\gamma C,s}-(D\gamma_C)_*\overline{\zeta}_{C,s}\|_{\mathrm{TV}}\,d\varrho(C)\longrightarrow0.
\end{equation}
The measures
$$d\widehat{\varrho}_s(C,[\Theta])=d\varrho(C)\,d\overline{\zeta}_{C,s}([\Theta]),$$
on the total space $\bbP T\overline{\calC}$ push forward to $\varrho$, hence have uniformly bounded mass on compact sets.
A diagonal subsequence therefore converges to a Radon measure $\widehat{\varrho}$ and, since the bundle projection is proper, the pushforward of this limit is still $\varrho$.
Equation~\eqref{eqn:averagedfieldequivariance}, applied on the compact projections of test supports, shows that $\widehat{\varrho}$ is $\Gamma$-invariant.
Disintegration over the fibers and uniqueness of conditional probabilities will give the equivariant probability field excluded by Proposition~\ref{prop:noequivprob}.
It remains to prove Equation~\eqref{eqn:currentzerodimension}, which will follow from the zero-entropy assumption.

Equation~\eqref{eqn:endpointdimension} gives pointwise dimension zero for every localized endpoint measure.
We transfer this information to the current by recovering a surface from three endpoints.
Fix a small chart of surfaces crossing two distinct boundary lifts $Z,Z'$, choose the endpoint labels continuously, and put
$$\Phi(C)=\bigl(\xi_Z^+(C),\xi_Z^-(C),\xi_{Z'}^+(C)\bigr).$$
The three null lines are independent because the intersection geodesics are ultraparallel, and their span recovers $V_C$.
After shrinking the chart, $\Phi$ therefore has a Lipschitz inverse.

Let $E$ be a positive-mass Borel set contained in a compact subset of this chart and put $\varrho_E=\varrho\vert_E/\varrho(E)$.
We will use the zero pointwise dimension of the localized endpoint measures to bound the number of small sets needed to carry $\varrho_E$.
The Lipschitz inverse of $\Phi$ will then let us compare the masses of these sets with the ball masses in Equation~\eqref{eqn:currentzerodimension}.

We first compare the endpoint distributions of $\varrho_E$ with the localized endpoint measures, for which the dimension bounds are already known.
Over a neighborhood of $\overline{E}$, choose local normal-frame and flow coordinates $(k,t)$.
Restrict these coordinates to a relatively compact product patch $P$ of positive Haar mass, small enough that the lift lies in an injective quotient chart with angles in a compact set $I$.
Lemma~\ref{lem:banana} gives the local Haar lift
$$d\widetilde{\mu}^\varrho_{Z,I}(C,k,t)=d\varrho(C)\,dk\,dt.$$
The endpoints are independent of $k,t$, so the endpoint images of the restriction to $E\times P$ are proportional to $(\xi_Z^\pm)_*\varrho_E$.

Covering $E\times P$ by finitely many bounded product Hopf boxes $\calB_1,\ldots,\calB_N$ and applying Equation~\eqref{eqn:localizedendpointmarginals} gives
$$(\xi_Z^\pm)_*\varrho_E\ll\sum_{\ell=1}^{N}\beta^\pm_{\calB_\ell}.$$
Repeating the construction at $Z'$ gives the same comparison for $(\xi_{Z'}^+)_*\varrho_E$.

Fix $\epsilon>0$ and let $\beta$ be one of these finite endpoint measures on a bounded chart.
For each positive integer $m$, let $F_{\beta,m}$ be the set of points $z$ for which
$$\beta(B(z,2^{-j}))\ge 2^{-\epsilon j},$$
for every $j\ge m$.
Since $\beta$ has pointwise dimension zero, the union of the sets $F_{\beta,m}$ has full $\beta$-measure.
To count the dyadic cubes meeting $F_{\beta,m}$, choose a maximal $2^{1-j}$-separated collection of points in this set.
The balls of radius $2^{-j}$ about these points are disjoint and each has $\beta$-mass at least $2^{-\epsilon j}$, so their number is $O(2^{\epsilon j})$.
By maximality, the balls with the same centers and radius $2^{1-j}$ cover $F_{\beta,m}$.
Each of these enlarged balls meets a bounded number of dyadic cubes of side length $2^{-j}$, giving the required bound of $O(2^{\epsilon j})$ cubes.

The measure comparison shows that each of the three endpoint distributions of $\varrho_E$ is carried by a countable union of sets with this covering bound.
For any choice of one such set in each endpoint coordinate, their product is covered by $O(2^{3\epsilon j})$ product cubes.
Taking the inverse images of these products under $\Phi$, we find that $\varrho_E$ is carried by countably many sets, each covered by $O(2^{3\epsilon j})$ cells of
$$\calQ_j=\Phi^{-1}(\calD_j^{\otimes3})\vert_E.$$
No independence of the endpoints is needed, since we only multiply the numbers of cubes covering the three chosen sets.
Moreover, the Lipschitz inverse of $\Phi$ gives a constant $K$ such that every cell of $\calQ_j$ has diameter at most $K2^{-j}$.

We now use this bound on the number of cells to obtain a lower bound for the mass of the cell containing almost every point.
For one of the sets $F$ just constructed, write $Q_j(C)$ for the cell containing $C$.
Since $F$ is covered by $O(2^{3\epsilon j})$ cells, summing the masses of the covering cells whose $\varrho_E$-mass is less than $2^{-4\epsilon j}$ gives
$$\varrho_E\bigl(\{C\in F\mid\varrho_E(Q_j(C))<2^{-4\epsilon j}\}\bigr)\le O(2^{3\epsilon j})\,2^{-4\epsilon j}=O(2^{-\epsilon j}).$$
The right-hand side is summable in $j$.
Borel--Cantelli therefore shows that almost every $C\in F$ belongs to such a cell for only finitely many $j$.
Since there are only countably many sets $F$ and their union has full $\varrho_E$-measure, we obtain
$$\varrho_E(Q_j(C))\ge 2^{-4\epsilon j},$$
for all sufficiently large $j$, for $\varrho_E$-almost every $C\in E$.

Choose a fixed integer $k$ with $K2^{-k}<1$ and put $j=\lceil t/\ln2\rceil+k$.
The cell $Q_j(C)$ contains $C$ and has diameter less than $e^{-t}$, hence $Q_j(C)\subset B(C,e^{-t})$.
Since $Q_j(C)\subset E$, the definition of $\varrho_E$ gives
$$\varrho(B(C,e^{-t}))\ge\varrho(E)\varrho_E(Q_j(C))\ge\varrho(E)\,2^{-4\epsilon j},$$
for all sufficiently large $t$ and almost every $C\in E$.
For each such $C$, local finiteness also gives a finite positive upper bound $M_C$ for these ball masses.
Taking logarithms of both bounds, dividing by $t$, and using $j\ln2/t\to1$, we obtain
$$-4\epsilon\le\liminf_{t\to\infty}\frac1t\ln\varrho(B(C,e^{-t}))\le\limsup_{t\to\infty}\frac1t\ln\varrho(B(C,e^{-t}))\le0,$$
for almost every $C\in E$.
Letting $\epsilon$ decrease through positive rational values therefore gives Equation~\eqref{eqn:currentzerodimension} for almost every $C\in E$.
Finally, cover the conull two-crossing locus by countably many such localizations.
The estimate then holds for $\varrho$-almost every $C$ and completes the construction of the equivariant probability field excluded by Proposition~\ref{prop:noequivprob}, giving the required contradiction.
\end{proof}

\subsection{Translation invariance of the conditional measures}\label{sec:surfaceperiods}

We now show that positive entropy forces the complete unstable leafwise measures to have nontrivial projective stabilizers.
Retain the overlap coordinates $(\sfX,R)$ and the target endpoint map $f=f_{Z,Z'}(\sfX,R)$ of the previous subsection.

The argument belongs to the shearing and drift methods originating in Ratner's $H$-property and developed with leafwise measures in the low-entropy method \cite[\S4]{EKLLittlewood} and \cite[\S2]{ELRankOne}.
Our use of exponential contraction and expansion is closer to the exponential-drift method of Benoist and Quint \cite[\S1.2, \S7]{BQStationary} and Eskin--Mirzakhani \cite[\S2.1, \S2.3]{EskinMirzakhani}.
The geometric comparison is closest to the eight-point construction of Eskin and Lindenstrauss \cite[\S1.2, Fig~1, \S6.1]{ELZariskiWalks}, \cite[\S1.5, \S10]{ELLocalWalks}.
There, an additional dynamical branch produces a separated pair whose leafwise measures are compared using flow covariance and continuity on a compact Lusin set \cite[Lem~5.7, Claim~6.9]{ELZariskiWalks}.
Here we compare the endpoints determined by the second boundary lift for corresponding points on two nearby unstable plaques.
Since $f$ is not assumed to preserve the measure, we cannot identify the target images by a leafwise transport relation.

The way around this is to use the sublinear entropy bound $\calI_j=o(j)$ of Proposition~\ref{prop:dyadiccompute}.
We argue by contradiction on a positive-entropy ergodic component $\sigma$ whose complete unstable measures have trivial projective stabilizers.
Our aim is to find, at every sufficiently small scale, a set of stable pairs whose coupling mass has a uniform positive lower bound.
Their target plaque measures, restricted and normalized to a common coarse cube, should give probability vectors on a fixed subdivision whose $\ell^1$-distance also has a uniform positive lower bound.
The common cube lets us compare both vectors with the same average, so Equations~\eqref{eqn:conditionalentropydifference} and \eqref{eqn:strictconcavitygap} give a uniform positive lower bound for the entropy increase across a fixed block of scales.

To obtain this separation, we need to compare unstable restrictions at a shrinking scale.
Weak-$*$ continuity of the normalized complete measures controls integrals against fixed tests but does not directly compare normalized restrictions to shrinking windows.
We therefore follow one stable pair through two successive flows, as in Figure~\ref{fig:surface-sixpoints}.
For an initial stable displacement $v$ in a fixed compact annulus, the first flow gives a middle pair $(x,y)$ with stable displacement $r_jv$, where $r_j=e^{-t_j}$.
At the middle pair, we translate both endpoint images by $-f(x)$ and dilate by $r_j^{-1}$.
In these coordinates, Taylor expansion shows that the endpoint of $yu^+(r_ju)$ is displaced from that of $xu^+(r_ju)$ by $D_{U^-}f(x)v+O(r_j)$, uniformly for $u$ in a fixed bounded set.
The uniform inverse bound for $D_{U^-}f(x)$ bounds the norm of the translation vector away from zero.
The second flow magnifies the unstable restrictions to fixed windows and reduces the stable displacement to $r_j^2v$.
When the final points lie in a compact Lusin set, the difference of their normalized complete measures weak-$*$ converges to zero, uniformly over the chosen pairs.

The image of a common source measure under $D_{U^+}f(x)$ is not proportional to its translate by $D_{U^-}f(x)v$, since proportionality would give a nonzero projective period.
We express this using cross-products of test integrals.
Compactness gives finitely many pairs of tests for which the maximum absolute cross-product has a uniform positive lower bound.
Using finitely many translated dyadic grids, we place both centers and the test supports inside a common coarse cube.
Uniform upper bounds for the masses of this cube give a positive lower bound for a test-integral difference after normalization and approximation on a fixed subdivision gives the required $\ell^1$-separation.
Summing the entropy increases over the finite family of grids forces linear growth, contradicting the sublinear entropy bound.

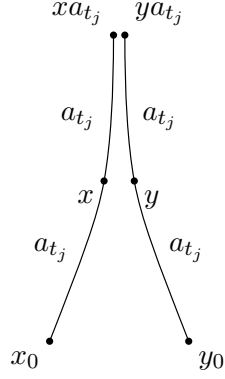
\begin{figure}[!t]
\centering

\begin{tikzpicture}[x=1cm,y=1cm,line width=.45pt,
  every node/.style={inner sep=1.5pt}]
\coordinate (xzero) at (-.92,0);
\coordinate (yzero) at (.92,0);
\coordinate (xmid) at (-.20,2.12);
\coordinate (ymid) at (.20,2.12);
\coordinate (xtop) at (-.075,4.05);
\coordinate (ytop) at (.075,4.05);
\draw (xzero) .. controls (-.49,1.10) and (-.29,1.62) .. (xmid)
  node[pos=.46,left=4pt] {$a_{t_j}$};
\draw (yzero) .. controls (.49,1.10) and (.29,1.62) .. (ymid)
  node[pos=.46,right=4pt] {$a_{t_j}$};
\draw (xmid) .. controls (-.11,2.62) and (-.075,3.27) .. (xtop)
  node[pos=.48,left=4pt] {$a_{t_j}$};
\draw (ymid) .. controls (.11,2.62) and (.075,3.27) .. (ytop)
  node[pos=.48,right=4pt] {$a_{t_j}$};
\foreach \p in {xzero,yzero,xmid,ymid,xtop,ytop}{\fill (\p) circle (1.35pt);}
\node[below left=2pt] at (xzero) {$x_0$};
\node[below right=2pt] at (yzero) {$y_0$};
\node[below left=2pt] at (xmid) {$x$};
\node[below right=2pt] at (ymid) {$y$};
\node[above left=2pt] at (xtop) {$xa_{t_j}$};
\node[above right=2pt] at (ytop) {$ya_{t_j}$};
\end{tikzpicture}
\caption{Two successive right translations by $a_{t_j}$ contract the stable displacement from $v$ to $r_jv$ and then to $r_j^2v$, where $r_j=e^{-t_j}$.
We compare rescaled endpoint images at the middle pair and complete unstable measures at the final pair.
Compare \cite[Fig~1]{ELZariskiWalks}.}
\label{fig:surface-sixpoints}
\end{figure}

\begin{prop}\label{prop:surfaceperiods}
Let $\varrho\le\overline{\eta}_\infty$ be a nonzero $\Gamma$-invariant current.
For every boundary lift $Z$ and compact angle set $I$ for which $\mu^\varrho_{Z,I}$ has positive mass, put $\sigma=\sigma^\varrho_{Z,I}$.
For $\sigma$-almost every $x$ with $h^\varrho_{Z,I}(x)>0$, the projective stabilizer $\calL_x$ of the complete unstable leafwise measure $\sigma_x^+$ is nontrivial.
Equivalently,
$$\sigma\bigl(\{x\mid h^\varrho_{Z,I}(x)>0,\ \calL_x=\{0\}\}\bigr)=0.$$
\end{prop}

\begin{proof}
Suppose that the assertion fails.
By the covariance and disintegration rules of Section~\ref{sec:leafwiseentropy}, we may pass to an ergodic component $\sigma$ of positive entropy on which $\calL_x=\{0\}$ almost everywhere and Propositions~\ref{prop:entropyindependence} and \ref{prop:dyadiccompute} apply.
Its angle is fixed and may be suppressed from $R$.
Choose a small overlap flow box $\calB$ for boundary lifts $Z,Z'$ and normalize its transverse measure to a probability $\kappa$, so that
$$d\sigma\vert_{\calB}=\sigma(\calB)\,d\kappa(\sfX,R)\,\frac{d\tau\vert_{I_t}}{|I_t|}.$$
Here $R$ specifies an unstable plaque and $f=f_{Z,Z'}(\sfX,R)$ records the chosen endpoint at $Z'$.
Shrink the box so that $f$ is uniformly bi-Lipschitz on unstable plaques and both $D_{U^\pm}f$ have uniformly bounded inverses.
Proposition~\ref{prop:dyadiccompute} gives
$$\calI_j=H(f_j)-H(f_j\mid R)=o(j).$$
To construct the stable pairs, choose a compact set $E$ in the interior of $\calB$ with $\sigma(E)=c>0$.
Disintegrate $\sigma$ over a measurable partition into bounded pieces of local stable plaques and choose two points independently from each conditional probability.
This gives
$$\sigma=\int\sigma_\omega\,d\vartheta(\omega),\qquad \Pi=\int\sigma_\omega\otimes\sigma_\omega\,d\vartheta(\omega),$$
where $\Pi$ is the self-coupling of $\sigma$, with both marginals equal to $\sigma$.
Take the plaque pieces small enough that almost every pair has the form
$$y_0=x_0u^-(v),\qquad |v|\le r_+.$$
Positive entropy makes the stable conditionals non-atomic so $\Pi$ gives the diagonal zero mass.
We may therefore choose $r_->0$ with $\Pi\{|v|<r_-\}<c/10$.

We choose representatives of the complete unstable leafwise measures by fixing a smooth compactly supported $0\le\chi\le1$, equal to one near zero, and put
$$\beta_x=\frac{\sigma_x^+}{\sigma_x^+(\chi)}.$$
Choose a compact Lusin set $E_{\mathrm{Lus}}$ with $\sigma(E_{\mathrm{Lus}}^c)<c/10$ on which $x\mapsto\beta_x$ is continuous for weak-$*$ convergence and every $\beta_x$ has trivial projective stabilizer.
Its image $\calK_\sigma$ under this map is compact.

Put $r_j=2^{-j}$, $t_j=j\ln2$, and let $\Pi_j$ be the pushforward of $\Pi$ under simultaneous right translation by $a_{t_j}$.
For the middle pair $x=x_0a_{t_j}$, $y=y_0a_{t_j}$, the conjugation rules give
\begin{equation}\label{eqn:stablepaircontraction}
y=xu^-(r_jv),\qquad ya_{t_j}=xa_{t_j}u^-(r_j^2v).
\end{equation}
The marginals remain $\sigma$ by $A$-invariance and hence the conditions
$$x\in E,\qquad r_-\le|v|\le r_+,\qquad xa_{t_j},ya_{t_j}\in E_{\mathrm{Lus}},$$
define good pairs of mass at least $7c/10$ for every $j$.
For large $j$, the middle point $y$ also lies in the interior of $\calB$.

Right translation by $a_{t_j}$ sends the unstable root parameter $r_ju$ to $u$, since
$$xu^+(r_ju)a_{t_j}=xa_{t_j}u^+(u).$$
By Equation~\eqref{eqn:leafwisecovariance}, pushing $\sigma_x^+$ and $\sigma_y^+$ forward under $u\mapsto r_j^{-1}u$ gives positive multiples of $\sigma_{xa_{t_j}}^+$ and $\sigma_{ya_{t_j}}^+$, respectively.
The final stable displacement is $r_j^2v$ so uniform continuity on $E_{\mathrm{Lus}}$ gives
$$\beta_{xa_{t_j}}-\beta_{ya_{t_j}}\longrightarrow0,$$
uniformly over good pairs against each compactly supported continuous test.

We now translate both endpoint images by $-f(x)$ and dilate by $r_j^{-1}$.
For corresponding points on the two plaques, this gives the maps
$$\Psi_{j,x}(u)=r_j^{-1}\bigl(f(xu^+(r_ju))-f(x)\bigr),\qquad \Psi_{j,y}(u)=r_j^{-1}\bigl(f(yu^+(r_ju))-f(x)\bigr).$$
These maps are defined while $xu^+(r_ju)$ and $yu^+(r_ju)$ remain in their respective plaques in $\calB$.
For every fixed bounded set of root parameters, both conditions hold for all sufficiently large $j$.
Uniformly on such sets, Taylor expansion gives
\begin{equation}\label{eqn:stablepairlinearization}
\Psi_{j,x}(u)=D_{U^+}f(x)u+O(r_j),\qquad \Psi_{j,y}(u)=D_{U^-}f(x)v+D_{U^+}f(x)u+O(r_j).
\end{equation}
By flow covariance, pushing the two plaque measures forward under $f$ and then under $z\mapsto r_j^{-1}(z-f(x))$ gives positive scalar multiples of
$$\alpha_{j,x}=(\Psi_{j,x})_*\beta_{xa_{t_j}},\qquad \alpha_{j,y}=(\Psi_{j,y})_*\beta_{ya_{t_j}},$$
with each root measure restricted to the domain of its map.
Thus $\beta_{xa_{t_j}}-\beta_{ya_{t_j}}$ weak-$*$ converges to zero while $\Psi_{j,y}(u)-\Psi_{j,x}(u)=D_{U^-}f(x)v+O(r_j)$ uniformly for $u$ in each fixed bounded set.

We next obtain a uniform positive lower bound on cross-products of test integrals for a measure and its translate.
For $x\in E$, $\beta\in\calK_\sigma$, and $r_-\le|v|\le r_+$, consider
$$\alpha=(D_{U^+}f(x))_*\beta,\qquad \alpha'=(T_{D_{U^-}f(x)v})_*\alpha=(T_{D_{U^-}f(x)v}\circ D_{U^+}f(x))_*\beta.$$
If these measures were proportional, then $(D_{U^+}f(x))^{-1}D_{U^-}f(x)v$ would be a nonzero projective period of $\beta$, contrary to our assumption.
We can therefore choose compactly supported smooth tests $\phi,\psi$ for which $\alpha(\phi)\alpha'(\psi)-\alpha(\psi)\alpha'(\phi)\ne0$.
After rescaling the tests to have sup norm at most one, continuity preserves a positive lower bound on a neighborhood of that parameter value.
The parameter space $E\times\calK_\sigma\times\{v\mid r_-\le|v|\le r_+\}$ is compact so finitely many such neighborhoods cover it.
Let $\calT$ consist of all the tests chosen for this finite subcover.
Taking the minimum of the corresponding positive lower bounds gives $\epsilon>0$ such that
\begin{equation}\label{eqn:stablepairseparation}
\max_{\phi,\psi\in\calT}\left|\alpha(\phi)\alpha'(\psi)-\alpha(\psi)\alpha'(\phi)\right|\ge\epsilon.
\end{equation}

For an actual good pair, take $\beta=\beta_{xa_{t_j}}$.
The inverse images of the test supports lie in a common compact set of root parameters on which the measures in $\calK_\sigma$ have uniformly bounded mass.
The errors in Equation~\eqref{eqn:stablepairlinearization} change these test integrals by quantities tending uniformly to zero.
Replacing $\beta_{ya_{t_j}}$ by $\beta_{xa_{t_j}}$ in these integrals also gives an error tending uniformly to zero.
Indeed, their difference weak-$*$ converges to zero uniformly over good pairs, and the pullbacks of the tests by $u\mapsto D_{U^+}f(x)u+D_{U^-}f(x)v$ form a compact family of continuous functions with common compact support.
Equation~\eqref{eqn:stablepairseparation} consequently gives that
\begin{equation}\label{eqn:actualcrossproductseparation}
\max_{\phi,\psi\in\calT}\left|\alpha_{j,x}(\phi)\alpha_{j,y}(\psi)-\alpha_{j,x}(\psi)\alpha_{j,y}(\phi)\right|\ge\epsilon/2,
\end{equation}
for all sufficiently large $j$.

To apply the entropy formula, we need to restrict both target plaque measures to the same coarse cube while keeping all the test supports inside it.
Small diameter alone does not ensure this, since a dyadic face may separate the centers or cut through a test support.
There is a positive constant $C_0$, independent of $j$ and the good pair, such that the ball centered at $f(x)$ of radius $C_0r_j$ contains both of the points $f(x)$, $f(y)$ and all the sets $f(x)+r_j\supp\phi$ for $\phi\in\calT$.
To choose such a cube, consider the finite family of nested translated grids
$$\calD_j^z=z+\calD_j,\qquad z\in\calZ_{\mathrm{gr}}=\{0,\tfrac13,\tfrac23\}^{n-2}.$$
At every dyadic side length at most one, the three sets of boundary points in each coordinate are offset by one-third of that side length.
Some choice of grid therefore places $f(x)$ at distance at least one-third of a side length from every face of its cube.
Choose $b_-$ large enough that, for every good pair and all large $j$, one grid has a cube $Q\in\calD_{j-b_-}^z$ containing both points and all the test supports above, with their distance from $\partial Q$ bounded below by a fixed positive multiple of $r_j$.
We call such a grid successful for the pair.
Since this grid may depend on the pair and on $j$, we will sum the entropy increments over all $z\in\calZ_{\mathrm{gr}}$.

We next express the common cube in the rescaled coordinates of our measures, restrict to this cube, and normalize.
The two cube masses may differ so we need an upper bound for both to retain a uniform lower bound on the cross-products.
For a successful grid, put $\widetilde{Q}=r_j^{-1}(Q-f(x))$.
The sets $\Psi_{j,x}^{-1}(\widetilde{Q})$ and $\Psi_{j,y}^{-1}(\widetilde{Q})$ lie in a fixed bounded set of root parameters.
Compactness of $\calK_\sigma$ therefore gives a common upper bound $M_0$ for $\alpha_{j,x}(\widetilde{Q})$ and $\alpha_{j,y}(\widetilde{Q})$.
These masses are positive by Equation~\eqref{eqn:actualcrossproductseparation}, since all test supports lie in $\widetilde{Q}$.
Put
$$\widehat{\alpha}=\frac{\alpha_{j,x}\vert_{\widetilde{Q}}}{\alpha_{j,x}(\widetilde{Q})},\qquad \widehat{\alpha}'=\frac{\alpha_{j,y}\vert_{\widetilde{Q}}}{\alpha_{j,y}(\widetilde{Q})}.$$
Restriction preserves the test integrals and normalization divides each cross-product by the product of the two masses, which is at most $M_0^2$.
In particular,
$$\max_{\phi,\psi\in\calT}\left|\widehat{\alpha}(\phi)\widehat{\alpha}'(\psi)-\widehat{\alpha}(\psi)\widehat{\alpha}'(\phi)\right|\ge\frac{\epsilon}{2M_0^2}.$$
For probabilities, this cross-product difference is at most $|\widehat{\alpha}(\phi)-\widehat{\alpha}'(\phi)|+|\widehat{\alpha}(\psi)-\widehat{\alpha}'(\psi)|$ and hence some test function satisfies
$$|\widehat{\alpha}(\phi)-\widehat{\alpha}'(\phi)|\ge\frac{\epsilon}{4M_0^2}.$$

We have separated the two normalized measures by a smooth test and now need to obtain a corresponding separation between their masses on a finite subdivision.
Subdivide $Q\in\calD_{j-b_-}^z$ into cubes $P\in\calD_{j+b_+}^z$, whose side length is $2^{-b_+}r_j$.
In the rescaled coordinates, these cubes subdivide $\widetilde{Q}$ into cubes of side length $2^{-b_+}$.
For each $\phi\in\calT$, choose one of its values on each rescaled cube and replace it there by this constant.
Since $\calT$ is finite and its tests are smooth, this changes either probability integral by $O(2^{-b_+})$, uniformly over the tests and good pairs.
The resulting step function has absolute value at most one and its integral is the sum of its values on the fine cubes weighted by their masses.
The difference of its two integrals is therefore bounded above in absolute value by the $\ell^1$-distance between the vectors of these masses.
Taking $b_+$ large enough gives a lower bound $\epsilon_0>0$ for this distance, independent of the good pair and of $j$.
We record these vectors in the original endpoint coordinates.

Let $Q_j^z(x)$ be the cube of $\calD_{j-b_-}^z$ containing $f(x)$.
Recall that $R(x)$ labels the unstable plaque through $x$, and write $\kappa_R$ for the conditional probability on the plaque labeled by $R$.
Apply $f$ to this conditional measure, restrict the resulting endpoint measure to $Q_j^z(x)$, and normalize it to a probability.
For a fine cube $P\in\calD_{j+b_+}^z$ contained in $Q_j^z(x)$, its mass is
$$\sfp_{j,x}^z(P)=\frac{\kappa_{R(x)}\bigl(f^{-1}(P)\bigr)}{\kappa_{R(x)}\bigl(f^{-1}(Q_j^z(x))\bigr)},$$
and $\sfp_{j,x}^z$ denotes the vector of these masses.
For a good pair in a successful grid, flow covariance shows that $\sfp_{j,x}^z$ and $\sfp_{j,y}^z$ record the masses assigned by $\widehat{\alpha}$ and $\widehat{\alpha}'$ to the corresponding fine cubes in $\widetilde{Q}$.
For large $j$, the portions of the two plaques mapped into $Q$ lie in the interior of $\calB$.
We have therefore proved that
\begin{equation}\label{eqn:stablepairlabelseparation}
Q_j^z(x)=Q_j^z(y),\qquad \|\sfp_{j,x}^z-\sfp_{j,y}^z\|_1\ge\epsilon_0,
\end{equation}
for every good pair in at least one grid, with $b_-,b_+,\epsilon_0$ independent of $j$.

We now convert this separation into an entropy increase by comparing the masses obtained from an individual plaque with those obtained by allowing all plaques to contribute.
Let $f_j^z(x)$ record the cube of $\calD_j^z$ containing $f(x)$, so that $Q_j^z(x)=f_{j-b_-}^z(x)$, and put
$$\calI_j^z=H(f_j^z)-H(f_j^z\mid R).$$
For a coarse cube $Q\in\calD_{j-b_-}^z$, restrict $\kappa$ to $f^{-1}(Q)$ and normalize, without fixing the plaque coordinate $R$.
The resulting probability that the endpoint lies in a fine cube $P\subset Q$ is
$$\sfq_{j,Q}^z(P)=\frac{\kappa\bigl(f^{-1}(P)\bigr)}{\kappa\bigl(f^{-1}(Q)\bigr)},$$ 
for $P\in\calD_{j+b_+}^z$, whenever the denominator is positive.
Write $\sfq_{j,Q}^z$ for the vector of these probabilities.
Thus $\sfp_{j,x}^z$ uses the conditional measure on one plaque, whereas $\sfq_{j,Q}^z$ averages over all plaques after conditioning on the same cube $Q$.
For $x\in\calB$, define the relative entropy with respect to this average
$$g_{j,z}(x)=D_{\mathrm{KL}}(\sfp_{j,x}^z\Vert\sfq_{j,Q_j^z(x)}^z),$$
and set it equal to zero outside $\calB$.
This compares the probabilities of the fine labels $f_{j+b_+}^z$ when both the coarse label $f_{j-b_-}^z$ and the plaque $R$ are specified with their probabilities when only the coarse label is specified.
The function is constant in the flow coordinate so Equation~\eqref{eqn:conditionalentropydifference} gives
\begin{equation}\label{eqn:stablepairentropycost}
\int g_{j,z}\,d\sigma=\sigma(\calB)\bigl(\calI_{j+b_+}^z-\calI_{j-b_-}^z\bigr).
\end{equation}
For a good pair in a successful grid, both separated vectors are compared with the same average $\sfq_{j,Q}^z$.
The triangle inequality puts at least one vector at $\ell^1$-distance at least $\epsilon_0/2$ from this average so Equation~\eqref{eqn:strictconcavitygap} gives a uniform positive lower bound for the sum of their relative entropies.
Since the number of fine cubes is fixed, there is a uniform $\epsilon_1>0$ such that
$$g_{j,z}(x)+g_{j,z}(y)\ge\epsilon_1,$$
independent of $j$.

The sum $\sum_z(g_{j,z}(x)+g_{j,z}(y))$ is at least $\epsilon_1$ on every good pair and is nonnegative on every pair.
Since both marginals of $\Pi_j$ equal $\sigma$ and the good pairs have mass at least $7c/10$, we obtain
$$2\sum_z\int g_{j,z}\,d\sigma=\int\sum_z\bigl(g_{j,z}(x)+g_{j,z}(y)\bigr)\,d\Pi_j(x,y)\ge\frac{7c\epsilon_1}{10}.$$
Equation~\eqref{eqn:stablepairentropycost} therefore yields
\begin{equation}\label{eqn:stablepairpositivegap}
\sum_{z\in\calZ_{\mathrm{gr}}}\bigl(\calI_{j+b_+}^z-\calI_{j-b_-}^z\bigr)\ge\frac{7c\epsilon_1}{20\sigma(\calB)}=c_*>0,
\end{equation}
for sufficiently large $j$.

It remains to compare the translated grids with the original grid, where the sublinear bound was proved.
At the same scale, a cube of one translated grid meets at most $2^{n-2}$ cubes of another.
The entropy of either grid label conditioned on the other is at most $n-2$ and the same bound holds after also conditioning on $R$.
The entropy chain rule therefore gives $|\calI_j^z-\calI_j|\le2(n-2)$.
Consequently,
$$\calJ_j=\sum_{z\in\calZ_{\mathrm{gr}}}\calI_j^z=o(j).$$
However, Equation~\eqref{eqn:stablepairpositivegap} says that $\calJ_{j+b_+}-\calJ_{j-b_-}\ge c_*$ for all large $j$.
Writing $L=b_-+b_+$ and iterating gives $\calJ_{k+mL}\ge\calJ_k+mc_*$ for every sufficiently large $k$ and every integer $m\ge1$, contradicting $\calJ_j=o(j)$.
\end{proof}

\subsection{Identification of the limiting current}\label{sec:surfaceidentification}

Let $h$ be the $\Gamma$-invariant entropy function of Proposition~\ref{prop:entropyindependence} and put
$$\overline{\eta}_+=\overline{\eta}_\infty\vert_{\{h>0\}}.$$
Let $\eta_+$ be its lift obtained by adding Haar probability measure in the compact normal fibers.
Proposition~\ref{prop:positiveentropycomponent} implies that $\eta_+$ is nonzero.

To obtain the homogeneous decomposition, we need to pass from the projective periods of Proposition~\ref{prop:surfaceperiods} to invariance under a nontrivial unipotent subgroup.
This follows the general leafwise-measure strategy of \cite[Lem~3.10, \S8]{ELRankOne}.
We must then verify that the intersection geodesic belongs to the resulting homogeneous carrier.

\begin{lemma}\label{lem:surfaceboundaryhomogeneous}
For every boundary lift $Z$, the current $\eta_+\vert_{U_Z}$ admits the homogeneous decomposition of Proposition~\ref{prop:boundaryhomogeneous}.
Every nonzero exceptional term $\eta_{Z,B}$ in this decomposition is associated to a subspace $B$ of dimension at least two.
\end{lemma}

\begin{proof}
Fix a compact angle set $I$ for which the boundary measure of $\eta_+$ has positive mass and normalize it to a probability $\sigma$.
Proposition~\ref{prop:surfaceperiods} makes $\calL_x$ nontrivial for $\sigma$-almost every $x$.
Write $E_x$ for its identity component, which is a vector subspace of $\bbR^{n-2}$.
Equation~\eqref{eqn:leafwisecovariance} gives
$$\calL_{xa_t}=e^t\calL_x,\qquad E_{xa_t}=E_x.$$
If the discrete quotient $\calL_x/E_x$ is nontrivial, let $b(x)$ be the shortest nonzero length in this quotient, identified with $E_x^\perp$.
Then $b(xa_t)=e^tb(x)$, so no positive-measure set on which $b$ lies between two fixed positive bounds can be recurrent.
Finiteness and $A$-invariance therefore give $\calL_x=E_x$ almost everywhere.
The projective multipliers are equal to one by \cite[Lem~3.10]{ELRankOne}.

Disintegrate $\sigma$ over the angle parameter and the $A$-invariant map $x\mapsto E_x$.
By \cite[Lem~3.8]{ELRankOne}, the resulting probabilities retain the same complete unstable leafwise measures up to scale.
For a fixed subspace $E$, these measures are invariant under translations by $E$, hence local disintegration makes the probability invariant under
$$U_E^+=\{u^+(v)\mid v\in E\}.$$
It is also $A$-invariant and $E\ne\{0\}$.
Ratner's theorem \cite{Ratner2} now makes its $U_E^+$-ergodic components homogeneous.
Since $\Delta_Z\backslash H_Z$ is compact, the subgroup description in \cite[Cor~3.8, Prop~5.3]{LO} gives a noncompact factor conjugate to $\SO_0(V_B)$, where $B\subseteq Z$ has dimension at least two.
The remaining factors lie in its compact centralizer.

It remains to show that the intersection geodesic $C\cap Z$ lies in $B$.
In the homogeneous decomposition, retain pairs $(h,B)$, where $h\in H_Z$ is a frame and $B\subseteq Z$ is the corresponding geometric lift, and take the quotient by the diagonal action of $\Delta_Z$.
Uniqueness of the ergodic decomposition makes the resulting measure on pairs finite and $A$-invariant, with $[h,B]a_t=[ha_t,B]$.
The endpoint fixed by $U_E^+$ belongs to $\partial B$.
Let $q_t$ be the basepoint of the frame $ha_{-t}$, so that $q_t$ moves along $C\cap Z$ toward the endpoint fixed by $U_E^+$.
The hyperboloid model then gives
$$\sinh d(q_t,B)=e^{-t}\sinh d(q_0,B).$$
Recurrence of this measure forces the distance to vanish so the entire intersection geodesic is contained in $B$.

There are only countably many such closed geodesic submanifolds and lifts by the finite-volume case of \cite[Thm~1.7(2)]{LO}.
Grouping the homogeneous components by $B$, while retaining their compact parameters, now gives the locally $\SO_0(V_B)$-invariant currents supported on $p^{-1}(\calE_{Z,B})$ exactly as in Proposition~\ref{prop:boundaryhomogeneous}.
The term with $B=Z$ is $H_Z$-invariant.
Exhausting the angle ball by compact sets proves the lemma.
\end{proof}

\begin{proof}[Proof of Theorem~\ref{thm:surfacehaar}]
We apply the carrier argument of Section~\ref{sec:intermediateclosures} to $\overline{\eta}_+$.
Lemma~\ref{lem:surfaceboundaryhomogeneous} supplies the two homogeneous boundary decompositions used there.
For $m=2$, the common positive-definite intersection in the proof of Lemma~\ref{lem:algebraiccarrier} has dimension $m-1=1$, which still suffices for Equation~\eqref{eqn:boundarygeneration}.
Thus every minimal algebraic carrier is $\calV_F$ and Proposition~\ref{prop:closedcarrier} makes its restricted current intrinsic Haar measure and shows that $F$ projects to a properly immersed closed geodesic submanifold.

We exclude these proper carriers using Proposition~\ref{prop:noequivprob}.
By maximality, $\overline{\eta}_i^{\calC}(\calV_F)=0$ for all sufficiently large $i$, so the blow-up construction in Proposition~\ref{prop:incidencenonconcentration} gives a $\Gamma_F$-equivariant probability field over $\overline{\eta}_+\vert_{\calV_F}$ in $\bbP\Hom(V_C,V_F^\perp)\subseteq\bbP T_C\overline{\calC}$.
Distinct translates of this minimal carrier have null intersections.
Transporting the field by $\Gamma$ therefore gives a well-defined equivariant field over the nonzero invariant current obtained by restricting $\overline{\eta}_+$ to their union, contradicting Proposition~\ref{prop:noequivprob}.
Consequently every proper algebraic subset has zero $\overline{\eta}_+$-measure.

The exceptional terms in Lemma~\ref{lem:surfaceboundaryhomogeneous} are supported on such subsets and hence vanish.
Thus $\eta_+\vert_{U_Z}$ is $H_Z$-invariant for every $Z$.
The local invariance and connectedness arguments in the proof of Theorem~\ref{thm:higherdimensionhaar} now give $\eta_+=b\lambda$ for some $b>0$.
Finally, the complementary current $\overline{\eta}_0=\overline{\eta}_\infty\vert_{\{h=0\}}$ is invariant and compatibility with invariant restriction makes all its boundary measures have zero entropy.
Proposition~\ref{prop:positiveentropycomponent} therefore gives $\overline{\eta}_0=0$.
Lifting through the compact normal fibers yields $\eta_\infty=b\lambda$, proving the theorem with $c=b^{-1}$.
\end{proof}

Together, Theorems~\ref{thm:higherdimensionhaar} and \ref{thm:surfacehaar} prove Theorem~\ref{thm:measurerigid} for every $2\le m<n$.

\section{Classification of \texorpdfstring{$W_m$}{W(m)}-invariant Radon measures on \texorpdfstring{$\calF^*$}{F*}}
\label{sec:generalratner}

We now apply the boundary and carrier arguments of Sections~\ref{sec:higherdimensionabsolute} and \ref{sec:surfaceabsolute} to an arbitrary ergodic Radon measure.
Without the maximality assumption, proper geodesic carriers may have positive mass, and we must also recover the measure in the normal-frame directions.
We recall the statement of the measure classification for the reader's convenience.

\generalmeasurerigidity*

For a geodesic $r$-plane $F$, write $\lambda_F$ for the intrinsic smooth measure on $\calV_F$, the family of underlying $m$-planes contained in $F$.
The preceding arguments will identify the underlying current with this measure on a carrier and its translates.
To recover the original frame measure, we will average over the compact normal directions and use ergodicity of the resulting Haar measure.
The following extension of Proposition~\ref{prop:lambdaergodicity} supplies this ergodicity while retaining the normal holonomy.

\begin{lemma}\label{lem:selectedorbitergodicity}
Let $F=gF_r$, where $F_r=\bbH^n\cap E_r$ is the standard $r$-plane and $m\le r\le n$.
Suppose that $F$ meets $\interior(D_\Gamma)$ and projects to a properly immersed closed geodesic submanifold.
If $r>m$, suppose also that $\lambda_F$-almost every plane in $\calV_F\cap\overline{\calC}^*$ crosses two distinct boundary lifts of its convex core.
If $\Delta=g^{-1}\Gamma_Fg<S_r$ and $L_F=\overline{\Delta W_r}$, then $\Gamma gL_F$ is closed in $\Gamma\backslash G$ and $m_{\Gamma gL_F}\vert_{\calF^*}$ is $W_m$-ergodic.
\end{lemma}

\begin{proof}
Let $\pi:S_r\to Q=S_r/W_r$ and put $D=\overline{\pi(\Delta)}$, so that $L_F=\pi^{-1}(D)$.
Since $F$ projects to a properly immersed submanifold, $\Gamma gL_F$ is closed in $\Gamma\backslash G$.

Let $C_0=\bbH^n\cap E_m\subseteq F_r$ be the standard $m$-plane defined in Section~\ref{sec:background}.
Split $S_r=W_r\rtimes Q$ using a fixed spatial reflection in $C_0$ to compensate normal transformations of determinant $-1$.
The preimage of $\calF^*$ in $L_F$ is $\calO\times D$, with product Haar measure, where
$$\calO=\{s\in W_r\mid gsC_0\cap\interior(D_\Gamma)\ne\emptyset\}.$$
Let $f$ be a bounded $W_m$-invariant function on $\Delta\backslash(\calO\times D)$ and lift it to $\calO\times D$.
We first show that $f(s,q)=\phi(q)$ almost everywhere.

For $r=m$, this immediately follows from right $W_m$-invariance since $\calO=W_m$.
Therefore suppose $r>m$.
The intrinsic core $F_r\cap g^{-1}D_\Gamma$ has compact quotient and totally geodesic boundary.
Put $\Delta^+=\Delta\cap S_r^\circ$ and let $\Delta_B^+$ be the stabilizer of an intrinsic boundary lift $B$.
Its intrinsic projection is a cocompact lattice in $H_B=\SO_0(r-1,1)$ and $X_B=\Delta_B^+\backslash(H_B\times D)$ is compact.
In the intrinsic crossing coordinates $H_B/J\times\bbB^{r-m}$ of Equation~\eqref{eqn:equivariant} with $J=\SO_0(m-1,1)$, the function $f$ gives a bounded right $J$-invariant function on $X_B$ for almost every angle.
Since $J$ is noncompact, the Howe--Moore theorem \cite{HoweMoore} implies that the function $f$ is invariant under $H_B$ with $q$ fixed.

We may now apply the boundary comparison in the proof of Proposition~\ref{prop:lambdaergodicity} inside $F_r$, keeping $q$ fixed.
The generation, connectedness, and overlap arguments there make $f$ constant in the intrinsic normally framed plane on the union of the double-crossing sets.
This union is conull by hypothesis, so Fubini's theorem gives $f(s,q)=\phi(q)$ almost everywhere.
Finally, $\Delta$-invariance and density of $\pi(\Delta)$ in $D$ force $\phi$ to be constant.
\end{proof}

\begin{proof}[Proof of Theorem~\ref{thm:generalmeasurerigidity}]
Let $\eta$ be the transverse current of $\mu$ on $\calC^*$ and put $\varrho=p_*\eta$.
The local product correspondence of Section~\ref{sec:limitingmeasure} and properness of $p$ make $\varrho$ a nonzero $\Gamma$-invariant, $\Gamma$-ergodic Radon current.
We first identify its carrier.
If $\varrho$ is atomic, it is supported on a single locally finite $\Gamma$-orbit $\Gamma C$ and Corollary~\ref{cor:planeclosures} gives the properly immersed carrier $F=C$ with $r=m$.
In this case, we interpret $\lambda_F$ as the point mass $\delta_F$.

Suppose instead that $\varrho$ is non-atomic.
The invariant set of planes crossing two distinct boundary lifts is either null or conull.
In the null case, Lemma~\ref{lem:aeboundarylift} and ergodicity place the support of the current in the $\Gamma$-saturation of the planes in one compact closed geodesic submanifold, whose frames form a compact subset of $\calF^*$.
Thus $\mu$ is finite and we may extend it by zero and normalize it to a probability.
Ratner's theorem applies to $W_m$ in the form of \cite[Thm~1.1]{ShahUnipotentMeasures}.
The subgroup description in \cite[Lem~3.2(1)]{BFMS}, including the compact conjugation, therefore gives Equation~\eqref{eqn:Ratnerstyle}.

In the conull case, the boundary decomposition follows from Proposition~\ref{prop:boundaryhomogeneous} when $m\ge3$.
For $m=2$, the projective-field and entropy arguments of Proposition~\ref{prop:noequivprob} and Propositions~\ref{prop:entropyindependence}--\ref{prop:surfaceperiods} require non-atomicity, the two-crossing property, and finiteness of the transferred boundary measures on compact angle sets.
These hold for $\varrho$ and its invariant subcurrents, with finiteness supplied by Lemma~\ref{lem:banana}.
Proposition~\ref{prop:positiveentropycomponent} and ergodicity make the invariant entropy function positive almost everywhere, hence the proof of Lemma~\ref{lem:surfaceboundaryhomogeneous} gives the same homogeneous boundary decomposition.

The proofs of Lemma~\ref{lem:algebraiccarrier} and Proposition~\ref{prop:closedcarrier} now identify a minimal positive-mass carrier with $\calV_F$, where $F$ projects to a properly immersed closed geodesic submanifold, and hence give
$$\varrho\vert_{\calV_F}=b\lambda_F\vert_{\calV_F\cap\overline{\calC}^*}.$$
Maximality of an approximating sequence was used only to exclude proper carriers and is not needed here.
If every proper algebraic subset is null, the exceptional boundary terms vanish and the local invariance argument in the proof of Theorem~\ref{thm:higherdimensionhaar} gives $F=\bbH^n$.
Ergodicity makes the union of the translated carriers conull and distinct translated carriers have null intersection.
Including the atomic case above and using $\gamma_*\lambda_F=\lambda_{\gamma F}$, we therefore have that
\begin{equation}\label{eqn:generalcarriercurrent}
\varrho=b\sum_{\gamma\in\Gamma/\Gamma_F}
\lambda_{\gamma F}\vert_{\calV_{\gamma F}\cap\overline{\calC}^*},
\end{equation}
whose carrier $F$ satisfies the hypotheses of Lemma~\ref{lem:selectedorbitergodicity}.

Write $F=gF_r$ and let $L_F$ be as in Lemma~\ref{lem:selectedorbitergodicity}.
The normal directions in the carrier define a $W_m$-invariant map to $\Gr_{r-m}(\bbR^{n-m})$ and therefore ergodicity yields a single element of $C_G(W_m)$ which puts these directions in standard position.
After right translating $\mu$ by some $k_0^{-1}\in C_G(W_m)$, the translated measure $\mu_0$ is then carried by $\Gamma gS_r\cap\calF^*$.
The map $\Gamma gn\mapsto L_Fn\in L_F\backslash S_r$ is $W_m$-invariant so ergodicity again lets us choose $g$ such that $\mu_0$ is carried by $xL\cap\calF^*$, where $x=\Gamma g$ and $L=\overline{(g^{-1}\Gamma_Fg)W_r}$.
Its underlying current remains $\varrho$ and consequently Lemma~\ref{lem:selectedorbitergodicity} shows that $xL$ is closed and $m_{xL}\vert_{\calF^*}$ is $W_m$-ergodic.

Put $T_L=L\cap S_m\cap K$.
This compact group normalizes $W_m$ and acts transitively on each normal-frame fiber within the orbit since $L\cap S_m=W_mT_L$.
Averaging over $T_L$ therefore replaces the conditional probabilities on these fibers by their invariant probabilities, while continuing to preserve the underlying current $\varrho$.
Equation~\eqref{eqn:generalcarriercurrent} and Haar measure along the $W_m$-fibers then give that
$$\mu_{\mathrm{av}}=\int_{T_L}(\sfR_t)_*\mu_0\,dt=c\,m_{xL}\vert_{\calF^*},$$
for some $c>0$, where $\sfR_t$ is right translation and $dt$ is Haar probability.
Since $\mu_{\mathrm{av}}$ is $W_m$-ergodic, it spans an extreme ray of $W_m$-invariant Radon measures.
The averaging identity therefore implies that $(\sfR_t)_*\mu_0$ is proportional to $\mu_{\mathrm{av}}$ for almost every $t\in T_L$.
Both measures have underlying current $\varrho$ so the proportionality constant is one and continuity in $t$ then gives $\mu_0=\mu_{\mathrm{av}}$.
Finally, setting $y=xk_0$ and $H=k_0^{-1}Lk_0$, we obtain the closed orbit $yH=xLk_0$ and
$$\mu=c\,m_{yH}\vert_{\calF^*},\qquad W_r\le k_0Hk_0^{-1}\le S_r,$$
completing the proof.
\end{proof}

\section{The Margulis criterion for arithmeticity and construction of the lifted measure}\label{sec:marguliscriterion}

In this section, we first recall the form of Margulis' arithmeticity criterion that we will use and then construct a measure on the suspension space associated to each representation arising from that criterion.
Disintegrating this measure will produce the $\Gamma$-equivariant, $\lambda$-measurable map used in Theorem~\ref{thm:superrigid}.
The contents of this section are well known and follow the strategy of \cite[\S 2,~3]{BFMS}.
For completeness, we include the construction in order to specify the measure and the equivariance conventions needed in Section~\ref{sec:AREA}.

\subsection{Margulis' arithmeticity criterion}\label{sec:arithmeticitycriterion}

We first recall the arithmetic construction.
Let $\ell$ be a number field with ring of integers $\calO_\ell$ and let $\bfG$ be a connected semisimple $\ell$-algebraic group.
A faithful $\ell$-rational representation $\iota:\bfG\to\bfGL_V$ defines the group of integral points via the formula
$$\bfG(\calO_\ell)=\iota^{-1}\left(\iota(\bfG(\ell))\cap\bfGL_V(\calO_\ell)\right).$$
By work of Borel and Harish-Chandra \cite{BHC}, the diagonal image of $\bfG(\calO_\ell)$ is a lattice in the product $\prod_{v\mid\infty}\bfG(\ell_v)$, where the notation means that the product is taken over the archimedean places of $\ell$.
Suppose that one factor has identity component isogenous to $G$ and that every other factor is compact.
Projection to the noncompact factor then gives a lattice in $G$.
A lattice in $G$ is called \emph{arithmetic} if, up to conjugacy and commensurability, it is obtained in this way.
This definition is independent of the faithful $\ell$-rational representation used to define $\bfG(\calO_\ell)$ \cite[\S2.3]{PrasadRapinchuk}.

We now let $\Lambda<G$ be a finitely generated Zariski dense discrete subgroup and assume for the moment that its adjoint trace field
$$\ell=\bbQ\bigl(\tr(\Ad(\gamma)):\gamma\in\Lambda\bigr)\subset\bbR,$$
is a number field.
Let $v_{\infty}:\ell\to\bbR$ denote the distinguished real embedding induced by the given inclusion $\Lambda<G$.
Identify $G$ with the identity component of its adjoint real algebraic group.
Work of Vinberg \cite{Vinberg} gives a connected semisimple adjoint $\ell$-group $\bfG$ such that $\bfG(\ell_{v_\infty})^\circ\cong G$ and $\Lambda<\bfG(\ell)$.
We refer to \cite[\S2]{BFMS} for this reduction.

For every place $v$ of $\ell$, let $\ell_v$ denote the corresponding local field.
The inclusion $\Lambda<\bfG(\ell)$ gives a family of natural representations indexed by the places
$$\rho_v:\Lambda\longrightarrow\bfG(\ell)\longrightarrow\bfG(\ell_v).$$
If $\rho_v(\Lambda)$ is unbounded, put $k=\ell_v$ and let $\bfH=\bfG\times_\ell k$.
We retain this full connected adjoint semisimple target, as in \cite[\S3.3]{BFMS}, and write $\rho:\Lambda\longrightarrow\bfH(k)$ for the resulting Zariski dense, unbounded representation.

Recall that $\rho$ is called \emph{superrigid} if either $\rho(\Lambda)$ is relatively compact or, after conjugating in $\bfH(k)$, there is a continuous homomorphism
$$\widetilde{\rho}:G\longrightarrow \bfH(k),$$
whose restriction to $\Lambda$ agrees with $\rho$.
Margulis' arithmeticity criterion, in the form needed here, is the following; see \cite[Ch~IX]{MargulisBook} and \cite[\S3.2]{BFMS}.

\begin{thm}[Margulis' arithmeticity criterion]\label{thm:marguliscriterion}
The group $\Lambda$ is contained in an arithmetic lattice of $G$ if and only if $\rho_v(\Lambda)$ is relatively compact for every place $v\ne v_{\infty}$.
In particular, it is enough to prove that the representations arising from the places of $\ell$ are superrigid.
\end{thm}

To explain the final assertion, suppose that $v\ne v_{\infty}$ and that the unbounded representation $\rho_v$ extends continuously to $G$.
The reduction in \cite[\S3.2]{BFMS} and work of Tits show that such an extension induces a continuous field embedding $\ell_{v_\infty}\hookrightarrow\ell_v$ which agrees with $v$ on every adjoint trace of an element of $\Lambda$ (see \cite[Rmk~1.8.2(III)]{MargulisBook}).
Since these adjoint traces generate $\ell$, the embeddings $v$ and $v_{\infty}$ must agree.
This contradicts the assumption that $v\ne v_{\infty}$.
Consequently, superrigidity forces $\rho_v(\Lambda)$ to be relatively compact at every nondistinguished place.
The construction of $\bfG(\calO_\ell)$ in the first paragraph then produces the requisite arithmetic lattice containing $\Lambda$.
We emphasize that this is a criterion for containment in an arithmetic lattice.
The group $\Lambda$ need not itself have finite covolume.

We now apply the preceding discussion with $\Lambda=\Gamma$ in our setting and verify the hypothesis that the adjoint trace field is a number field.
Indeed, since $\partial\core(Y_\infty)$ is totally geodesic, doubling $\core(Y_\infty)$ along its boundary produces a closed hyperbolic $n$-manifold with fundamental group $\Lambda_{\mathrm{dbl}}<G$.
Since this double is a finite-volume hyperbolic manifold, the adjoint trace field is a number field (see \cite[\S2]{BFMS} for references).
The natural inclusion of the convex core into its double realizes $\Lambda$ as a subgroup and therefore its adjoint trace field is a subfield of this number field, hence itself a number field as well.
This verifies the hypothesis.

\subsection{Suspension spaces and lifts of measures to them}\label{sec:suspension}

Let $X$ be a compact metrizable $\bfH(k)$-space.
Then given a representation $\rho:\Gamma\to \bfH(k)$, we define the \emph{suspension space} by
$$X_\rho=\Gamma\backslash(G\times X),$$
where $\Gamma$ acts diagonally on the left via
$$\gamma\cdot(g,x)=(\gamma g,\rho(\gamma)x).$$
We write $[g,x]$ for the image of $(g,x)$ in $X_\rho$.
Consequently, the projection
\begin{align*}
\pi_\rho:X_\rho&\longrightarrow\Gamma\backslash G,\\
 [g,x]&\longmapsto\Gamma g,
\end{align*}
makes $X_\rho$ an associated bundle with fiber $X$.
Right translation on the first coordinate defines a right $G$-action $[g,x]s=[gs,x]$ for which $\pi_\rho$ is $G$-equivariant.
In this way, the suspension space encodes the representation $\rho$ while simultaneously allowing for actions by subgroups of $G$.

Suppose that $\mu^\rho$ is a right $W$-invariant Radon measure on $X_\rho$ and that $(\pi_\rho)_*\mu^\rho=\mu$ for a right $W$-invariant Radon measure $\mu$ on $\Gamma\backslash G$.
Disintegrating $\mu^\rho$ over the bundle projection gives a measurable family of fiberwise probability measures on $X$.
After lifting this family to $G$, it may be represented by a measurable map
$$\Phi:G\longrightarrow\Prob(X),$$
satisfying
$$\Phi(\gamma gw)=\rho(\gamma)_*\Phi(g),$$
for every $\gamma\in\Gamma$, every $w\in W$, and almost every $g$.
It therefore descends to a measurable $\Gamma$-equivariant map
$$G/W\longrightarrow\Prob(X).$$
The purpose of the remainder of this section is to construct the measure $\mu^\rho$ with $X=\bbP(V)(k)$.

Fix the unbounded local representation $\rho:\Gamma\to\bfH(k)$ arising from $\rho_v$ at a place of $\ell$ as in the previous subsection.
Recall from Section~\ref{sec:limitingmeasure} that
$$\calF_i=\Gamma g_iS_m\subset\Gamma\backslash G,\qquad
\Delta_i=\Gamma\cap g_iS_mg_i^{-1},$$
and that the right $W$-invariant Radon measures $\mu_i$ on $\calF_i$ weak-$*$ converge to $\mu_\infty$.
Let $\bfL_i<\bfH$ be the Zariski closure of $\rho(\Delta_i)$.
Since $\Delta_i$ is contained in a conjugate of the proper algebraic stabilizer $S_m<G$, the same extension-of-scalars argument as in \cite[Prop~3.4]{BFMS} shows that $\bfL_i$ is a proper $k$-subgroup of $\bfH$.
Here the definition of $\Delta_i$ includes the compact normal factor, so the point constructed below is fixed by the full stabilizer.

After passing to a subsequence, assume that $d=\dim\bfL_i$ is independent of $i$.
Put $\frakh=\Lie(\bfH)$ and $\frakl_i=\Lie(\bfL_i)$.
Then the line
$$\left[\bigwedge^d\frakl_i\right]\in\bbP\left(\bigwedge^d\frakh\right)(k),$$
is fixed by $\rho(\Delta_i)$.
By complete reducibility, and after taking another subsequence, there is a common nontrivial irreducible summand $V\subseteq\bigwedge^d\frakh$ on which the projection of these lines is nonzero.
We denote the corresponding points of $\bbP(V)(k)$ by $x_i$ and hence $\rho(\Delta_i)x_i=x_i$.
As in \cite[Prop~3.4]{BFMS}, we may choose $V$ faithful.
When $n=3$ and $\bfH$ is not $k$-simple, we use the tensor product construction at the end of that proof.

Take $X=\bbP(V)(k)$ and form the suspension $X_\rho$.
The fixed-point identity above makes
\begin{align*}
\sigma_i:\calF_i&\longrightarrow X_\rho,\\
\Gamma g_is&\longmapsto [g_is,x_i],\qquad s\in S_m,
\end{align*}
a well-defined measurable section of $\pi_\rho$ over $\calF_i$.
Since right translation by $S_m$ leaves the second coordinate $x_i$ unchanged, the section $\sigma_i$ is right $S_m$-equivariant.
Consequently the lifted measure $\mu_i^\rho=(\sigma_i)_*\mu_i$ is right $S_m$-invariant, and in particular right $W$-invariant.
Moreover, $(\pi_\rho)_*\mu_i^\rho=\mu_i$.

Since $\bbP(V)(k)$ is compact, $\pi_\rho$ is proper.
The uniform bounds for the measures $\mu_i$ on compact subsets of $\Gamma\backslash G$ therefore give uniform bounds for the measures $\mu_i^\rho$ on compact subsets of $X_\rho$.
After passing to a subsequence, we obtain a weak-$*$ limit
$$\mu_i^\rho\longrightarrow\mu_\infty^\rho.$$
The limit is right $W$-invariant and satisfies
$$(\pi_\rho)_*\mu_\infty^\rho=\mu_\infty.$$
In particular, $\mu_\infty^\rho$ is nonzero.

As above, we obtain a $W$-invariant map $\Phi$ to probability measures on $X$.
Since $\Phi$ is essentially constant on the right $W$-leaves, it descends transversely to a $\Gamma$-equivariant map
$$\psi:\calC^*\longrightarrow\Prob(\bbP(V)(k)),$$
which is measurable with respect to the transverse current $\eta_\infty$.
By Theorem~\ref{thm:measurerigid}, the measures $\eta_\infty$ and $\lambda$ differ by a positive scalar.
Thus $\psi$ is $\lambda$-measurable and, after changing it on a null set, gives the equivariant map required by Theorem~\ref{thm:superrigid}.

We have therefore associated to every unbounded local representation $\rho_v$ a suspension space with compact fiber, a right $W$-invariant lifted measure, and the equivariant probability-measure map that will be the starting point of Section~\ref{sec:AREA}.

\section{Algebraic representations of ergodic actions}\label{sec:AREA}

In this section, we take the data of a $\Gamma$-equivariant measurable map
$$\calC^*\to\Prob(\bbP(V)(k)),$$
as constructed in Section~\ref{sec:suspension}, and convert it to a superrigidity statement.
We follow the broad strategy of \cite{BFMS}, but the measures arising here require several modifications to the finite-volume argument.
Section~\ref{sec:insufficient} explains where that argument breaks down and describes how we overcome the obstruction by transferring algebraic representations between measure classes.
Before turning to these issues, we recall the background we need from Bader and Furman's theory of Algebraic Representations of Ergodic Actions (AREA).
Throughout we continue to fix a local field $k$ of characteristic $0$ and $W=W_m$ with $2\le m<n$.
For algebraic quotients, we use $\bfH/\bfL(k)$ as shorthand for $(\bfH/\bfL)(k)$, and similarly for quotients of algebraic normalizers.
These expressions denote $k$-points of quotient varieties, whereas a quotient of groups of $k$-points is written $\bfH(k)/C$.
The dimension $m$ enters these arguments only through the subgroup $U=N\cap W$, the inclusions $\langle A,w\rangle<W$ and $AU<W$, and the inequality $\delta_\Gamma+\dim U>n-1$.

For the reader familiar with the work of Bader and Furman, we follow their categorical approach, proving general statements about AREA before specializing to the setting at hand.
This keeps the properties of the general machinery distinct from the geometric features we will need to exploit in our setting.

\subsection{Background on AREA}

We give a fairly complete account of the parts of AREA needed below, both to fix the language and to make the measure hypotheses explicit.
In the Haar-measure application of \cite{BFMS}, every subgroup preserves the same measure, so these hypotheses are automatic.

Here we must allow actions that preserve only a measure class and compare algebraic representations defined with respect to different classes.
Most of the arguments of Bader and Furman carry over with nonsingularity in place of invariance.
We indicate where this suffices and where the comparison of the BMS and BR classes requires an additional argument.

\begin{defn}
Fix a measured space $(X,[\mu])$ with commuting nonsingular left $\Gamma$- and right $S$-actions, where $S$ is a locally compact second countable (lcsc) group, and fix the representation $\rho:\Gamma\to\bfH(k)$.
An $S$-algebraic representation consists of the following data.
\begin{itemize}
\item a $k$-algebraic group $\bfI$,
\item a $k$-algebraic variety $\bfV$, with commuting left $\bfH$- and faithful right $\bfI$-actions,
\item a continuous, Zariski dense homomorphism $\tau:S\to\bfI(k)$, and
\item a measurable map $\phi:X\to\bfV(k)$ satisfying
\begin{equation}\label{eq:object}
 \phi(\gamma xt)=\rho(\gamma)\phi(x)\tau(t),
\end{equation}
for every $\gamma\in\Gamma$, every $t\in S$, and $[\mu]$-almost every $x\in X$.
\end{itemize}
\end{defn}

We frequently suppress the data by saying that $\phi$ or $\phi:X\to\bfV(k)$ is an $S$-algebraic representation of $(X,[\mu])$.
We may also define a morphism between two algebraic representations $\phi:X\to\bfV(k)$ and $\phi':X\to\bfV'(k)$.
Indeed, let $\bfI_{\bfV,\bfV'}$ be the Zariski closure of $\tau\times \tau'$ in $\bfI\times \bfI'$.
This group acts on $\bfV$ and $\bfV'$ through the coordinate projections.
A morphism of $S$-algebraic representations is a $\bfH\times\bfI_{\bfV,\bfV'}$-equivariant $k$-algebraic map $\alpha:\bfV\to\bfV'$ such that
$$\xymatrix{
X\ar[r]^-{\phi}\ar[dr]_-{\phi'}&\bfV(k)\ar[d]^-\alpha\\
&\bfV'(k),}$$
commutes $[\mu]$-almost everywhere.
In particular, we obtain a category.

\begin{defn}
Fix a measured space $(X,[\mu])$.
Define $\AR_S(X,[\mu])$ to be the category whose objects are $S$-algebraic representations and whose morphisms are those described above.
This category only depends on the measure class of $\mu$.
\end{defn}

There is a particular type of $S$-algebraic representation that will be of interest in the sequel.
A \emph{coset $S$-algebraic representation} is an $S$-algebraic representation for which $\bfV(k)=\bfH/\bfL(k)$ for some closed $k$-algebraic subgroup $\bfL\le\bfH$ and $\bfH$ acts by translation on the left.
We note that in the specific case of a coset $S$-algebraic representation, we may identify $\bfI$ with a $k$-algebraic subgroup of $\bfN_{\bfH}(\bfL)/\bfL$, since this is the group of $\bfH$-equivariant algebraic automorphisms of $\bfH/\bfL$ and the right $\bfI$-action is faithful.

An \emph{initial object} in the category $\AR_S(X,[\mu])$ is simply an initial object in the usual category-theoretic sense.
More explicitly, an initial object is an $S$-algebraic representation $\phi_0:X\to \bfV_0(k)$ such that for every other object $\phi:X\to \bfV(k)$, there exists a unique morphism $\alpha$ from $\phi_0$ to $\phi$.
When this object is additionally a coset algebraic representation, we will call it an \emph{$S$-gate}.

Gates need not exist, as we will see in Section~\ref{sec:insufficient}.
Bader and Furman give the following criterion for their existence.
To this end, we say that the $S$-space $\Gamma\backslash X$ has only trivial algebraic representations if every $S$-algebraic representation of the trivial homomorphism $\rho_{\mathrm{triv}}:\Gamma\to\bfH(k)$ is essentially constant.
Equivalently, every measurable map $\phi:\Gamma\backslash X\to\bfV(k)$ satisfying
$$\phi(\Gamma xt)=\phi(\Gamma x)\tau(t),$$
for every $t\in S$ and almost every $\Gamma x$ is essentially constant.
Bader and Furman prove the following criterion in \cite[Thm~7.3]{BF}, which is proved in the same way as the corresponding statement in \cite[Thm~4.3]{BF2}.

\begin{restatable}[Bader--Furman]{thm}{BFinitial}\label{thm:BFinitial}
Suppose that the $S$-space $\Gamma\backslash X$ has only trivial algebraic representations.
Then there exists an $S$-gate in the category $\AR_S(X,[\mu])$.
\end{restatable}

We will be particularly interested in the relationship between different categories of algebraic representations when the acting group is changed.
For example, if $S\le T$, restriction gives a functor
$$\Res_S:\AR_T(X,[\mu])\to \AR_S(X,[\mu]),$$
which sends a $T$-algebraic representation $\phi:X\to\bfV(k)$, $\tau:T\to\bfI(k)$ to the $S$-algebraic representation with the same $\phi$ and with equivariance homomorphism $\tau'=\tau\vert_S$ (with $\bfI$ replaced by the Zariski closure of $\tau'(S)$).
When the underlying measurable $\Gamma$-map is unchanged, as $\phi$ is in this example, we say the algebraic representations \emph{have the same map}.
Bader and Furman show that restriction preserves gates under a normalizer hypothesis.
For this normalizer statement, suppose that $S$ is a closed subgroup of $G$ and that the right action on $X$ extends to $N_G(S)$, commuting with the left $\Gamma$-action.
Write $N(S)=N_G(S)$.
The following is \cite[Prop~7.4]{BF} (see also \cite[Thm~4.6]{BF2}).

\begin{restatable}[Bader--Furman]{thm}{BFnormal}\label{thm:BFnormal}
Suppose that the $S$-space $\Gamma\backslash X$ has only trivial algebraic representations and that $N(S)$ acts nonsingularly on $(X,[\mu])$.
Then the functor $\Res_S$ preserves gates, that is, if $\widetilde{\phi}_0$ is an $N(S)$-gate then the image $\phi_0=\Res_S(\widetilde{\phi}_0)$ is an $S$-gate with the same map.
Moreover, every $S$-gate admits a lift to $\AR_{N(S)}(X,[\mu])$ with the same map.
That is, if $\phi_0$ is an $S$-gate then there exists an $N(S)$-gate $\widetilde{\phi}_0$ with the same map for which $\phi_0=\Res_S(\widetilde{\phi}_0)$.
\end{restatable}

The nonsingularity hypothesis ensures that the normalizer acts on the same measure class, which is automatic in the Haar-measure setting.
We now apply these results to $(X_0,[\widetilde{\mu}_{\BMS}])$ and $(X_+,[\widetilde{\mu}_{\BR}])$, beginning with the lifted BMS measure.
Fix once and for all the Weyl element $w\in N_W(A)\cap K$ acting as $-1$ on $e_{n-1},e_n$ and as $+1$ on the other basis vectors, for $n\ge3$.

\begin{prop}\label{prop:BMSgates}
The $A$-space $\Gamma\backslash X_0$ has only trivial algebraic representations.
In particular, each of the categories $\AR_A(X_0,[\widetilde{\mu}_{\BMS}])$ and $\AR_{N(A)}(X_0,[\widetilde{\mu}_{\BMS}])$ has gates.
Moreover, any $A$-gate can be lifted to an $N(A)=\langle M, w, A\rangle$-gate with the same map.
Consequently, such a gate is equivariant for the fixed Weyl element $w$.
\end{prop}

\begin{proof}
By Theorem~\ref{thm:BMSmixing}, the pushforward of the normalized BMS measure under any $A$-algebraic representation of $\Gamma\backslash X_0$ is an invariant weakly mixing probability measure.
By \cite[Cor~1.13]{BDL} (see also \cite[Prop~3.3]{BF2}), this measure is a point mass so the $A$-space $\Gamma\backslash X_0$ has only trivial algebraic representations.
The action of $N(A)$ is nonsingular by the definition of $\widetilde{\mu}_{\BMS}$ so Theorems~\ref{thm:BFinitial} and \ref{thm:BFnormal} apply.
\end{proof}

In the finite-volume case, the BR measure is proportional to Haar measure on $\Gamma\backslash G$ and the $N$-action is mixing (hence weakly mixing) by Howe--Moore \cite{HoweMoore}.
Consequently, in this case \cite[Cor~1.13]{BDL} gives algebraic triviality for the $N$-action, just as for the BMS measure above.
However, the BR measure is infinite in our setting, so this argument does not apply directly.
The following elementary lemma allows us instead to obtain algebraic triviality from the subgroup ergodicity established in Theorem~\ref{thm:BRsubgroupergodicity}.

\begin{lemma}\label{lem:vectoralgebraictriviality}
Let $Y$ be a nonsingular $\bbR^d$-Lebesgue space and fix an integer $c$ with $1\le c\le d$.
If every codimension-$c$ linear subgroup of $\bbR^d$ acts ergodically on $Y$, then $Y$ has only trivial algebraic representations.
\end{lemma}

\begin{proof}
Let $\phi:Y\to\bfV(k)$ be an algebraic representation with associated $\tau:\bbR^d\to\bfI(k)$.
Since $\tau$ has Zariski dense image and $\bbR^d$ is abelian, $\bfI$ is abelian.
Ergodicity and local closedness of algebraic orbits show that the essential image of $\phi$ lies in a single $\bfI(k)$-orbit.
We therefore may identify it with the locally compact abelian group $Q=\bfI(k)/\bfJ(k)$.
We regard $\phi$ as taking values in $Q$ and continue to use $\tau$ for the induced homomorphism $\bbR^d\to Q$.

For every continuous character $\chi:Q\to S^1$, there is a real linear functional $\alpha$ such that $\chi(\tau(v))=e^{i\alpha(v)}$.
If $\alpha\ne0$, choose a codimension-$c$ linear subgroup $V\subseteq\ker\alpha$.
Then $\chi\circ\phi$ is $V$-invariant and hence essentially constant.
Equivariance would then give $e^{i\alpha(v)}=1$ for every $v$, contradicting that $\alpha\ne0$.
Therefore every character is trivial on $\tau(\bbR^d)$.
Since characters separate points of $Q$, the homomorphism $\tau$ is trivial, and ergodicity makes $\phi$ essentially constant.
\end{proof}

\begin{prop}\label{prop:BRgates}
The $N$-space $\Gamma\backslash X_+$ has only trivial algebraic representations.
In particular, each of the categories $\AR_U(X_+,[\widetilde{\mu}_{\BR}])$, $\AR_N(X_+,[\widetilde{\mu}_{\BR}])$, and $\AR_P(X_+,[\widetilde{\mu}_{\BR}])$ has gates.
Moreover, any $U$-gate can be lifted to a $P$-gate with the same map.
\end{prop}

\begin{proof}
Theorem~\ref{thm:BRsubgroupergodicity} gives ergodicity of every hyperplane subgroup of $N$ on $\Gamma\backslash X_+$.
The case $c=1$ of Lemma~\ref{lem:vectoralgebraictriviality} therefore gives algebraic triviality for the $N$-action.

To obtain the same conclusion for $U$, let $\phi:\Gamma\backslash X_+\to\bfV(k)$ be a $U$-algebraic representation.
The group $U$ is abelian and acts ergodically by Theorem~\ref{thm:BRsubgroupergodicity} so its essential image lies in an abelian orbit $Q=\bfI(k)/\bfJ(k)$, as in the preceding proof.
For each $n\in N$, commutativity makes the map $$x\mapsto\phi(xn)\phi(x)^{-1},$$
$U$-invariant and hence essentially constant.
Denoting its value by $\tau_N(n)$, nonsingularity and joint measurability show that $\tau_N:N\to Q$ is a measurable homomorphism and therefore continuous.
Thus $\phi$ is an $N$-algebraic representation and hence essentially constant.
Consequently, Theorem~\ref{thm:BFinitial} gives $U$- and $N$-gates.

Since $U\trieq N\trieq P$ and $P$ acts nonsingularly, the normalizer argument in Theorem~\ref{thm:BFnormal}, applied successively to $U\le N$ and $N\le P$, extends each $U$-gate to a $P$-gate with the same map.
Uniqueness of gates also shows that restricting any $P$-gate gives $N$- and $U$-gates.
\end{proof}

\subsection{The insufficiency of the category \texorpdfstring{$\AR_U(G,[\widetilde{\lambda}])$}{AlgRep for U (G, [lambda tilde])} and a brief description of the proof of Theorem~\ref{thm:superrigid}}\label{sec:insufficient}

We now explain why the argument of \cite{BFMS} needs to be modified and sketch the proof of Theorem~\ref{thm:superrigid}.
We begin with the step in \cite[Thm~1.6]{BFMS} that deduces superrigidity from a $\Gamma$-equivariant map
$$G/W\longrightarrow\Prob(\bbP(V)(k)),$$
which is measurable with respect to Haar measure.

\subsubsection{AREA as applied by Bader--Fisher--Miller--Stover}

For this subsection only, assume that $\Gamma$ is a lattice.
The starting point is the confluence of tameness of the action $\bfH(k)\curvearrowright\Prob(\bbP(V)(k))$ \cite[Cor~1.8]{BDL}, work of Furstenberg on stabilizers of measures in $\Prob(\bbP(V)(k))$ (see \cite[Ch~3]{Zimmer}), and work of Bader and Gelander on metric ergodicity \cite[Cor~6.7]{BG}.
Together, these results convert the above map to a measurable $\Gamma$-equivariant map
$$G/W\longrightarrow\bfH/\bfL(k),$$
for a proper closed algebraic subgroup $\bfL<\bfH$.
Pulling back along $\pi_W$ gives a $\Gamma\times W$-map
$$\phi:G\longrightarrow\bfH/\bfL(k),$$
which is $W$-invariant.
In the language of AREA, $\phi$ is an object in $\AR_W(G,[\mu_G])$ which is $W$-equivariant for the trivial homomorphism $\tau$.

As $U<W$, we obtain the restricted algebraic representation $\phi_U=\Res_U(\phi)$ which has the same map as $\phi$.
By Howe--Moore \cite{HoweMoore}, any noncompact subgroup $S\le G$ is mixing on $\Gamma\backslash G$ and hence all such categories $\AR_S(G,[\mu_G])$ admit $S$-gates.
In particular, there is a $U$-gate $\phi_0$ in $\AR_U(G,[\mu_G])$, which admits a unique morphism to $\phi_U$
\begin{equation}\label{eqn:initialdomination}
\xymatrix{
G\ar[r]^-{\phi_0}\ar[dr]_-{\phi_U}&\bfH/\bfL_0(k)\ar[d]^-\alpha\\
&\bfH/\bfL(k)},
\end{equation}
and moreover Theorem~\ref{thm:BFnormal} shows that $\phi_0$ is also a $P$-gate.
Writing $\tau_0:P\to\bfN_{\bfH}(\bfL_0)/\bfL_0(k)$ for the homomorphism associated to $\phi_0$, the $U$-invariance of $\phi$ and the morphism in Equation~\eqref{eqn:initialdomination} give
$$\alpha(h\tau_0(u))=\alpha(h),$$
for all $h\in\bfH/\bfL_0(k)$ and $u\in U$.
We distinguish the cases according to whether $\bfL_0$ is trivial.

In the former case, $\phi_0:G\to \bfH(k)$ is $\Gamma\times P$-equivariant and therefore $\Res_A(\phi_0)$ is an element of $\AR_A(G,[\mu_G])$.
However, since $\bfL_0$ is trivial, $\Res_A(\phi_0)$ is an $A$-gate and lifts to an $N(A)$-gate with the same map.
In particular, an $N(A)$-gate and a $P$-gate are algebraic representations with the same map.
One then uses \cite[Lem~5.1]{BF2} to show that the same map $\phi_0$ is in fact $G$-equivariant from which one easily extends the representation (see the argument in \cite[proof of Thm~1.3]{BF2}).

To rule out the case in which $\bfL_0$ is nontrivial, we use the compatibility condition of \cite{BFMS}.

\begin{defn}
We say that the triple $(G,\bfH,k)$ is compatible if, for every nontrivial $k$-algebraic subgroup $\bfL\le\bfH$ and every continuous homomorphism $\tau:P\to\bfN_{\bfH}(\bfL)/\bfL(k)$, the Zariski closures of $\tau(U)$ and $\tau(N)$ agree.
\end{defn}

The natural representations used in the arithmeticity argument satisfy this condition by \cite[Lem~3.5]{BFMS}.
When $\bfL_0$ is nontrivial, compatibility gives equality of the Zariski closures of $\tau_0(U)$ and $\tau_0(N)$.
Since $\alpha$ is algebraic, its invariance under $\tau_0(U)$ therefore gives
$$\alpha(h\tau_0(n))=\alpha(h),$$
for all $h\in\bfH/\bfL_0(k)$ and $n\in N$.
The $N$-equivariance of $\phi_0$ and the diagram in Equation~\eqref{eqn:initialdomination} now give $N$-invariance of $\phi_U$.
Since $\phi_U$ has the same map as $\phi$, the latter is invariant under $N$ and $W$, which generate $G$.
This shows that the image of $\rho$ is contained in a conjugate of $\bfL$, contradicting Zariski density as $\bfL$ is a proper subgroup of $\bfH$.

\subsubsection{The inability to apply these ideas in our setting}

In the preceding sketch, every category used Haar measure $\mu_G$.
The lattice assumption and the Howe--Moore theorem allowed us to compare gates for the noncompact subgroups in this common measure class.
In our setting, let $\widetilde{\lambda}$ be the Haar lift of $\lambda$ along the right $W$-fibers of $\pi_W:G\to G/W$, extended by zero outside $\pi_W^{-1}(\calC^*)$.
When $\Gamma$ is the fundamental group of a manifold as in Theorem~\ref{thm:measurerigid}, this part of the argument no longer applies.
More precisely, we will begin with a measurable $\Gamma$-equivariant map
$$\calC^*\longrightarrow\Prob(\bbP(V)(k)),$$
which is measurable with respect to the measure $\lambda$.
Since $\lambda$ is extended by zero outside $\calC^*$, we may regard this as a measurable $\Gamma$-map
$$G/W\longrightarrow\Prob(\bbP(V)(k)),$$
and as before we obtain a $\Gamma$-map
$$G/W\longrightarrow\bfH/\bfL(k),$$
where we may still show that $\bfL$ is a proper subgroup of $\bfH$.
Consequently we get an object in $\AR_W(G,[\widetilde{\lambda}])$ via
$$\phi:G\longrightarrow\bfH/\bfL(k).$$
The right $W$-action on its target is trivial.
To repeat the finite-volume argument, we would need to compare its restriction to $U$ with a $U$-gate.
The obstruction comes from the lack of ergodicity of the right $U$-action on $\Gamma\backslash G$ in the quotient measure class induced by $[\widetilde{\lambda}]$.
In fact, the next proposition shows that the relevant category has no initial objects.

\begin{prop}\label{prop:noinitial}
The category $\AR_U(G,[\widetilde{\lambda}])$ does not have initial objects.
\end{prop}

\begin{proof}
We first note that infinitely many distinct $\Gamma\times U$-invariant measurable indicators obstruct the existence of an initial object.
Indeed, suppose that $\phi_0:G\to\bfV_0(k)$ was such an initial object.
For each invariant measurable set $E$, its indicator $\mathbbm{1}_E$ is an object with target $\bbA^1$ and trivial $\bfH$-action.
Initiality therefore supplies a regular map $\alpha_E:\bfV_0\to\bbA^1$ such that $\mathbbm{1}_E=\alpha_E\circ\phi_0$ for $[\widetilde{\lambda}]$-almost every point.
Let $\bfZ$ be the essential Zariski closure of the image of $\phi_0$.
Then the identity $\alpha_E^2-\alpha_E=0$ holds on $\bfZ$, so $\alpha_E$ is constant, with value $0$ or $1$, on each connected component of $\bfZ$.
Since $\bfZ$ has only finitely many connected components, only finitely many distinct indicators can arise this way.

To obtain the required indicators in our setting, write $\Omega_\Gamma=\partial\bbH^n\ssm\Lambda_\Gamma$ and let $g^+$ denote the endpoint fixed by right translation by $N$.
On the set
$$D=\{g\in\pi_W^{-1}(\calC^*)\mid g^+\in\Omega_\Gamma\},$$
the map $g\mapsto[g^+]\in\Gamma\backslash\Omega_\Gamma$ is left $\Gamma$-invariant and right $U$-invariant.
The set $D$ is nonempty and open, and $\widetilde{\lambda}$ is in the Haar measure class on $\pi_W^{-1}(\calC^*)$.
Since the endpoint map is a submersion and $\Gamma$ acts properly discontinuously on $\Omega_\Gamma$, its pushforward class in a local chart of $\Gamma\backslash\Omega_\Gamma$ is nonzero, smooth, and non-atomic.
Pulling back indicators of infinitely many distinct measurable subsets of this chart and extending them by zero outside $D$ gives the required invariant indicators.
The preceding paragraph therefore rules out an initial object.
\end{proof}

Thus the initial-object argument of \cite{BFMS} cannot be carried out in this measure class.
We instead compare algebraic representations in the lifted BMS and BR measure classes, as described next.

\subsubsection{The sketch of the proof of Theorem~\ref{thm:superrigid}}\label{sec:sketchsuperrigid}

We now sketch the proof of Theorem~\ref{thm:superrigid}, leaving the details to the remainder of the section.
Neither the BMS nor the BR measure class is preserved by all the group actions needed in the proof, so we will need to transfer algebraic representations between them.

The first step is to realize the map on $G/W$ using both the lifted BMS and BR measure classes.
The comparison of their pushforwards to $G/W$ and the passage to a proper coset space give, by Corollary~\ref{cor:Winitial}, a proper subgroup $\bfL<\bfH$ and two algebraic representations
$$\phi_{\BMS}:X_0\longrightarrow\bfH/\bfL(k),\qquad
\phi_{\BR}:X_+\longrightarrow\bfH/\bfL(k).$$
These come from the same $W$-invariant map, with $\phi_{\BMS}$ invariant under $\langle A,w\rangle$ and $\phi_{\BR}$ invariant under $AU$.
The device for comparing them will be the notion of $T$-linked spaces.
In particular, $U$-linking will give a transfer functor satisfying
$$\Tr_U(\phi_{\BR})=\Res_A(\phi_{\BMS}),$$
where we recall that the right-hand side is $\phi_{\BMS}$ considered as an $A$-algebraic representation.
As in the previous sketch, choose a $P$-gate
$$\phi_0:X_+\longrightarrow\bfH/\bfL_0(k),$$
with corresponding homomorphism $\tau_0:P\to\bfN_{\bfH}(\bfL_0)/\bfL_0(k)$.
Restricting the actions yields $\phi_{0,U}=\Res_U(\phi_0)$ and $\phi_{0,AU}=\Res_{AU}(\phi_0)$, which are gates in the categories of $U$- and $AU$-algebraic representations, respectively.
Initiality in the $AU$-category therefore gives a morphism $\alpha$ with
$$\xymatrix{
X_+\ar[r]^-{\phi_{0,AU}}\ar[dr]_-{\phi_{\BR}}&\bfH/\bfL_0(k)\ar[d]^-\alpha\\
&\bfH/\bfL(k),}$$
commuting in the BR measure class.
We again distinguish the cases according to whether $\bfL_0$ is trivial.

Suppose first that $\bfL_0$ is nontrivial.
The trivial right action on the target of $\phi_{\BR}$ gives $\alpha(h\tau_0(u))=\alpha(h)$ for $u\in U$.
Compatibility extends this invariance to $\tau_0(N)$, so $\phi_{\BR}=\alpha\circ\phi_{0,AU}$ is $AN$-invariant.
Its transfer $\Res_A(\phi_{\BMS})$ is also $w$-invariant.
The constancy statement in Lemma~\ref{lem:linkedconstancy} will then show that $\phi_{\BR}$ is essentially constant, contradicting Zariski density of $\rho(\Gamma)$ since $\bfL$ is proper.
Here the $N$-invariance and $w$-invariance occur on different measured spaces, so we must compare them through transfer.

In the remaining case $\bfL_0$ is trivial, so $\phi_0$ takes values in $\bfH(k)$.
Write
$$\phi_0^{\BMS}=\Tr_U(\phi_{0,AU}):X_0\longrightarrow\bfH(k).$$
This is an $A$-algebraic representation for $\tau_0\vert_A$ and is itself an $A$-gate by the triviality assumption.
The normalizer argument extends its equivariance to $N(A)$, giving a homomorphism $\tau_0^{\BMS}:N(A)\to\bfH(k)$ extending $\tau_0\vert_A$.
Thus we have a $P$-equivariant map $\phi_0$ on the BR space and an $N(A)$-equivariant map $\phi_0^{\BMS}$ on the BMS space, related by transfer.
The final step uses $P$-linking to compare these maps on a Bruhat cell and extend $\rho$ to $G$.

\subsection{Converting the \texorpdfstring{$W$}{W}-invariant map to algebraic representations with measures \texorpdfstring{$\widetilde{\mu}_{\BMS}$}{mu tilde BMS}, \texorpdfstring{$\widetilde{\mu}_{\BR}$}{mu tilde BR}}\label{sec:measureclasscomparison}

In this section, we show how to convert the $\lambda$-measurable $\Gamma$-map
\begin{equation}\label{eqn:psieqn}
\psi:G/W\longrightarrow\Prob(\bbP(V)(k)),
\end{equation}
into coset algebraic representations in $\AR_{\langle A,w\rangle}(X_0,[\widetilde{\mu}_{\BMS}])$ and $\AR_{AU}(X_+,[\widetilde{\mu}_{\BR}])$.
Pullback along $\pi_W$ gives a $[\widetilde{\lambda}]$-measurable, $\Gamma$-equivariant, right $W$-invariant map on $G$.
To use the BMS and BR classes instead, we first show that both project to $[\lambda]$ on $\calC^*$.

Applications of classical projection theory to Patterson--Sullivan measures appear in Maucourant--Schapira \cite[\S4.3]{MS} and Mohammadi--Oh \cite[Thm~5.3, Cor~5.5]{MO}.
For related absolute-continuity statements for horospherical transverse measures, see Dufloux \cite[Thm~4.1]{Dufloux}.
Here we regard the ideal boundary of a hyperbolic $m$-plane as the intersection of $S^{n-1}\subset\bbR^n$ with an affine $m$-plane.
For each boundary point, we consider the affine $m$-planes through it and apply the projection theorem to control their distribution in the transverse directions.

\begin{prop}\label{prop:BRprojection}
Let $\pi_W$ be the natural projection from $G$ to $G/W$.
Then
$$(\pi_W)_*[\widetilde{\mu}_{\BMS}]=(\pi_W)_*[\widetilde{\mu}_{\BR}]=[\lambda],$$
on $\calC^*$.
\end{prop}

\begin{proof}
Since $\delta_\Gamma+\dim U>n-1$ and the quotient BMS measure is finite, \cite[Prop~4.10]{MS} gives equality of the BMS and BR measure classes on $U$-invariant Borel sets in $\Gamma\backslash G$.
To apply this upstairs, take a Borel set $E\subset G/W$ and form the countable left $\Gamma$-saturation of $\pi_W^{-1}(E)$.
This saturation is right $U$-invariant because $U<W$ and it is null for either lifted measure exactly when $\pi_W^{-1}(E)$ is null.
Passing to $\Gamma\backslash G$ therefore gives equality of the two pushforward classes so it suffices to prove the assertion for $\widetilde{\mu}_{\BR}$.

We first forget the normal frame and work on $\overline{\calC}$.
Indeed, the compact fiber $S_m/W$ has representatives in $M\cap S_m$ and both the BR measure and Haar measure on $G$ are right $M$-invariant.
Their projected classes on $\calC$ are consequently obtained from those on $\overline{\calC}$ by lifting with Haar probability measure along these fibers.
It will therefore suffice to prove absolute continuity on $\overline{\calC}$ and then use ergodicity on $\calC^*$ to obtain equality of the classes.

Put $\calG=\Gr_m(\bbR^n)$ and identify the boundary with the unit sphere $S^{n-1}\subset\bbR^n$.
Every affine $m$-plane meeting the open unit ball intersects $S^{n-1}$ in a round $(m-1)$-sphere and every such sphere arises uniquely in this way.
We may therefore identify $\overline{\calC}$ with this open subset of the affine Grassmannian by taking the affine spans of the ideal boundaries.

For a fixed boundary point $\xi$, the affine $m$-planes through $\xi$ have the form $\xi+L$ with $L\in\calG$.
Their intersections with $S^{n-1}$ give the spheres through $\xi$, except when $L\subseteq\xi^\perp$ and the affine plane is tangent to the sphere.
Writing $dL$ for invariant probability measure on $\calG$, these excluded directions form a lower-dimensional Grassmannian of $dL$-measure zero.
Thus, identifying a hyperbolic plane with its ideal boundary, the map
$$F(\xi,L)=S^{n-1}\cap(\xi+L),$$
is a map to $\overline{\calC}$ which is defined $d\nu\,dL$-almost everywhere.

We can now express the projected BR class using this family of intersections.
A local trivialization over the positive endpoint gives
$$[d\widetilde{\mu}_{\BR}]=[d\nu(\xi)\,dm_P],$$
where $dm_P$ is Haar measure on $P$.
For each fixed $\xi$, the $P$-fiber projects onto the family of spheres through $\xi$ with smooth measure class.
In the parameters above, this conditional class is $[dL]$, so the projected BR class is represented by
\begin{equation}\label{eqn:sphericalincidence}
\overline{\kappa}=F_*(d\nu\,dL).
\end{equation}
The remaining fiber coordinates and the positive smooth densities do not affect the measure class.
In particular, $\overline{\kappa}$ is a finite representative of the projected class, as required by our convention for pushforwards.

To compare this measure with the smooth class, we now record the transverse position of an affine plane with direction $L$.
Such a plane has a unique expression $a+L$ with $a\in L^\perp$ and it meets the open unit ball exactly when $|a|<1$.
These parameters give smooth coordinates on $\overline{\calC}$ in which the smooth class is $[dL\,da]$, where $da$ is Lebesgue measure on $L^\perp$.
If $P_{L^\perp}:\bbR^n\to L^\perp$ is orthogonal projection then $\xi+L=P_{L^\perp}\xi+L$, so the transverse position of the plane through $\xi$ is $a=P_{L^\perp}\xi$.
Disintegrating Equation~\eqref{eqn:sphericalincidence} over $L$ therefore gives
$$d\overline{\kappa}(L,a)=dL\,d((P_{L^\perp})_*\nu)(a),$$
so it remains to show that $(P_{L^\perp})_*\nu$ is absolutely continuous for almost every $L$.
This is precisely the conclusion of the finite-energy projection theorem, applied to $\nu$ as a measure on $\bbR^n$ supported on $S^{n-1}$.
Indeed, Sullivan's estimate \cite[Thms~7,~8]{Sullivan79} gives $\nu(B(\xi,r))\le Cr^{\delta_\Gamma}$ for $0<r<1$ and $\delta_\Gamma>n-2\ge n-m$.
Consequently, we see that
$$I_{n-m}(\nu)=\iint\frac{d\nu(\xi)\,d\nu(\eta)}{|\xi-\eta|^{n-m}}
\le C'\left(1+\int_0^1r^{\delta_\Gamma-(n-m)-1}\,dr\right)<\infty,$$
and the same estimate gives finite $t$-energy for $n-m<t<\delta_\Gamma$.
As $L$ varies with measure $dL$, its orthogonal complement is distributed according to the invariant probability on the Grassmannian of $(n-m)$-dimensional subspaces.
The projection theorem \cite[Thm~9.7]{MattilaGSMES}, in the form recalled in \cite[Thm~4.5]{MS}, now gives $(P_{L^\perp})_*\nu\ll da$ for $dL$-almost every $L$.
It follows that $\overline{\kappa}\ll dL\,da$, proving absolute continuity of the projected BR class with respect to the smooth class on $\overline{\calC}$.

Finally, lift back to $\calC$ using Haar probability measure on the normal-frame fibers and let $\kappa$ be a finite representative of the projected BR class.
We have proved that $\kappa\ll\lambda$ on $\calC^*$, so the usual ergodicity argument will give equality of the classes.
First note that this restriction is nonzero.
For $\xi\in\supp\nu=\Lambda_\Gamma$, we can choose an $m$-plane whose closure contains $\xi$ and which passes through the interior of the convex hull.
Meeting this interior is an open condition so a neighborhood of the corresponding pair $(\xi,L)$ has positive $d\nu\,dL$-measure and maps into $\overline{\calC}^*$.
Since $[\kappa]$ is $\Gamma$-invariant, the positive-density set $\{d\kappa/d\lambda>0\}$ is invariant modulo $\lambda$ and is therefore conull by Proposition~\ref{prop:lambdaergodicity}.
This gives $[\kappa]=[\lambda]$ on $\calC^*$, as required.
Moreover, $\pi_W^{-1}(\calC^*)$ is left $\Gamma$-invariant and right $U$-invariant.
Since it has positive BR measure, ergodicity of the $U$-action on $(\Gamma\backslash X_+,\mu_{\BR})$ shows that it is BR-conull.
The equality of the projected BMS and BR classes proved at the start of the argument then makes it BMS-conull as well.
\end{proof}

We can now apply the orbit reduction from \cite[Prop~4.1]{BFMS} in both measure classes.

\begin{cor}\label{cor:Winitial}
Let $\psi$ be as in Equation~\eqref{eqn:psieqn} and let
$$\Phi=\psi\circ\pi_W:G\longrightarrow\Prob(\bbP(V)(k)),$$
be its pullback, considered as a $\Gamma\times W$-map with trivial right action.
Then there is a proper closed $k$-algebraic subgroup $\bfL<\bfH$ such that $\Phi$ induces coset algebraic representations with target $\bfH/\bfL(k)$ in the two categories $\AR_{\langle A,w\rangle}(X_0,[\widetilde{\mu}_{\BMS}])$ and $\AR_{AU}(X_+,[\widetilde{\mu}_{\BR}])$.
The right actions are trivial, and both maps are obtained by pulling back the same $\Gamma$-equivariant map from $G/W$.
\end{cor}

\begin{proof}
By Proposition~\ref{prop:BRprojection}, we may consider $\Phi$ in both measure classes.
Tameness of the action on $\Prob(\bbP(V)(k))$, together with ergodicity of the $U$-action on $\Gamma\backslash X_+$, gives a single $\bfH(k)$-orbit $\bfH(k)/C$ containing its essential image in the $[\widetilde{\mu}_{\BR}]$ measure class.
As the two pushforward classes agree, this orbit contains $\psi$ for $\lambda$-almost every plane and consequently also contains the essential image of $\Phi$ in the $[\widetilde{\mu}_{\BMS}]$ measure class.

We first show that $C$ cannot be compact.
Indeed, if it were then $\bfH(k)/C$ admits a compatible $\bfH(k)$-invariant metric.
Write $\Lambda_\Gamma^{(2)}=\{(\xi,\xi')\in\Lambda_\Gamma^2\mid \xi\ne\xi'\}$ for the space of ordered pairs of distinct limit points.
As $\Phi$ is $A$-invariant, we may average its Dirac measures over the compact $M$-fiber to obtain a $\Gamma$-equivariant map
$$\Lambda_\Gamma^{(2)}\longrightarrow\Prob(\bfH(k)/C),$$
which is measurable for $\nu\otimes\nu$.
The target admits a separable invariant metric induced from a bounded invariant metric on $\bfH(k)/C$.
Since Bader and Furman proved metric ergodicity of the boundary-pair action \cite[Thm~1.6]{BFmetric}, this map is essentially constant and its essential image is a $\rho(\Gamma)$-invariant probability measure $\eta$.
When $C$ is compact, the $\bfH(k)$ action on $\bfH(k)/C$ is proper and therefore the set of elements carrying a compact subset $K_1$ to a set meeting $K_1$ is compact.
Choose a compact $K_1$ with $\eta(K_1)>1/2$.
Since $\eta$ is $\rho(\Gamma)$-invariant, every translate $\rho(\gamma)K_1$ meets $K_1$, contradicting the unboundedness of $\rho(\Gamma)$.

Now let $\bfL$ be the $k$-Zariski closure of $C$.
The Furstenberg stabilizer argument in \cite[Prop~4.1]{BFMS} shows that $\bfL$ is proper, and projection to $\bfH/\bfL(k)$ gives the required $\Gamma$-equivariant map on $G/W$.
Choosing a representative on a $\Gamma$-invariant conull set of planes and pulling back along $\pi_W$, we obtain a common $W$-invariant representative in both measure classes by Proposition~\ref{prop:BRprojection}.
The result now follows since $\langle A,w\rangle<W$ acts nonsingularly on $X_0$ and $AU<W$ acts nonsingularly on $X_+$, where the associated homomorphisms $\tau$ are the trivial ones.
\end{proof}

\subsection{$T$-linked categories}\label{sec:Tlinked}

To carry out the sketch of Section~\ref{sec:sketchsuperrigid}, we need to transfer algebraic representations between the measured spaces $(X_0,\widetilde{\mu}_{\BMS})$ and $(X_+,\widetilde{\mu}_{\BR})$.
The following definition records the measure-class relation needed for this transfer.

\begin{defn}
Let $(Z,\eta)$ and $(X,\mu)$ be two measure spaces with $Z\subseteq X$.
Suppose that $X$ has a nonsingular left $\Gamma$-action preserving $Z$ and acting nonsingularly on $(Z,\eta)$, and that $T$ is an lcsc group acting nonsingularly on the right on $(X,\mu)$ and commuting with the $\Gamma$-action.
Define
\begin{align*}
\alpha_T:Z\times T&\longrightarrow X,\\
(z,t)&\longmapsto z\cdot t.
\end{align*}
We say that $(Z,\eta)$ and $(X,\mu)$ are \emph{$T$-linked} if $(\alpha_T)_*(\eta\otimes m_T)$ is absolutely continuous with respect to $\mu$, where $m_T$ is left Haar measure on $T$.
\end{defn}

The definition does not require $T$ to preserve $Z$, so $T$ need not act on $Z$ itself.
This distinction is needed in our application.
Under the following hypotheses, $T$-linking nevertheless allows us to transfer algebraic representations from $X$ to $Z$.

\begin{prop}\label{prop:Tlinker}
Let $ST$ be an lcsc group with closed subgroups $S$ and $T$ such that $S$ normalizes $T$.
Suppose that $ST$ acts nonsingularly on the right on $(X,\mu)$ and commutes with the nonsingular left $\Gamma$-action.
Let $Z\subseteq X$ be a $\Gamma$- and $S$-invariant measurable subset equipped with a measure $\eta$ for which the restricted actions are nonsingular.
Assume that $(Z,\eta)$ and $(X,\mu)$ are $T$-linked.
Then there is a functor
$$\Tr_T:\AR_{ST}(X,[\mu])\to \AR_S(Z,[\eta]),$$
which takes an $ST$-algebraic representation $\phi$ with corresponding homomorphism $\tau$ to the unique object $\phi_0$ in $\AR_S(Z,[\eta])$ with the property that
$$\phi(zt)=\phi_0(z)\tau(t),$$
for $\eta\otimes m_T$-almost every $(z,t)$.
On morphisms, $\Tr_T$ retains the underlying $k$-algebraic map.
\end{prop}

\begin{proof}
Choose a measurable representative of $\phi$ and define
$$\phi_0(z)=\mathrm{ess~img}_{t\in T}\phi(zt)\tau(t)^{-1},$$
where the essential image is taken with respect to $m_T$.
Equivariance, $T$-linking, and Fubini's theorem show that this image is a single point for $\eta$-almost every $z$.
The linking condition also makes it independent of the representative, since changing $\phi$ on a $\mu$-null set changes $\phi_0$ only on an $\eta$-null set.
Measurability follows from Fubini's theorem by integrating $\mathbbm{1}_E(\phi(zt)\tau(t)^{-1})$ over a fixed Borel subset of $T$ of positive finite Haar measure, for Borel $E\subseteq\bfV(k)$.
Since $S$ normalizes $T$, conjugation preserves Haar-null sets, and the identity $zst=z(sts^{-1})s$ gives $\phi_0(zs)=\phi_0(z)\tau(s)$ for every $s\in S$ and $\eta$-almost every $z$.
The left $\Gamma$-equivariance also passes from $\phi$ to $\phi_0$.
Thus $\phi_0$ is an $S$-algebraic representation, with $\tau$ restricted to $S$ and $\bfI$ replaced by the Zariski closure of $\tau(S)$.
The defining identity gives uniqueness and equivariance of morphisms makes this construction functorial.
\end{proof}

\begin{rem}
In our applications, $S$ and $T$ will lie in an ambient group $G$, although the statement itself does not require one.
The transfer changes the measure class and is defined even when $\mu(Z)=0$.
Further hypotheses would be needed to determine whether this functor preserves initial objects or gates.
\end{rem}

We will also need the compatibility of transfer with enlargement of the linking subgroup.

\begin{prop}\label{prop:enlargedlinking}
Retain the hypotheses of Proposition~\ref{prop:Tlinker}.
Let $\widehat{T}$ be an lcsc group containing $T$ as a closed subgroup and acting nonsingularly on $(X,[\mu])$, extending the given $T$-action and commuting with $\Gamma$.
Then $(Z,\eta)$ and $(X,\mu)$ are $\widehat{T}$-linked.
If the given actions of $S$ and $\widehat{T}$ arise from a nonsingular right action of an lcsc group $S\widehat{T}$, with $S$ normalizing $\widehat{T}$ and both embedded as closed subgroups, then
$$\Tr_T(\Res_{ST}(\phi))=\Tr_{\widehat{T}}(\phi),$$
for every $\phi\in \AR_{S\widehat{T}}(X,[\mu])$.
\end{prop}

\begin{proof}
Write $m_{\widehat{T}}$ for left Haar measure on $\widehat{T}$.
If $E\subseteq X$ is $\mu$-null, nonsingularity of the $\widehat{T}$-action and $T$-linking give $\mathbbm{1}_E(zt\widehat{t})=0$ for $\eta\otimes m_T\otimes m_{\widehat{T}}$-almost every $(z,t,\widehat{t})$.
By Fubini and left invariance, for $\eta$-almost every $z$ and $m_T$-almost every $t$ we have
$$0=\int_{\widehat{T}}\mathbbm{1}_E(zt\widehat{t})\,dm_{\widehat{T}}(\widehat{t})
=\int_{\widehat{T}}\mathbbm{1}_E(z\widehat{t})\,dm_{\widehat{T}}(\widehat{t}).$$
Hence $(\alpha_{\widehat{T}})_*(\eta\otimes m_{\widehat{T}})(E)=0$, proving $\widehat{T}$-linking.

Now let $\phi$ have corresponding homomorphism $\tau$, and put $\phi_0=\Tr_T(\Res_{ST}(\phi))$.
Equivariance, $T$-linking, and Fubini give
$$\phi(zt\widehat{t})=\phi_0(z)\tau(t\widehat{t}),$$
for $\eta\otimes m_T\otimes m_{\widehat{T}}$-almost every $(z,t,\widehat{t})$.
By left invariance of $m_{\widehat{T}}$, this gives the defining identity for $\Tr_{\widehat{T}}(\phi)$, so the transferred maps agree $\eta$-almost everywhere.
Their $S$-actions are both given by $\tau\vert_S$ and the two constructions send a morphism to the same underlying algebraic map.
\end{proof}

We now apply this to the measure spaces of interest in our application.

\begin{prop}\label{prop:BMSBRlink}
The spaces $(X_0,\widetilde{\mu}_{\BMS})$ and $(X_+,\widetilde{\mu}_{\BR})$ are $U$- and $P$-linked.
\end{prop}

\begin{proof}
The $U$-linking assertion follows from \cite[Prop~4.10]{MS}, since $\delta_\Gamma+\dim U>n-1$.
The $P$-linking assertion follows by enlargement of the linking subgroup in Proposition \ref{prop:enlargedlinking}.
\end{proof}

\subsection{Reducing the stabilizer using compatibility}\label{sec:stabilizerreduction}

We now use compatibility to show that the $P$-gate in the sketch of Section~\ref{sec:sketchsuperrigid} has trivial stabilizer.

\begin{restatable}{thm}{stabilizertrivial}\label{thm:stabtrivial}
Let $\Phi$ be as in Corollary~\ref{cor:Winitial} and let
$$\phi_0:X_+\to\bfH/\bfL_0(k),$$
be a $P$-gate in $\AR_P(X_+,[\widetilde{\mu}_{\BR}])$.
Then $\bfL_0$ is trivial.
\end{restatable}

\begin{proof}
Let $\phi_{\BMS}$ and $\phi_{\BR}$ be the objects supplied by Corollary~\ref{cor:Winitial}.
They come from the same $W$-invariant map, and Proposition~\ref{prop:BMSBRlink} gives
$$\Tr_U(\phi_{\BR})=\Res_A(\phi_{\BMS}).$$
By Proposition~\ref{prop:BRgates}, the restriction $\phi_{0,U}=\Res_U(\phi_0)$ is a $U$-gate.
Since $U\trieq AU$, the normalizer argument also makes $\phi_{0,AU}=\Res_{AU}(\phi_0)$ an $AU$-gate.

Write $\tau_0:P\to\bfN_{\bfH}(\bfL_0)/\bfL_0(k)$ for the homomorphism associated to $\phi_0$.
By initiality in the $AU$-category, there is a unique morphism
$$\alpha:\bfH/\bfL_0\longrightarrow\bfH/\bfL,$$
to the target of $\phi_{\BR}$, with $\phi_{\BR}=\alpha\circ\phi_{0,AU}$.
Since $\phi_{\BR}$ comes from the $W$-invariant map, its right $AU$-action is trivial, and the morphism identity gives
$$\alpha(h\tau_0(u))=\alpha(h),$$
for all $h\in\bfH/\bfL_0(k)$ and $u\in U$.

Suppose that $\bfL_0$ is nontrivial.
The invariance of this algebraic morphism extends to the Zariski closure of $\tau_0(U)$, and compatibility therefore gives
$$\alpha(h\tau_0(n))=\alpha(h),$$
for all $h\in\bfH/\bfL_0(k)$ and $n\in N$.
Consequently
$$\phi_{\BR}(xn)=\alpha(\phi_0(x)\tau_0(n))=\alpha(\phi_0(x))=\phi_{\BR}(x),$$
for each $n\in N$ and $[\widetilde{\mu}_{\BR}]$-almost every $x$.
Thus $\phi_{\BR}$ is $AN$-invariant.
Functoriality of transfer also gives
$$\Res_A(\phi_{\BMS})=\alpha\circ\Tr_U(\phi_{0,AU}).$$
Moreover, $\phi_{\BMS}$ is $w$-invariant, and its restriction $\Res_A(\phi_{\BMS})$ is the $U$-transfer of $\phi_{\BR}$.
If $\phi_{\BR}$ is essentially constant, its $\Gamma$-equivariance gives a fixed point of $\rho(\Gamma)$ on $\bfH/\bfL(k)$, contradicting Zariski density since $\bfL$ is proper.
The following lemma gives this constancy and completes the proof.
\end{proof}

\begin{lemma}\label{lem:linkedconstancy}
Suppose that $\phi$ is an object in $\AR_{AN}(X_+,[\widetilde{\mu}_{\BR}])$ with trivial right action.
Let
$$\phi_A=\Tr_U(\Res_{AU}(\phi)),$$
be its transfer to $\AR_A(X_0,[\widetilde{\mu}_{\BMS}])$.
If $\phi_A(xw)=\phi_A(x)$ for $[\widetilde{\mu}_{\BMS}]$-almost every $x$, then both $\phi$ and $\phi_A$ are essentially constant with the same value.
\end{lemma}

The proof is related in spirit to \cite[\S4]{Winter}, however the objects and equivariances considered here are different enough that a new proof is needed.

\begin{proof}
Since $\phi$ is $AN$-invariant, it descends to a measurable map
$$\phi:X_+/AN\longrightarrow\bfV(k),$$
which we denote by the same symbol.
The Iwasawa decomposition identifies $G/AN$ with $K=\SO(n)$, and the boundary projection becomes $\pi_{\partial}:K\to K/M=\partial\bbH^n$.
Under this identification, $X_+/AN=\pi_{\partial}^{-1}(\Lambda_\Gamma)$.

The quotient of $[\widetilde{\mu}_{\BR}]$ on $K$ is represented by the probability measure $\overline{\mu}_{\BR}$ whose disintegration over $\pi_{\partial}$ is
$$\overline{\mu}_{\BR}=\int_{\partial\bbH^n}m_\xi\,d\nu(\xi),$$
where $\nu$ is our Patterson--Sullivan probability measure and $m_\xi$ is Haar probability on the fiber $\pi_{\partial}^{-1}(\xi)$.
We first show that, for $\nu$-almost every $\xi$, the map $\phi$ is $m_\xi$-almost everywhere constant on $\pi_{\partial}^{-1}(\xi)$.
We then use Weyl invariance of $\phi_A$ to show that these fiber values agree $\nu$-almost everywhere, completing the proof.

Take $e=e_n$, so that $M=\Stab_K(e)$ and $\pi_{\partial}(q)=qe$.
Then for $v\ne0$, the orthogonal reflection in $v^\perp$ has matrix representative
$$r_v=I-2\frac{vv^{\sfT}}{v\cdot v}.$$
For the fixed Weyl representative, we have $w=r_er_{e_{n-1}}$.
Let $\pi_K:G\to K$ denote the Iwasawa projection.
If $\xi=g^+$, $\xi'=g^-$, and $q=\pi_K(g)$, then right translation by $w$ gives $\pi_K(gw)=r_{\xi-\xi'}qr_{e_{n-1}}$.
Here $r_{\xi-\xi'}\xi=\xi'$, and the two reflections ensure that the right-hand side has determinant $1$ and hence belongs to $K$.

For distinct $\xi,\xi'\in\partial\bbH^n$, we therefore have a transport map between their corresponding fibers over $\pi_{\partial}$ given by
\begin{align*}
T_{\xi,\xi'}:\pi_{\partial}^{-1}(\xi)&\longrightarrow\pi_{\partial}^{-1}(\xi'),\\
q&\longmapsto r_{\xi-\xi'}qr_{e_{n-1}}.
\end{align*}
The pushforward along this map takes $m_\xi$ to $m_{\xi'}$.
Since $\phi$ has trivial right action, the $U$-linking identity defining its transfer is $\phi(xu)=\phi_A(x)$ for $\widetilde{\mu}_{\BMS}\otimes m_U$-almost every $(x,u)$, where $m_U$ is Haar measure on $U$.
The $U$-linking condition allows us to substitute the identity $\phi(g)=\phi(\pi_K(g))$, which holds $[\widetilde{\mu}_{\BR}]$-almost everywhere, into this equality.
Since $U<AN$, we have $\pi_K(xu)=\pi_K(x)$, and hence
$$\phi_A(x)=\phi(xu)=\phi(\pi_K(xu))=\phi(\pi_K(x)),$$
for $\widetilde{\mu}_{\BMS}\otimes m_U$-almost every $(x,u)$.
Thus $\phi_A(x)=\phi(\pi_K(x))$ for $[\widetilde{\mu}_{\BMS}]$-almost every $x$.
Since the measure class under consideration is $[\nu\otimes\nu]$, Weyl invariance of $\phi_A$ and nonsingularity of the Weyl action now give the transport invariance
\begin{equation}\label{eqn:weyltransport}
\phi(q)=\phi(T_{\xi,\xi'}(q)),
\end{equation}
for $\nu\otimes\nu$-almost every $(\xi,\xi')$ and $m_\xi$-almost every $q$.

Compose these identities around a loop $\xi\to \xi'\to \xi''\to \xi'''\to \xi$ of typical points to conclude that
$$\phi(q)=\phi(r_{\xi'''-\xi}r_{\xi''-\xi'''}r_{\xi'-\xi''}r_{\xi-\xi'}q),$$
for $\nu^{\otimes4}$-almost every tuple $(\xi,\xi',\xi'',\xi''')$ and $m_\xi$-almost every $q$.
Consequently
$$r_{\xi'''-\xi}r_{\xi''-\xi'''}r_{\xi'-\xi''}r_{\xi-\xi'}\in\Stab_K(\xi),$$
for $\nu^{\otimes4}$-almost every tuple $(\xi,\xi',\xi'',\xi''')$.
In particular
$$L_\xi=\left\{h\in\Stab_K(\xi)\mid \phi(hq)=\phi(q)\text{ for }m_\xi\text{-a.e. }q\right\},$$
is a closed real algebraic subgroup of $\Stab_K(\xi)$, for which $r_{\xi'''-\xi}r_{\xi''-\xi'''}r_{\xi'-\xi''}r_{\xi-\xi'}\in L_\xi$ for $\nu^{\otimes3}$-almost every triple $(\xi',\xi'',\xi''')$.
Note that for fixed $\xi$ the entries of the matrix $r_{\xi'''-\xi}r_{\xi''-\xi'''}r_{\xi'-\xi''}r_{\xi-\xi'}$ are rational functions in $(\xi',\xi'',\xi''')$.
In particular, substituting them into any polynomial equation defining $L_\xi$ and clearing denominators gives a polynomial vanishing $\nu^{\otimes3}$-almost everywhere.
Therefore the polynomial vanishes on $(\supp\nu)^3$.
As $\supp\nu=\Lambda_\Gamma$ is Zariski dense in $\partial\bbH^n$, this polynomial therefore vanishes identically and hence $r_{\xi'''-\xi}r_{\xi''-\xi'''}r_{\xi'-\xi''}r_{\xi-\xi'}\in L_\xi$ for all pairwise distinct triples $(\xi',\xi'',\xi''')$.
Take $\xi''=-\xi$ and choose $\xi',\xi'''$ in the unit sphere of $\xi^\perp$.
The resulting products restrict to products of two orthogonal reflections on $\xi^\perp$, which generate $\SO(\xi^\perp)=\Stab_K(\xi)$.
Thus $L_\xi=\Stab_K(\xi)$, and $\phi$ is $m_\xi$-almost everywhere constant on a typical fiber.

Finally, it remains to show that the values of $\phi$ on these fibers agree.
By the preceding fiberwise constancy, there is a measurable map $c:\Lambda_\Gamma\to\bfV(k)$ such that $\phi(g)=c(g^+)$ for $[\widetilde{\mu}_{\BR}]$-almost every $g$.
The defining transfer identity and $U$-linking then give
$$\phi_A(x)=\phi(xu)=c((xu)^+)=c(x^+),$$
for $\widetilde{\mu}_{\BMS}\otimes m_U$-almost every $(x,u)$, since $U<N$ preserves the positive endpoint.
Weyl invariance of $\phi_A$ and nonsingularity of the Weyl action on $X_0$ now give
$$c(x^-)=c((xw)^+)=\phi_A(xw)=\phi_A(x)=c(x^+).$$
By the Hopf-coordinate description of $[\widetilde{\mu}_{\BMS}]$, it follows that $c(\xi)=c(\xi')$ for $\nu\otimes\nu$-almost every $(\xi,\xi')$.
Fubini's theorem therefore makes $c$ essentially constant and the identities above then show that $\phi$ and $\phi_A$ are essentially constant with the same value.
\end{proof}

\subsection{The proof of Theorems~\ref{thm:superrigid} and \ref{thm:arithmetic}}\label{sec:superrigidityproof}

We now prove Theorem~\ref{thm:superrigid}, which we restate for the reader's convenience.

\superrigidity*

\begin{proof}[Proof of Theorem~\ref{thm:superrigid}]
By the hypotheses, Corollary~\ref{cor:Winitial}, and Theorem~\ref{thm:stabtrivial}, the category $\AR_P(X_+,[\widetilde{\mu}_{\BR}])$ has a $P$-gate with trivial stabilizer.
Explicitly, we have a $P$-algebraic representation
$$\phi_0:X_+\longrightarrow\bfH(k),$$
which is a $[\widetilde{\mu}_{\BR}]$-measurable $\Gamma\times P$-map with associated continuous homomorphism $\tau_0:P\to\bfH(k)$.

Suppose first that $k$ is non-archimedean.
Since $P$ is connected and $\bfH(k)$ is totally disconnected, $\tau_0$ is trivial.
The restricted map $\phi_{0,AU}=\Res_{AU}(\phi_0)$ gives rise to a morphism $\alpha$ with $\phi_{\BR}=\alpha\circ\phi_{0,AU}$.
Therefore $\phi_{\BR}$ is $AN$-invariant.
Its $U$-transfer is $\phi_A=\Res_A(\phi_{\BMS})$, which is $w$-invariant, and therefore Lemma~\ref{lem:linkedconstancy} shows that $\phi_{\BR}$ is essentially constant.
Since its target is $\bfH/\bfL(k)$ with $\bfL$ proper, this contradicts Zariski density of $\rho(\Gamma)$ as it implies that $\rho(\Gamma)$ is contained in a conjugate of $\bfL$.
We may therefore assume that $k$ is archimedean.

Put $\phi_{0,AU}=\Res_{AU}(\phi_0)$, which is an $AU$-gate.
Proposition~\ref{prop:BMSBRlink} and compatibility of transfer give
$$\phi_0^{\BMS}=\Tr_{MN}(\phi_0)
=\Tr_U(\phi_{0,AU}),$$
which is an $A$-algebraic representation in $\AR_A(X_0,[\widetilde{\mu}_{\BMS}])$ with associated continuous homomorphism $\tau_0^{\BMS}=\tau_0\vert_A$.
Moreover, $\phi_0^{\BMS}$ is an $A$-gate.
Indeed, the morphism from an $A$-gate to this object forces the gate's stabilizer to be trivial and is therefore an isomorphism.
Hence Theorem~\ref{thm:BFnormal} extends it to an $N(A)$-gate with the same map and associated extended homomorphism which we continue to denote by $\tau_0^{\BMS}$.
In particular,
$$\phi_0^{\BMS}(xw)=\phi_0^{\BMS}(x)\tau_0^{\BMS}(w),$$
where $w$ is the fixed Weyl element.
The restriction of $\tau_0^{\BMS}$ to $MA$ agrees with $\tau_0\vert_{MA}$ by cancellation.
Keeping $S=A$, compatibility of transfer gives $\Tr_{MN}(\phi_0)=\Tr_P(\phi_0)$ in $\AR_A(X_0,[\widetilde{\mu}_{\BMS}])$.
We may therefore use the construction of $\Tr_P$ to choose a common representative of $\phi_0$ and $\phi_0^{\BMS}$ which is exactly $P$-equivariant on a $P$-invariant domain conull for both $[\widetilde{\mu}_{\BMS}]$ and $[\widetilde{\mu}_{\BR}]$.

Fix a typical $P$-orbit $g_0P$, with $g_0\in G$.
By Fubini, we may choose this orbit so that the Weyl identity holds for Patterson--Sullivan-almost every opposite endpoint and almost every $MA$-parameter.
$P$-equivariance and continuity then extend it to every $MA$-parameter.
Since $G$ has rank $1$, its open Bruhat cell is $NwP$.
We work on its translate $g_0NwP$.
Define the map $\Phi_0:g_0NwP\to\bfH(k)$ via
\begin{equation}\label{eqn:BRextension}
\Phi_0(g_0nwp)=\phi_0(g_0)\tau_0(n)\tau_0^{\BMS}(w)\tau_0(p).
\end{equation}
Then the preceding identity and $P$-equivariance show that this map agrees with $\phi_0$ $[\widetilde{\mu}_{\BR}]$-almost everywhere since
\begin{align*}
\phi_0(g_0nwp)&=\phi_0^{\BMS}(g_0nw)\tau_0(p),\\
&=\phi_0^{\BMS}(g_0n)\tau_0^{\BMS}(w)\tau_0(p),\\
&=\phi_0(g_0)\tau_0(n)\tau_0^{\BMS}(w)\tau_0(p).
\end{align*}
Since $k$ is archimedean, the continuous homomorphism $\tau_0$ is real analytic.
Equation~\eqref{eqn:BRextension} therefore defines a real-analytic map on $g_0NwP$.
The $P$-transfer is precisely $\phi_0^{\BMS}$, so
$$\Phi_0(xw)=\Phi_0(x)\tau_0^{\BMS}(w),$$
holds BMS-almost everywhere.

We now apply real-analytic uniqueness successively in the two endpoint variables.
Indeed, a nonzero real-analytic function on a boundary chart has zero set of dimension at most $n-2$, while the Patterson--Sullivan measure has dimension $\delta_\Gamma>n-2$.
The Weyl identity therefore holds wherever both sides are defined, and in particular Haar-almost everywhere.

It remains to check that $\tau_0$ and $\tau_0^{\BMS}(w)$ respect the relations among $P$ and $w$.
For each such relation, we evaluate the Weyl identity together with right $P$-equivariance away from the finitely many translates of the complement of $g_0NwP$.
Canceling the common factor $\Phi_0(x)$ then gives the corresponding relation in $\bfH(k)$.
Since $G=\langle P,w\rangle$, they therefore define a homomorphism $\tau:G\to\bfH(k)$ extending $\tau_0$ and satisfying $\tau(w)=\tau_0^{\BMS}(w)$.
The Bruhat formula $\tau(nwp)=\tau_0(n)\tau_0^{\BMS}(w)\tau_0(p)$ makes $\tau$ continuous on $NwP$, hence everywhere.
Writing $b=\phi_0(g_0)\tau(g_0)^{-1}$, Equation~\eqref{eqn:BRextension} gives
$$\Phi_0(x)=\phi_0(g_0)\tau(g_0^{-1}x)=b\tau(x),$$
on the translated Bruhat cell.
Finally, for each $\gamma\in\Gamma$, choose a $[\widetilde{\mu}_{\BR}]$-typical $x$ so that $\Phi_0$ agrees with $\phi_0$ at both $x$ and $\gamma x$.
The $\Gamma$-equivariance of $\phi_0$ then gives
$$b\tau(\gamma)\tau(x)=\Phi_0(\gamma x)=\rho(\gamma)\Phi_0(x)=\rho(\gamma)b\tau(x).$$
Thus $\widetilde{\rho}(g)=b\tau(g)b^{-1}$ is a continuous extension of $\rho$.
\end{proof}

We now complete the proof of arithmeticity, restating the theorem for the reader's convenience.

\arithmeticity*

\begin{proof}[Proof of Theorem~\ref{thm:arithmetic}]
After passing to an infinite subsequence, we may assume that the maximal submanifolds have constant dimension $m$.
The trace field is a number field by Section~\ref{sec:arithmeticitycriterion}.
For each unbounded local representation in the arithmeticity criterion, the suspension construction of Section~\ref{sec:suspension} and Theorem~\ref{thm:measurerigid} give the equivariant map required by Theorem~\ref{thm:superrigid}.
The natural target groups satisfy the compatibility condition by \cite[Lem~3.5]{BFMS}, so Theorem~\ref{thm:superrigid}, applied with $W=W_m$, shows that each such representation is superrigid.
The required measure rigidity is supplied by Theorem~\ref{thm:higherdimensionhaar} for $m\ge3$ and Theorem~\ref{thm:surfacehaar} for $m=2$.
Margulis' criterion, Theorem~\ref{thm:marguliscriterion}, then implies that $\Gamma$ is contained in an arithmetic lattice $\Lambda$.

It remains to find a torsion-free finite-index subgroup of $\Lambda$ containing $\Gamma$.
The compact totally geodesic boundary of $\core(Y_\infty)$ implies that $\Lambda$ is of simplest type by \cite[Rem~2.2]{FLMS}.
Since $\Gamma$ is convex cocompact, the virtual-retraction theorem of Bergeron, Haglund, and Wise \cite{BHW}, together with work of Kolpakov, Reid, and Slavich \cite[\S6]{KRS} shows that its closure in the profinite completion $\widehat{\Lambda}$ is torsion-free.
Compactness then separates this closure from the finitely many profinite conjugacy classes represented by nontrivial torsion elements of $\Lambda$.
As in \cite[\S7.1]{KRS}, this gives a torsion-free finite-index subgroup $\Lambda'<\Lambda$ containing $\Gamma$.
Thus $Y_\infty$ covers the arithmetic manifold $\Lambda'\backslash\bbH^n$.
\end{proof}

\section{Examples of Theorem~\ref{thm:arithmetic} and counterexamples to its weakened form}

In this section, we give some well-known constructions of infinitely many commensurability classes of hyperbolic $n$-manifolds satisfying Theorem~\ref{thm:arithmetic}.
We then give infinitely many commensurability classes of convex cocompact manifolds without Fuchsian ends that satisfy all the hypotheses of Theorem~\ref{thm:arithmetic} except the Fuchsian ends hypothesis, but not its conclusion.
In particular, this proves Theorem~\ref{thm:counterexamples} and gives examples of manifolds showing that the conclusion can fail when the Fuchsian ends hypothesis is omitted.

\begin{ex}[Examples of manifolds satisfying Theorem~\ref{thm:arithmetic}]\label{ex:arithmeticpieces}
Fix $n\ge 3$ and let $Y=\Lambda\backslash\bbH^n$ be a closed arithmetic hyperbolic manifold containing a compact embedded separating geodesic hypersurface $\Sigma$.
Cutting $Y$ along $\Sigma$ gives two compact hyperbolic manifolds $Q_1$ and $Q_2$ with totally geodesic boundary and a splitting $Y=Q_1\cup_\Sigma Q_2$ whose fundamental group has $\Lambda=\Gamma_1*_{\Delta}\Gamma_2$, where $\Gamma_i=\pi_1(Q_i)$ and $\Delta=\pi_1(\Sigma)$.
Each $\Gamma_i$ is Zariski dense in $G$ by Borel density \cite[\S0.2,~\S1.7]{GPS}.
Consequently, the infinite-volume manifold $Y_{\infty,i}=\Gamma_i\backslash\bbH^n$ is convex cocompact with Fuchsian ends, its convex core is naturally isometric to $Q_i$, and it covers the arithmetic manifold $Y$.
Moreover, the arithmeticity of $Y$ implies that $\core(Y_{\infty,i})$ meets infinitely many distinct maximal infinite-volume geodesic submanifolds by the arithmetic geodesic submanifold dichotomy \cite[p.~3592]{FLMS} and the finiteness result \cite[Thm~1.7(3)]{LO}.
In dimension three, the existence of infinitely many properly immersed geodesic planes in such arithmetic covers is also recorded in \cite[Thm~1.11]{MOIsolation}.

These pieces are the basic building blocks in two standard constructions of nonarithmetic manifolds.
For the Gromov--Piatetski--Shapiro construction, one begins with noncommensurable arithmetic manifolds $Y^{(1)}$ and $Y^{(2)}$ containing pieces $Q_1$ and $Q_2$ whose boundaries are isometric.
If $f:\partial Q_1\to\partial Q_2$ is an isometry, then $Y_{\mathrm{GPS}}=Q_1\cup_fQ_2$ is a closed nonarithmetic hyperbolic manifold \cite{GPS}.
Thus $\pi_1(Q_i)\backslash\bbH^n$ covers both its original arithmetic manifold $Y^{(i)}$ and the nonarithmetic manifold $Y_{\mathrm{GPS}}$.
The same observation applies to the pieces with two isometric boundary components in \cite[\S2.9]{GPS}.
By cyclically gluing varying numbers of copies of two such pieces from noncommensurable arithmetic manifolds, we obtain nonarithmetic manifolds of unbounded volume.
Keeping one piece fixed shows that its associated convex cocompact manifold covers infinitely many nonisometric nonarithmetic manifolds.

The construction of Agol and Belolipetsky--Thomson instead starts with an appropriate compact piece $Q$ cut from a single arithmetic manifold and forms its double $Y_{\mathrm{BT}}=DQ$.
The piece may be chosen so that $Y_{\mathrm{BT}}$ is nonarithmetic \cite{Agol3,BelThom}.
Once again, if $\Gamma=\pi_1(Q)$, then $Y_\infty=\Gamma\backslash\bbH^n$ covers both the original arithmetic manifold and the resulting nonarithmetic double.

Although these examples show that $Y_\infty$ can cover finite-volume manifolds from distinct commensurability classes, at most one of these classes can be arithmetic.
Indeed, suppose that $Y_\infty=\Gamma\backslash\bbH^n$ covers arithmetic manifolds $\Lambda_1\backslash\bbH^n$ and $\Lambda_2\backslash\bbH^n$.
After conjugating, we have $\Gamma<\Lambda_1\cap\Lambda_2$.
Since $\Gamma$ is Zariski dense, the intersection $\Lambda_1\cap\Lambda_2$ is Zariski dense.
The commensurability criterion of Gromov and Piatetski--Shapiro \cite[\S1.6]{GPS} then implies that $\Lambda_1\cap\Lambda_2$ has finite index in both $\Lambda_1$ and $\Lambda_2$.
Hence the two arithmetic manifolds are commensurable.

Finally, these constructions give infinitely many commensurability classes in every dimension $n\ge 3$, as there are infinitely many commensurability classes of arithmetic manifolds satisfying the initial hypothesis.
Indeed, if two of the resulting convex cocompact manifolds were commensurable, a common finite-index subgroup of their fundamental groups would be Zariski dense and contained in both arithmetic lattices.
The same commensurability criterion would force the arithmetic lattices to be commensurable, a contradiction.
\end{ex}

\begin{ex}[Examples of manifolds in Theorem~\ref{thm:counterexamples}]\label{ex:counterexample}
For every $n\ge 3$, we now construct infinitely many commensurability classes of infinite-volume convex cocompact hyperbolic $n$-manifolds with infinitely many hypersurfaces meeting the convex core which do not cover any arithmetic manifold.
In particular, these examples cannot have Fuchsian ends by Theorem~\ref{thm:arithmetic}.

Begin with the splitting $\Lambda=\Gamma_1*_{\Delta}\Gamma_2$ from Example~\ref{ex:arithmeticpieces} and choose a nontrivial element $\gamma\in\Gamma_2$ which is not conjugate into $\Delta$.
A combination theorem of Gitik \cite{Gitik} shows that for all sufficiently large $r$, the subgroup $\langle\Gamma_1,\gamma^r\rangle\cong\Gamma_1*\langle\gamma^r\rangle$ is quasiconvex in $\Lambda$.
The fixed points $\gamma^\pm$ lie outside $\Lambda_{\Gamma_1}$ and in distinct $\Gamma_1$-orbits.
Choose round neighborhoods $B^\pm$ of these points whose closures lie in $S^{n-1}\ssm\Lambda_{\Gamma_1}$ and whose $\Gamma_1$-translates have pairwise disjoint closures.
Increasing $r$ if necessary, we may arrange
$$\gamma^r(S^{n-1}\ssm B^-)\Subset B^+,
\qquad \gamma^{-r}(S^{n-1}\ssm B^+)\Subset B^-.$$
Convex cocompactness is open so there is a neighborhood $V$ of $\gamma^r$ such that $\langle\Gamma_1,g\rangle\cong\Gamma_1*\langle g\rangle$ is discrete, torsion-free, and convex cocompact for every $g\in V$.
Shrinking $V$ if necessary, we may also assume that the two strict inclusions above hold with $g$ in place of $\gamma^r$.
The commensurator $\Comm_G(\Lambda)$ is countable so we may choose $g\in V\ssm\Comm_G(\Lambda)$ and put $\Gamma_g=\langle\Gamma_1,g\rangle$.
Then $Y_g=\Gamma_g\backslash\bbH^n$ is an infinite-volume convex cocompact hyperbolic $n$-manifold.
The group $\Gamma_g$ is Zariski dense since it contains $\Gamma_1$ and moreover $\delta_{\Gamma_g}\ge\delta_{\Gamma_1}>n-2$ by work of Dal'bo--Otal--Peign\'e \cite[Prop~2]{DOP}.

The strict ping-pong inclusions and the reduced-word description of $\Lambda_{\Gamma_g}$ give a neighborhood of $\partial B^+$ disjoint from $\Lambda_{\Gamma_g}$.
Moreover, $\partial B^+$ separates $g^+$ from $\Lambda_{\Gamma_1}$, where $g^+$ is the attracting fixed point of $g$.
Thus the hyperplane bounded by $\partial B^+$ meets $D_{\Gamma_g}$.
The same is true for every round sphere in a sufficiently small open family about $\partial B^+$ and all these spheres remain disjoint from $\Lambda_{\Gamma_g}$.
If $C$ is one of the corresponding hyperplanes then $C\to Y_g$ is proper.
Indeed, otherwise there are points $p_j\in C$ escaping to infinity and elements $\alpha_j\in\Gamma_g$ such that $\alpha_jp_j$ remain in a compact set.
For a fixed basepoint $o$, the distances $d(p_j,\alpha_j^{-1}o)$ are bounded, giving a point of $\partial_\infty C\cap\Lambda_{\Gamma_g}$, a contradiction.
The stabilizer of $C$ is trivial since every nontrivial element of $\Gamma_g$ is loxodromic.
These hyperplanes therefore give properly immersed infinite-volume hypersurfaces meeting $\core(Y_g)$.
Since the family is uncountable and $\Gamma_g$ is countable, infinitely many of the resulting hypersurfaces are distinct.

Suppose that $\Gamma_g$ is contained in an arithmetic lattice $\Lambda'$.
Since $\Gamma_1<\Lambda\cap\Lambda'$ is Zariski dense, the commensurability criterion \cite[\S1.6]{GPS} implies that $\Lambda$ and $\Lambda'$ are commensurable.
Every element of $\Lambda'$ therefore belongs to $\Comm_G(\Lambda)$, contradicting the choice of $g$.
Thus $Y_g$ does not cover an arithmetic manifold and Theorem~\ref{thm:arithmetic} therefore implies that it does not have Fuchsian ends.

To obtain infinitely many commensurability classes, vary the initial arithmetic manifold over pairwise noncommensurable choices as in Example~\ref{ex:arithmeticpieces}.
The group $\Gamma_1$ is one-ended since its limit set is the connected complement in $S^{n-1}$ of the disjoint open balls corresponding to the boundary lifts of $Q_1$.
If two resulting groups $\Gamma_g$ and $\Gamma_{g'}$ were commensurable, then, after conjugating, a finite-index subgroup of $\Gamma_1$ would be contained in $\Gamma_{g'}=\Gamma'_1*\langle g'\rangle$.
This subgroup is still one-ended, so the Kurosh subgroup theorem implies that it is contained in a conjugate of $\Gamma'_1$.
The initial arithmetic lattices would then have a Zariski dense intersection after conjugation and hence would be commensurable, again by \cite[\S1.6]{GPS}, giving the requisite contradiction.
\end{ex}

\appendix

\section{The simplifications in codimension one when \texorpdfstring{$n\ge 4$}{n is at least 4}}\label{sec:hypersurfaceappendix}

In this appendix, we explain how the proof in Section~\ref{sec:higherdimensionabsolute} simplifies when the original submanifolds are hypersurfaces.
We retain the notation of the main text, set $m=n-1$, and assume $n\ge4$ throughout.
The limiting current is non-atomic by Section~\ref{sec:limitingmeasure}, and Corollary~\ref{cor:twocrossings} shows that it is nonzero and almost every hyperplane crosses two boundary lifts.
For hypersurfaces, we can see the latter assertion directly.
The hyperplanes crossing at most one boundary lift form a proper closed $\Gamma$-invariant subset of $\calC^*$, so the closed-or-dense dichotomy in Corollary~\ref{cor:hyperplaneclosures} shows that all their orbits are closed.
There are only countably many closed orbits, each of which is countable, so non-atomicity gives $\eta_\infty(\calC^*\ssm\calU_\partial)=0$.

The exceptional sets from the boundary disintegration are now countable unions of curves.
We will show that the exceptional sets associated with two distinct boundary lifts have countable intersection, so non-atomicity suffices to show that they have measure zero.

\subsection{The scalar angle parameter}

Here $W=W_{n-1}=\Stab_G(e_1)$ and $\calC=G/W$ is the space of oriented hyperplanes.
For a boundary lift $Z$ with unit normal $n_Z$, set
$$H_Z=\SO_0(V_Z),\qquad \Delta_Z=\Gamma\cap H_Z.$$
Choose a unit vector $u_Z\in V_Z$ with $\langle u_Z,u_Z\rangle >0$ and put
$$p_a=a n_Z+\sqrt{1-a^2}u_Z,\qquad C_a=p_a^\perp\cap\bbH^n,$$
for $a\in\bbB^1=(-1,1)$.
Then
\begin{align*}
H_Z/J\times(-1,1)&\longrightarrow U_Z,\\
(hJ,a)&\longmapsto hC_a,
\end{align*}
is the identification from Equation~\eqref{eqn:equivariant}, with $J\cong\SO_0(n-2,1)$.
If $dm_Z$ is invariant measure on $H_Z/J$, then, up to normalization,
$$d\lambda=(1-a^2)^{(n-2)/2}\,dm_Z\,da.$$
In particular, the smooth measure class on $U_Z$ is the product of the Haar measure class on $H_Z/J$ and Lebesgue measure on $(-1,1)$.

For a compact interval $I\subset(-1,1)$, Lemma~\ref{lem:banana} gives a finite measure on $(\Delta_Z\backslash H_Z)\times I$ by lifting along the $J$-fibers and then passing to the quotient by $\Delta_Z$.
Disintegrating over $I$ gives $J$-invariant probability measures.
Since $n\ge4$, the group $J$ is generated by unipotents, so Ratner's theorem applies to their ergodic components.
However, the only possible noncompact factors are now $J$ and $H_Z$.
In particular, there are no intermediate orbit closures.
Consequently each component is either Haar on the full frame bundle over the boundary or projects to a closed totally geodesic hypersurface of the boundary component.

\subsection{Eliminating the exceptional families}

Let $B\subset Z$ be a lift of one of these closed boundary hypersurfaces, whose linear span $V_B$ is indefinite of dimension $n-1$.
The corresponding exceptional planes lie in the family
$$\calP_{Z,B}=\{C\in U_Z\mid C\cap Z=B\}.$$
For each choice of orientation, this is a one-dimensional family parametrized by the angle.
We write $\calS_Z$ for the union of all these exceptional families.

Let $Z'$ be a second distinct boundary lift, and let $B'\subset Z'$ be one of its exceptional subspaces.
The linear span of a hyperplane belonging to both families must contain $V_B+V_{B'}$, which has dimension at least $n$.
Consequently there is at most one underlying hyperplane containing them both, with two choices of orientation.
Taking the union over the countably many pairs $(B,B')$ therefore shows that $\calS_Z\cap\calS_{Z'}$ is countable.
Since $\eta_\infty$ is non-atomic, this intersection has measure zero.

We now compute the measure of $\calS_Z\cap U_{Z'}$ using the boundary disintegration at $Z'$.
The Haar part gives this set measure zero.
Indeed, at each fixed angle with $Z'$, the homogeneous orbit has dimension $n-1\ge3$, whereas each family in $\calS_Z$ has dimension one.
The exceptional part is supported on $\calS_{Z'}$, and we have just shown that $\calS_Z\cap\calS_{Z'}$ has measure zero.
It follows that
$$\eta_\infty(\calS_Z\cap U_{Z'})=0.$$
However, $\eta_\infty$-almost every hyperplane in $\calS_Z$ crosses another boundary lift and there are only countably many such lifts, so the preceding equality gives $\eta_\infty(\calS_Z)=0$.
Repeating the argument for every $Z$ shows that all exceptional components have measure zero.

The boundary conditional measures are therefore Haar, so $\eta_\infty$ is $H_Z$-invariant on each $U_Z$.
The groups associated with two distinct boundary lifts generate $G$ and the local invariance argument in Section~\ref{sec:absolutecontinuityhigher} gives $\eta_\infty=b\lambda$ for some $b>0$.
With $c=b^{-1}$, this proves Theorem~\ref{thm:higherdimensionhaar} in codimension one.
The shorter proof is possible because two exceptional families have finite intersection, to which a non-atomic current gives measure zero.
In higher codimension these intersections need not be discrete, which is why we used the algebraic sets and the properness argument in Section~\ref{sec:intermediateclosures}.

\bibliographystyle{abbrv}
\bibliography{infinitevol}

@misc{BF,
  author        = {Bader, Uri and Furman, Alex},
  title         = {{Margulis}' super-rigidity theorem for non-lattices},
  year          = {2018},
  eprint        = {1810.01608},
  archivePrefix = {arXiv},
  primaryClass  = {math.GR},
  doi           = {10.48550/arXiv.1810.01608},
  url           = {https://arxiv.org/abs/1810.01608v1},
  note          = {Preprint, \href{https://arxiv.org/abs/1810.01608v1}{arXiv:1810.01608}}
}

@incollection{BF2,
  author        = {Bader, Uri and Furman, Alex},
  title         = {An extension of {Margulis}'s superrigidity theorem},
  booktitle     = {Dynamics, Geometry, Number Theory: The Impact of {Margulis} on Modern Mathematics},
  editor        = {Fisher, David and Kleinbock, Dmitry and Soifer, Gregory},
  publisher     = {University of Chicago Press},
  address       = {Chicago},
  year          = {2022},
  pages         = {47--65},
  doi           = {10.7208/chicago/9780226804163-003},
  url           = {https://doi.org/10.7208/chicago/9780226804163-003}
}

@book{MattilaGSMES,
  author = {Mattila, Pertti},
  title = {Geometry of Sets and Measures in {Euclidean} Spaces: Fractals and Rectifiability},
  series = {Cambridge Studies in Advanced Mathematics},
  volume = {44},
  publisher = {Cambridge University Press},
  address = {Cambridge},
  year = {1995}
}

@article{MM22,
  author  = {Margulis, Gregori and Mohammadi, Amir},
  title   = {Arithmeticity of hyperbolic $3$-manifolds containing infinitely many totally geodesic surfaces},
  journal = {Ergodic Theory Dynam. Systems},
  volume  = {42},
  number  = {3},
  pages   = {1188--1219},
  year    = {2022},
  doi     = {10.1017/etds.2021.21}
}

@article{BaldiUllmo,
  author  = {Baldi, Gregorio and Ullmo, Emmanuel},
  title   = {Special subvarieties of non-arithmetic ball quotients and {H}odge theory},
  journal = {Ann. of Math. (2)},
  volume  = {197},
  number  = {1},
  pages   = {159--220},
  year    = {2023},
  doi     = {10.4007/annals.2023.197.1.3}
}

@article{BFMS,
  author        = {Bader, Uri and Fisher, David and Miller, Nicholas and Stover, Matthew},
  title         = {Arithmeticity, superrigidity, and totally geodesic submanifolds},
  journal       = {Ann. of Math. (2)},
  volume        = {193},
  number        = {3},
  year          = {2021},
  pages         = {837--861},
  doi           = {10.4007/annals.2021.193.3.4},
  url           = {https://annals.math.princeton.edu/2021/193-3/p04}
}

@misc{BFmetric,
  author        = {Bader, Uri and Furman, Alex},
  title         = {Some ergodic properties of metrics on hyperbolic groups},
  year          = {2017},
  eprint        = {1707.02020},
  archivePrefix = {arXiv},
  primaryClass  = {math.DS},
  doi           = {10.48550/arXiv.1707.02020},
  url           = {https://arxiv.org/abs/1707.02020v3},
  note          = {Preprint, \href{https://arxiv.org/abs/1707.02020v3}{arXiv:1707.02020v3}}
}

@article{MS,
  author        = {Maucourant, Fran{\c{c}}ois and Schapira, Barbara},
  title         = {On topological and measurable dynamics of unipotent frame flows for hyperbolic manifolds},
  journal       = {Duke Math. J.},
  volume        = {168},
  number        = {4},
  year          = {2019},
  pages         = {697--747},
  doi           = {10.1215/00127094-2018-0050},
  url           = {https://doi.org/10.1215/00127094-2018-0050}
}

@article{Winter,
  author        = {Winter, Dale},
  title         = {Mixing of frame flow for rank one locally symmetric spaces and measure classification},
  journal       = {Israel J. Math.},
  volume        = {210},
  number        = {1},
  year          = {2015},
  pages         = {467--507},
  doi           = {10.1007/s11856-015-1258-5},
  url           = {https://doi.org/10.1007/s11856-015-1258-5}
}

@article{BDL,
  author        = {Bader, Uri and Duchesne, Bruno and L{\'e}cureux, Jean},
  title         = {Almost algebraic actions of algebraic groups and applications to algebraic representations},
  journal       = {Groups Geom. Dyn.},
  volume        = {11},
  number        = {2},
  year          = {2017},
  pages         = {705--738},
  doi           = {10.4171/GGD/413},
  url           = {https://doi.org/10.4171/GGD/413}
}

@article{BG,
  author        = {Bader, Uri and Gelander, Tsachik},
  title         = {Equicontinuous actions of semisimple groups},
  journal       = {Groups Geom. Dyn.},
  volume        = {11},
  number        = {3},
  year          = {2017},
  pages         = {1003--1039},
  doi           = {10.4171/GGD/420},
  url           = {https://doi.org/10.4171/GGD/420}
}

@article{HoweMoore,
  author        = {Howe, Roger E. and Moore, Calvin C.},
  title         = {Asymptotic properties of unitary representations},
  journal       = {J. Funct. Anal.},
  volume        = {32},
  number        = {1},
  year          = {1979},
  pages         = {72--96},
  doi           = {10.1016/0022-1236(79)90078-8},
  url           = {https://doi.org/10.1016/0022-1236(79)90078-8}
}

@book{Zimmer,
  author        = {Zimmer, Robert J.},
  title         = {Ergodic Theory and Semisimple Groups},
  series        = {Monographs in Mathematics},
  volume        = {81},
  publisher     = {Birkh{\"a}user},
  address       = {Boston, MA},
  year          = {1984},
  doi           = {10.1007/978-1-4684-9488-4},
  url           = {https://doi.org/10.1007/978-1-4684-9488-4}
}

@article{Sullivan79,
  author        = {Sullivan, Dennis},
  title         = {The density at infinity of a discrete group of hyperbolic motions},
  journal       = {Publ. Math. Inst. Hautes \'Etudes Sci.},
  volume        = {50},
  year          = {1979},
  pages         = {171--202},
  doi           = {10.1007/BF02684773},
  url           = {https://www.numdam.org/item/PMIHES_1979__50__171_0/}
}

@article{Patterson76,
  author        = {Patterson, S. J.},
  title         = {The limit set of a {Fuchsian} group},
  journal       = {Acta Math.},
  volume        = {136},
  year          = {1976},
  pages         = {241--273},
  doi           = {10.1007/BF02392046},
  url           = {https://doi.org/10.1007/BF02392046}
}

@book{Roblin03,
  author        = {Roblin, Thomas},
  title         = {Ergodicit\'e et \'equidistribution en courbure n\'egative},
  series        = {M\'emoires de la Soci\'et\'e Math\'ematique de France, Nouvelle S\'erie},
  number        = {95},
  publisher     = {Soci\'et\'e Math\'ematique de France},
  year          = {2003},
  doi           = {10.24033/msmf.408},
  url           = {https://www.numdam.org/item/MSMF_2003_2_95__1_0/}
}

@article{MMO,
  author = {McMullen, Curtis T. and Mohammadi, Amir and Oh, Hee},
  title = {Geodesic planes in hyperbolic {$3$}-manifolds},
  journal = {Invent. Math.},
  volume = {209},
  year = {2017},
  pages = {425--461},
  doi = {10.1007/s00222-016-0711-3},
  url = {https://doi.org/10.1007/s00222-016-0711-3}
}

@article{MMOhoro,
  author = {McMullen, Curtis T. and Mohammadi, Amir and Oh, Hee},
  title = {Horocycles in hyperbolic {$3$}-manifolds},
  journal = {Geom. Funct. Anal.},
  volume = {26},
  year = {2016},
  pages = {961--973},
  doi = {10.1007/s00039-016-0373-8},
  url = {https://doi.org/10.1007/s00039-016-0373-8}
}

@article{MMOcore,
  author = {McMullen, Curtis T. and Mohammadi, Amir and Oh, Hee},
  title = {Geodesic planes in the convex core of an acylindrical {$3$}-manifold},
  journal = {Duke Math. J.},
  volume = {171},
  number = {5},
  year = {2022},
  pages = {1029--1060},
  doi = {10.1215/00127094-2021-0030},
  url = {https://doi.org/10.1215/00127094-2021-0030}
}

@article{BOPlanes,
  author = {Benoist, Yves and Oh, Hee},
  title = {Geodesic planes in geometrically finite acylindrical {$3$}-manifolds},
  journal = {Ergodic Theory Dynam. Systems},
  volume = {42},
  number = {2},
  year = {2022},
  pages = {514--553},
  doi = {10.1017/etds.2021.19},
  url = {https://doi.org/10.1017/etds.2021.19}
}

@article{LO,
  author = {Lee, Minju and Oh, Hee},
  title = {Orbit closures of unipotent flows for hyperbolic manifolds with {Fuchsian} ends},
  journal = {Geom. Topol.},
  volume = {28},
  year = {2024},
  pages = {3373--3473},
  doi = {10.2140/gt.2024.28.3373},
  url = {https://doi.org/10.2140/gt.2024.28.3373}
}

@article{LOcore,
  author = {Lee, Minju and Oh, Hee},
  title = {Properness and finiteness of totally geodesic submanifolds in the convex core},
  journal = {Geom. Topol.},
  year = {to appear},
  eprint = {2511.09377},
  archivePrefix = {arXiv},
  primaryClass = {math.GT},
  doi = {10.48550/arXiv.2511.09377},
  url = {https://arxiv.org/abs/2511.09377},
  note = {\href{https://arxiv.org/abs/2511.09377}{arXiv:2511.09377}}
}

@article{BHC,
  author = {Borel, Armand and Harish-Chandra},
  title = {Arithmetic subgroups of algebraic groups},
  journal = {Ann. of Math. (2)},
  volume = {75},
  year = {1962},
  pages = {485--535},
  doi = {10.2307/1970210},
  url = {https://doi.org/10.2307/1970210}
}

@book{MargulisBook,
  author = {Margulis, G. A.},
  title = {Discrete Subgroups of Semisimple Lie Groups},
  series = {Ergebnisse der Mathematik und ihrer Grenzgebiete (3)},
  volume = {17},
  publisher = {Springer-Verlag},
  address = {Berlin},
  year = {1991},
  doi = {10.1007/978-3-642-51445-6},
  url = {https://doi.org/10.1007/978-3-642-51445-6}
}

@article{Vinberg,
  author = {Vinberg, {\`E}. B.},
  title = {Rings of definition of dense subgroups of semisimple linear groups},
  journal = {Math. USSR-Izv.},
  volume = {5},
  number = {1},
  year = {1971},
  pages = {45--55},
  doi = {10.1070/IM1971v005n01ABEH001006},
  url = {https://doi.org/10.1070/IM1971v005n01ABEH001006}
}

@article{GPS,
  author = {Gromov, M. and Piatetski-Shapiro, I.},
  title = {Non-arithmetic groups in {Lobachevsky} spaces},
  journal = {Publ. Math. Inst. Hautes \'Etudes Sci.},
  volume = {66},
  year = {1987},
  pages = {93--103},
  doi = {10.1007/BF02698928},
  url = {https://doi.org/10.1007/BF02698928}
}

@article{FLMS,
  author = {Fisher, David and Lafont, Jean-Fran\c{c}ois and Miller, Nicholas and Stover, Matthew},
  title = {Finiteness of maximal geodesic submanifolds in hyperbolic hybrids},
  journal = {J. Eur. Math. Soc. (JEMS)},
  volume = {23},
  number = {11},
  year = {2021},
  pages = {3591--3623},
  doi = {10.4171/JEMS/1077},
  url = {https://doi.org/10.4171/JEMS/1077}
}

@unpublished{Agol3,
  author = {Agol, Ian},
  title = {Systoles of hyperbolic 4-manifolds},
  year = {2006},
  note = {Preprint, arXiv:math/0612290},
  eprint = {math/0612290},
  archivePrefix = {arXiv},
  url = {https://arxiv.org/abs/math/0612290}
}

@article{BelThom,
  author = {Belolipetsky, Mikhail V. and Thomson, Scott A.},
  title = {Systoles of hyperbolic manifolds},
  journal = {Algebr. Geom. Topol.},
  volume = {11},
  number = {3},
  year = {2011},
  pages = {1455--1469},
  doi = {10.2140/agt.2011.11.1455},
  url = {https://doi.org/10.2140/agt.2011.11.1455}
}

@article{Furstenberg,
  author = {Furstenberg, Harry},
  title = {A note on {Borel}'s density theorem},
  journal = {Proc. Amer. Math. Soc.},
  volume = {55},
  number = {1},
  year = {1976},
  pages = {209--212},
  doi = {10.1090/S0002-9939-1976-0422497-X},
  url = {https://doi.org/10.1090/S0002-9939-1976-0422497-X}
}

@incollection{PrasadRapinchuk,
  author = {Prasad, Gopal and Rapinchuk, Andrei S.},
  title = {Developments on the congruence subgroup problem after the work of {Bass}, {Milnor}, and {Serre}},
  booktitle = {Collected Papers of John Milnor. V. Algebra},
  publisher = {American Mathematical Society},
  address = {Providence, RI},
  year = {2010},
  pages = {307--325}
}

@article{Zhang,
  author = {Zhang, Yongquan},
  title = {Existence of an exotic plane in an acylindrical {$3$}-manifold},
  journal = {Math. Res. Lett.},
  volume = {30},
  number = {2},
  year = {2023},
  pages = {611--631},
  doi = {10.4310/MRL.2023.v30.n2.a11},
  url = {https://doi.org/10.4310/MRL.2023.v30.n2.a11}
}

@article{ZhangElementary,
  author = {Zhang, Yongquan},
  title = {Elementary planes in the {Apollonian} orbifold},
  journal = {Trans. Amer. Math. Soc.},
  volume = {376},
  number = {1},
  year = {2023},
  pages = {453--506},
  doi = {10.1090/tran/8751},
  url = {https://doi.org/10.1090/tran/8751}
}

@article{Ratner90,
  author = {Ratner, Marina},
  title = {On measure rigidity of unipotent subgroups of semisimple groups},
  journal = {Acta Math.},
  volume = {165},
  number = {3--4},
  year = {1990},
  pages = {229--309},
  doi = {10.1007/BF02391906},
  url = {https://doi.org/10.1007/BF02391906}
}

@article{Ratner2,
  author = {Ratner, Marina},
  title = {On {Raghunathan}'s measure conjecture},
  journal = {Ann. of Math. (2)},
  volume = {134},
  number = {3},
  year = {1991},
  pages = {545--607},
  doi = {10.2307/2944357},
  url = {https://doi.org/10.2307/2944357}
}

@article{RatnerOrbit91,
  author = {Ratner, Marina},
  title = {{Raghunathan}'s topological conjecture and distributions of unipotent flows},
  journal = {Duke Math. J.},
  volume = {63},
  number = {1},
  year = {1991},
  pages = {235--280},
  doi = {10.1215/S0012-7094-91-06311-8},
  url = {https://doi.org/10.1215/S0012-7094-91-06311-8}
}

@article{Gitik,
  author = {Gitik, Rita},
  title = {Ping-pong on negatively curved groups},
  journal = {J. Algebra},
  volume = {217},
  number = {1},
  year = {1999},
  pages = {65--72},
  doi = {10.1006/jabr.1998.7789},
  url = {https://doi.org/10.1006/jabr.1998.7789}
}

@article{DOP,
  author = {Dal'bo, Fran\c{c}oise and Otal, Jean-Pierre and Peign\'e, Marc},
  title = {S\'eries de {P}oincar\'e des groupes g\'eom\'etriquement finis},
  journal = {Israel J. Math.},
  volume = {118},
  year = {2000},
  pages = {109--124},
  doi = {10.1007/BF02803518},
  url = {https://doi.org/10.1007/BF02803518}
}

@article{BFMS2,
  author = {Bader, Uri and Fisher, David and Miller, Nicholas and Stover, Matthew},
  title = {Arithmeticity, superrigidity and totally geodesic submanifolds of complex hyperbolic manifolds},
  journal = {Invent. Math.},
  volume = {233},
  number = {1},
  year = {2023},
  pages = {169--222},
  doi = {10.1007/s00222-023-01186-5},
  url = {https://doi.org/10.1007/s00222-023-01186-5}
}

@misc{FFL,
  author = {Filip, Simion and Fisher, David and Lowe, Ben},
  title = {Finiteness of totally geodesic hypersurfaces},
  year = {2024},
  eprint = {2408.03430},
  archivePrefix = {arXiv},
  primaryClass = {math.DG},
  url = {https://arxiv.org/abs/2408.03430},
  note = {Preprint, \href{https://arxiv.org/abs/2408.03430}{arXiv:2408.03430}},
  doi = {10.48550/arXiv.2408.03430}
}

@incollection{FilipSurvey,
  author = {Filip, Simion},
  title = {Measure rigidity beyond homogeneous dynamics},
  booktitle = {Proceedings of the International Congress of Mathematicians 2026, Volume 5: Invited Lectures (Sections 9--11)},
  editor = {Friedlander, Susan and Tschinkel, Yuri},
  publisher = {Society for Industrial and Applied Mathematics},
  address = {Philadelphia, PA},
  year = {2026},
  pages = {136--155},
  doi = {10.1137/25M1811756},
  url = {https://doi.org/10.1137/25M1811756}
}

@incollection{OhSurvey,
  author = {Oh, Hee},
  title = {Dynamics for discrete subgroups of {$\mathrm{SL}_2(\mathbb C)$}},
  booktitle = {Dynamics, Geometry, Number Theory: The Impact of {Margulis} on Modern Mathematics},
  editor = {Fisher, David and Kleinbock, Dmitry and Soifer, Gregory},
  publisher = {University of Chicago Press},
  address = {Chicago},
  year = {2022},
  pages = {506--565},
  doi = {10.7208/chicago/9780226804163.003.0014},
  url = {https://doi.org/10.7208/chicago/9780226804163.003.0014}
}

@article{MozesShah,
  author = {Mozes, Shahar and Shah, Nimish},
  title = {On the space of ergodic invariant measures of unipotent flows},
  journal = {Ergodic Theory Dynam. Systems},
  volume = {15},
  number = {1},
  year = {1995},
  pages = {149--159},
  doi = {10.1017/S0143385700008282},
  url = {https://doi.org/10.1017/S0143385700008282}
}

@incollection{ELClay,
  author = {Einsiedler, Manfred and Lindenstrauss, Elon},
  title = {Diagonal actions on locally homogeneous spaces},
  booktitle = {Homogeneous Flows, Moduli Spaces and Arithmetic},
  series = {Clay Mathematics Proceedings},
  volume = {10},
  pages = {155--241},
  publisher = {American Mathematical Society},
  year = {2010},
  url = {https://www.claymath.org/library/proceedings/cmip010c.pdf}
}

@article{ELRankOne,
  author = {Einsiedler, Manfred and Lindenstrauss, Elon},
  title = {On measures invariant under diagonalizable actions---the rank-one case and the general low-entropy method},
  journal = {J. Mod. Dyn.},
  volume = {2},
  number = {1},
  pages = {83--128},
  year = {2008},
  doi = {10.3934/jmd.2008.2.83},
  url = {https://doi.org/10.3934/jmd.2008.2.83}
}

@article{LedrappierYoungII,
  author = {Ledrappier, Fran{\c{c}}ois and Young, Lai-Sang},
  title = {The metric entropy of diffeomorphisms. {P}art {II}: relations between entropy, exponents and dimension},
  journal = {Ann. of Math. (2)},
  volume = {122},
  number = {3},
  pages = {540--574},
  year = {1985},
  doi = {10.2307/1971329},
  url = {https://annals.math.princeton.edu/1985/122-3/p03}
}

@article{LedrappierWalters,
  author = {Ledrappier, Fran{\c{c}}ois and Walters, Peter},
  title = {A relativised variational principle for continuous transformations},
  journal = {J. London Math. Soc. (2)},
  volume = {16},
  number = {3},
  pages = {568--576},
  year = {1977},
  doi = {10.1112/jlms/s2-16.3.568},
  url = {https://doi.org/10.1112/jlms/s2-16.3.568}
}

@article{KaimanovichEntropy,
  author = {Kaimanovich, Vadim A.},
  title = {Invariant measures of the geodesic flow and measures at infinity on negatively curved manifolds},
  journal = {Ann. Inst. H. Poincar{\'e} Phys. Th{\'e}or.},
  volume = {53},
  number = {4},
  pages = {361--393},
  year = {1990},
  url = {https://www.numdam.org/item/AIHPA_1990__53_4_361_0/}
}

@incollection{BrinKatokEntropy,
  author = {Brin, Michael and Katok, Anatole},
  title = {On local entropy},
  booktitle = {Geometric Dynamics},
  series = {Lecture Notes in Mathematics},
  volume = {1007},
  pages = {30--38},
  publisher = {Springer},
  year = {1983},
  editor = {Palis, J.}
}

@article{HochmanShmerkin,
  author = {Hochman, Michael and Shmerkin, Pablo},
  title = {Local entropy averages and projections of fractal measures},
  journal = {Ann. of Math. (2)},
  volume = {175},
  number = {3},
  pages = {1001--1059},
  year = {2012},
  doi = {10.4007/annals.2012.175.3.1},
  url = {https://annals.math.princeton.edu/2012/175-3/p01}
}

@article{BQStationary,
  author = {Benoist, Yves and Quint, Jean-Fran{\c{c}}ois},
  title = {Mesures stationnaires et ferm{\'e}s invariants des espaces homog{\`e}nes},
  journal = {Ann. of Math. (2)},
  volume = {174},
  number = {2},
  pages = {1111--1162},
  year = {2011}
}

@article{EKLLittlewood,
  author = {Einsiedler, Manfred and Katok, Anatole and Lindenstrauss, Elon},
  title = {Invariant measures and the set of exceptions to {Littlewood}'s conjecture},
  journal = {Ann. of Math. (2)},
  volume = {164},
  number = {2},
  pages = {513--560},
  year = {2006}
}

@article{EskinMirzakhani,
  author = {Eskin, Alex and Mirzakhani, Maryam},
  title = {Invariant and stationary measures for the {$\mathrm{SL}(2,\mathbb R)$} action on moduli space},
  journal = {Publ. Math. Inst. Hautes {\'E}tudes Sci.},
  volume = {127},
  pages = {95--324},
  year = {2018}
}

@unpublished{ELZariskiWalks,
  author = {Eskin, Alex and Lindenstrauss, Elon},
  title = {Zariski dense random walks on homogeneous spaces},
  note = {Preprint}
}

@unpublished{ELLocalWalks,
  author = {Eskin, Alex and Lindenstrauss, Elon},
  title = {Random walks on locally homogeneous spaces},
  note = {Preprint}
}

@incollection{ShahUnipotentMeasures,
  author = {Shah, Nimish A.},
  title = {Invariant Measures and Orbit Closures on Homogeneous Spaces
           for Actions of Subgroups Generated by Unipotent Elements},
  booktitle = {Lie Groups and Ergodic Theory ({Mumbai}, 1996)},
  series = {Tata Inst. Fund. Res. Stud. Math.},
  volume = {14},
  publisher = {Tata Institute of Fundamental Research},
  address = {Bombay},
  year = {1998},
  pages = {229--271}
}

@article{KullbackLeibler,
  author = {Kullback, S. and Leibler, R. A.},
  title = {On information and sufficiency},
  journal = {Ann. Math. Statist.},
  volume = {22},
  number = {1},
  pages = {79--86},
  year = {1951},
  doi = {10.1214/aoms/1177729694},
  url = {https://doi.org/10.1214/aoms/1177729694}
}

@article{MOIsolation,
  author = {Mohammadi, Amir and Oh, Hee},
  title = {Isolations of geodesic planes in the frame bundle of a hyperbolic {$3$}-manifold},
  journal = {Compos. Math.},
  volume = {159},
  number = {3},
  year = {2023},
  pages = {488--529},
  doi = {10.1112/S0010437X22007928},
  url = {https://doi.org/10.1112/S0010437X22007928}
}

@article{MO,
  author        = {Mohammadi, Amir and Oh, Hee},
  title         = {Ergodicity of unipotent flows and {Kleinian} groups},
  journal       = {J. Amer. Math. Soc.},
  volume        = {28},
  number        = {2},
  year          = {2015},
  pages         = {531--577},
  doi           = {10.1090/S0894-0347-2014-00811-0},
  url           = {https://doi.org/10.1090/S0894-0347-2014-00811-0}
}

@article{Dufloux,
  author        = {Dufloux, Laurent},
  title         = {Projections of {Patterson--Sullivan} measures and the dichotomy of {Mohammadi--Oh}},
  journal       = {Israel J. Math.},
  volume        = {223},
  number        = {1},
  year          = {2018},
  pages         = {399--421},
  doi           = {10.1007/s11856-017-1622-8},
  url           = {https://doi.org/10.1007/s11856-017-1622-8}
}

@article{BHW,
  author = {Bergeron, Nicolas and Haglund, Fr{\'e}d{\'e}ric and Wise, Daniel T.},
  title = {Hyperplane sections in arithmetic hyperbolic manifolds},
  journal = {J. Lond. Math. Soc. (2)},
  volume = {83},
  number = {2},
  year = {2011},
  pages = {431--448},
  doi = {10.1112/jlms/jdq082}
}

@article{KRS,
  author = {Kolpakov, Alexander and Reid, Alan W. and Slavich, Leone},
  title = {Embedding arithmetic hyperbolic manifolds},
  journal = {Math. Res. Lett.},
  volume = {25},
  number = {4},
  year = {2018},
  pages = {1305--1328}
}

@article{Bur90,
  author = {Burger, Marc},
  title = {Horocycle flow on geometrically finite surfaces},
  journal = {Duke Math. J.},
  volume = {61},
  number = {3},
  year = {1990},
  pages = {779--803},
  doi = {10.1215/S0012-7094-90-06129-0},
  url = {https://doi.org/10.1215/S0012-7094-90-06129-0}
}

@article{Sch16,
  author = {Schapira, Barbara},
  title = {A short proof of unique ergodicity of horospherical foliations on infinite volume hyperbolic manifolds},
  journal = {Confluentes Math.},
  volume = {8},
  number = {1},
  year = {2016},
  pages = {165--174},
  doi = {10.5802/cml.29},
  url = {https://doi.org/10.5802/cml.29}
}

@article{MO16,
  author = {Mohammadi, Amir and Oh, Hee},
  title = {Classification of joinings for {Kleinian} groups},
  journal = {Duke Math. J.},
  volume = {165},
  year = {2016},
  pages = {2155--2223},
  doi = {10.1215/00127094-3476807},
  url = {https://doi.org/10.1215/00127094-3476807}
}

@incollection{BL98,
  author = {Babillot, Martine and Ledrappier, Fran{\c{c}}ois},
  title = {Geodesic paths and horocycle flow on abelian covers},
  booktitle = {Lie Groups and Ergodic Theory ({Mumbai}, 1996)},
  series = {Tata Inst. Fund. Res. Stud. Math.},
  volume = {14},
  publisher = {Tata Institute of Fundamental Research},
  address = {Bombay},
  year = {1998},
  pages = {1--32}
}

@article{Sar04,
  author = {Sarig, Omri},
  title = {Invariant {Radon} measures for horocycle flows on {Abelian} covers},
  journal = {Invent. Math.},
  volume = {157},
  year = {2004},
  pages = {519--551},
  doi = {10.1007/s00222-004-0357-4},
  url = {https://doi.org/10.1007/s00222-004-0357-4}
}

@article{LS07,
  author = {Ledrappier, Fran{\c{c}}ois and Sarig, Omri},
  title = {Invariant measures for the horocycle flow on periodic hyperbolic surfaces},
  journal = {Israel J. Math.},
  volume = {160},
  year = {2007},
  pages = {281--315},
  doi = {10.1007/s11856-007-0064-0},
  url = {https://doi.org/10.1007/s11856-007-0064-0}
}

@article{Led08,
  author = {Ledrappier, Fran{\c{c}}ois},
  title = {Invariant measures for the stable foliation on negatively curved periodic manifolds},
  journal = {Ann. Inst. Fourier (Grenoble)},
  volume = {58},
  number = {1},
  year = {2008},
  pages = {85--105},
  doi = {10.5802/aif.2345},
  url = {https://doi.org/10.5802/aif.2345}
}

@article{OP19,
  author = {Oh, Hee and Pan, Wenyu},
  title = {Local mixing and invariant measures for horospherical subgroups on abelian covers},
  journal = {Int. Math. Res. Not. IMRN},
  volume = {2019},
  number = {19},
  year = {2019},
  pages = {6036--6088},
  doi = {10.1093/imrn/rnx292},
  url = {https://doi.org/10.1093/imrn/rnx292}
}

@article{Sar10,
  author = {Sarig, Omri},
  title = {The horocycle flow and the {Laplacian} on hyperbolic surfaces of infinite genus},
  journal = {Geom. Funct. Anal.},
  volume = {19},
  number = {6},
  year = {2010},
  pages = {1757--1812},
  doi = {10.1007/s00039-010-0048-9},
  url = {https://doi.org/10.1007/s00039-010-0048-9}
}

@article{LL22,
  author = {Landesberg, Or and Lindenstrauss, Elon},
  title = {On {Radon} measures invariant under horospherical flows on geometrically infinite quotients},
  journal = {Int. Math. Res. Not. IMRN},
  volume = {2022},
  number = {15},
  year = {2022},
  pages = {11602--11641},
  doi = {10.1093/imrn/rnab024},
  url = {https://doi.org/10.1093/imrn/rnab024}
}

@article{Lan21,
  author = {Landesberg, Or},
  title = {Horospherically invariant measures and finitely generated {Kleinian} groups},
  journal = {J. Mod. Dyn.},
  volume = {17},
  year = {2021},
  pages = {337--352},
  doi = {10.3934/jmd.2021012},
  url = {https://doi.org/10.3934/jmd.2021012}
}

@misc{CK26,
  author = {Choi, Inhyeok and Kim, Dongryul M.},
  title = {Classification of horospherical invariant measures in higher rank: {The Full Story}},
  year = {2026},
  eprint = {2601.22668},
  archivePrefix = {arXiv},
  primaryClass = {math.DS},
  doi = {10.48550/arXiv.2601.22668},
  url = {https://arxiv.org/abs/2601.22668},
  note = {Preprint, \href{https://arxiv.org/abs/2601.22668}{arXiv:2601.22668}}
}

@incollection{OhICM26,
  author = {Oh, Hee},
  title = {Dynamics and rigidity through the lens of circles},
  booktitle = {Proceedings of the International Congress of Mathematicians 2026, Volume 2: Plenary Lectures},
  editor = {Friedlander, Susan and Tschinkel, Yuri},
  publisher = {Society for Industrial and Applied Mathematics},
  address = {Philadelphia, PA},
  year = {2026},
  pages = {240--271},
  doi = {10.1137/25M1810726},
  url = {https://doi.org/10.1137/25M1810726}
}

\end{document}